\documentclass{amsart}
\pdfoutput=1
\usepackage{adjustbox}
\usepackage{amssymb}
\usepackage{booktabs}
\usepackage{enumitem}
\usepackage{float} 
\usepackage[T1]{fontenc}
\usepackage[margin=1.33in]{geometry}
\usepackage{graphicx}
\usepackage[utf8]{inputenc}
\usepackage{lmodern}
\usepackage{mathtools}
\usepackage{microtype}
\usepackage{setspace}
\usepackage{subcaption}
\usepackage[usenames, dvipsnames]{xcolor}
\usepackage{tikz}
\usepackage{tikz-cd}
\usepackage{xspace}
\usepackage[backref=page, bookmarks=false]{hyperref}
\usepackage[capitalize, noabbrev]{cleveref}
\usepackage{hypcap}

\numberwithin{equation}{section}
\newcommand{\term}{\emph} 

\hypersetup{
 colorlinks,
 linkcolor={red!50!black},
 citecolor={green!50!black},
 urlcolor={blue!80!black}
}

\newcommand{\colim}{\textrm{colim}}
\newcommand{\mb}{\mathbb}

\newcommand{\Z}{\mathbb{Z}}
\newcommand{\R}{\mathbb{R}}

\newcommand{\C}{\mathbb{C}}
\renewcommand{\H}{\mathbb{H}}

\newcommand{\T}{\mathrm{U}(1)}

\newcommand{\ko}{\mathit{ko}}
\newcommand{\ku}{\mathit{ku}}
\newcommand{\KO}{\mathit{KO}}
\newcommand{\KU}{\mathit{KU}}

\newcommand{\SSp}{\mathrm{Sp}}

\newcommand{\GL}{\mathrm{GL}}

\newcommand{\cat}{\mathsf}

\newcommand{\Bord}{\cat{Bord}}

\newcommand{\RP}{\mathbb{RP}}
\newcommand{\CP}{\mathbb{CP}}
\newcommand{\HP}{\mathbb{HP}}

\newcommand{\Det}{\mathrm{Det}}
\renewcommand{\dim}{\mathrm{dim}}
\newcommand{\U}{\mathrm U}

\newcommand{\Sigmainfty}{\Sigma^{\infty}_+}

\newsavebox{\pullback}
\sbox\pullback{%
\begin{tikzpicture}%
\draw (0,0) -- (1ex,0ex);%
\draw (1ex,0ex) -- (1ex,1ex);%
\end{tikzpicture}}

\DeclareMathOperator{\Hom}{Hom}

\newcommand{\NewThomSpectrum}[1]{\expandafter\newcommand\csname M#1\endcsname{\mathit{M#1}}}
\newcommand{\NewMTSpectrum}[1]{\expandafter\newcommand\csname MT#1\endcsname{MT\mathrm{#1}}}
\newcommand{\BothThomSpectra}[1]{\NewThomSpectrum{#1}\NewMTSpectrum{#1}}
\BothThomSpectra{O}
\BothThomSpectra{SO}
\BothThomSpectra{Spin}
\BothThomSpectra{Pin}
\BothThomSpectra{U}
\BothThomSpectra{SU}
\BothThomSpectra{String}

\DeclareTextFontCommand{\df}{\bf}

\newcommand{\pt}{\mathrm{pt}}
\newcommand{\Snb}{S_{\mathit{nb}}}

\newcommand{\MTSpinc}{MT\mathrm{Spin}^c}

\newcommand{\String}{\mathrm{String}}

\newcommand{\Pin}{\mathrm{Pin}}
\newcommand{\Spin}{\mathrm{Spin}}

\newcommand{\SO}{\mathrm{SO}}
\renewcommand{\O}{\mathrm O}
\newcommand{\SU}{\mathrm{SU}}
\newcommand{\bl}{\text{--}}

\newcommand{\Sp}{\cat{Sp}}

\renewcommand{\ker}{\mathrm{ker}}

\newcommand{\Spinc}{\relax\ifmmode{\Spin^c}\else Spin\textsuperscript{$c$}\xspace\fi}
\newcommand{\spinc}{spin\textsuperscript{$c$}\xspace}
\newcommand{\spinh}{spin\textsuperscript{$h$}\xspace}
\newcommand{\Pinc}{\relax\ifmmode{\Pin^c}\else Pin\textsuperscript{$c$}\xspace\fi}
\newcommand{\pinc}{pin\textsuperscript{$c$}\xspace}
\newcommand{\Pinp}{\relax\ifmmode{\Pin^+}\else Pin\textsuperscript{$+$}\xspace\fi}
\newcommand{\pinp}{pin\textsuperscript{$+$}\xspace}
\newcommand{\Pinm}{\relax\ifmmode{\Pin^-}\else Pin\textsuperscript{$-$}\xspace\fi}
\newcommand{\pinm}{pin\textsuperscript{$-$}\xspace}

\newcommand{\sm}{\mathrm{sm}}

\newcommand{\Cl}{\mathit{C\ell}}

\usetikzlibrary{arrows}

\def\[#1\]{%
  \begin{equation}\begin{gathered}#1\end{gathered}\end{equation}%
}

\newcommand{\shortexact}[6]{
\begin{tikzcd}[ampersand replacement=\&]
        0 \& {#3}
        \&  {#4}
        \& {#5}
        \& {0 #6}
        \arrow[from=1-1, to=1-2]
        \arrow["#1", from=1-2, to=1-3]
        \arrow["#2", from=1-3, to=1-4]
        \arrow[from=1-4, to=1-5]
\end{tikzcd}
}

\DeclarePairedDelimiter\abs{\lvert}{\rvert}
\DeclarePairedDelimiter\set{\{}{\}}
\DeclarePairedDelimiter\paren{(}{)}
\DeclarePairedDelimiter\ang{\langle}{\rangle}

\makeatletter
	\let\oldparen\paren
	\def\paren{\@ifstar{\oldparen}{\oldparen*}}
\makeatother

\newtheorem*{thm*}{Theorem}
\newtheorem{lemma}[equation]{Lemma}
\newtheorem{corollary}[equation]{Corollary}
\newtheorem{proposition}[equation]{Proposition}
\newtheorem{theorem}[equation]{Theorem}

\newtheorem*{fiberthm}{\cref{the_cofiber_sequence}}

\newtheorem*{LEScor1}{\cref{bordism_LES_cor}}
\newtheorem*{LEScor2}{\cref{IZ_LES_cor}}

\theoremstyle{definition}
\newtheorem{example}[equation]{Example}
\newtheorem{definition}[equation]{Definition}

\theoremstyle{remark}
\newtheorem{remark}[equation]{Remark}

\crefname{theorem}{Theorem}{Theorems}
\crefname{lemma}{Lemma}{Lemmas}
\crefname{corollary}{Corollary}{Corollaries}
\crefname{propositiom}{Proposition}{Propositions}
\crefname{example}{Example}{Examples}
\crefname{definition}{Definition}{Definitions}
\crefname{claim}{Claim}{Claims}
\crefname{remark}{Remark}{Remarks}

\title{The Smith Fiber Sequence and Invertible Field Theories}
\date{\today}

\author[Debray]{Arun Debray}
\address{Department of Mathematics, The University of Kentucky, 719 Patterson Office Tower,
Lexington, KY 40506, USA}
\email{\href{mailto:a.debray@uky.edu}{a.debray@uky.edu}}

\author[Devalapurkar]{Sanath K. Devalapurkar}
\address{Department of Mathematics, The University of Chicago,
Eckhart Hall,
5734 S University Ave,
Chicago IL, 60637, USA}
\email{\href{mailto:sanathd@uchicago.edu}{sanathd@uchicago.edu}}

\author[Krulewski]{Cameron Krulewski}
\address{Department of Mathematics and Statistics,
Dalhousie University,
6316 Coburg Road,
PO Box 15000,
Halifax, NS, B3H 4R2, Canada}
\email{\href{mailto:ckrulewski@dal.ca}{ckrulewski@dal.ca}}

\author[Liu]{Yu Leon Liu}
\address{Harvard University Department of Mathematics,
1 Oxford Street,
Cambridge, MA 02138, USA}
\email{\href{mailto:yuleonliu@math.harvard.edu}{yuleonliu@math.harvard.edu}}

\author[Pacheco-Tallaj]{Natalia Pacheco-Tallaj}
\address{Massachusetts Institute of Technology,
Department of Mathematics,
Simons Building (Building 2)
77 Massachusetts Avenue,
Cambridge, MA 02139, USA}
\email{\href{mailto:nataliap@mit.edu}{nataliap@mit.edu}}

\author[Thorngren]{Ryan Thorngren}
\address{Mani L. Bhaumik Institute for Theoretical Physics, Department of Physics and Astronomy,
University of California, Los Angeles, CA 90095, USA}
\email{\href{mailto:ryan.thorngren@physics.ucla.edu}{ryan.thorngren@physics.ucla.edu}}

\begin{document}
\begin{abstract}
Smith homomorphisms are maps between bordism groups that change both the dimension and the tangential structure. We give a completely general account of Smith homomorphisms, unifying the many examples in the literature. We provide three definitions of Smith homomorphisms, including as maps of Thom spectra, and show they are equivalent. Using this, we identify the cofiber of the spectrum-level Smith map and extend the Smith homomorphism to a long exact sequence of bordism groups, which is a powerful computation tool. We discuss several examples of this long exact sequence, relating them to known constructions such as Wood's and Wall's sequences.
Furthermore, taking Anderson duals yields a long exact sequence of invertible field theories, which has a rich physical interpretation.
We developed the theory in this paper with applications in mind to symmetry breaking in quantum field theory, which we study in~\cite{PhysSmith}.
\end{abstract}
\maketitle

\tableofcontents

\section{Introduction}

Let $M$ be a closed, smooth $n$-manifold together with a real line bundle $\pi\colon L\to M$. For any section $s\colon M\to L$ of $\pi$ transverse to the zero section, standard theorems in differential topology imply $N_s\coloneqq s^{-1}(0)$ is a smooth, $(n-1)$-dimensional submanifold of $M$. We would like to make $N_s$ into an invariant of $M$ and $L$, but its diffeomorphism type depends on $s$: consider the trivial line bundle over $S^1$ with the constant section valued in $1$ versus any section intersecting the zero section. However, all choices of $N_s$ are \term{bordant}: given two sections $s_1,s_2\colon M\to L$, there is a compact $n$-manifold $Y$ whose boundary is diffeomorphic to $N_{s_1}\amalg N_{s_2}$. Thus the image of $N_s$ in the \term{bordism group} $\Omega_{n-1}^\O$, the set of bordism equivalence classes with group structure given by disjoint union, is a well-defined invariant of $M$ and $L$.

More is true: one can refine the bordism equivalence relation to extend the line bundle $L|_{N_{s_1}}\amalg L|_{N_{s_2}}$ across $Y$, obtaining an invariant valued in the larger group $\Omega_{n-1}^\O(B\O(1))$ of bordism classes of the data of a closed $(n-1)$-manifold and a real line bundle. The value of this invariant also only depends on the bordism class of $M$ and $L$ and is additive in disjoint unions, re-expressing our invariant as a homomorphism of abelian groups
\begin{equation}
\label{OG_Smith}
    \sm_\sigma\colon \Omega_n^\O(B\O(1)) \longrightarrow \Omega_{n-1}^\O(B\O(1)).
\end{equation}
This map was first studied by Conner-Floyd~\cite[Theorem 26.1]{CF64}, who called it the \term{Smith homomorphism} after P. A. Smith. Subsequently, many authors studied similarly-defined maps between other bordism groups, focusing on two methods of generalization:\footnote{Kirby-Taylor~\cite[Theorem 6.11, Remark 6.15]{KT90} generalize the Smith homomorphism in a different direction in the setting of \term{characteristic bordism}; that generalization is out of scope of this paper.}
\begin{enumerate}
    \item Generalize from real line bundles to real or complex vector bundles of other ranks.
    \item Keep track of other topological data on $M$, and how it is affected by passing to $N$.
\end{enumerate}
For example, suppose we give $M$ an orientation structure.  The submanifold $N_s$ does not inherit an induced orientation, and can be unorientable. 
However, since $TM|_{N_s} = TN_{s} \oplus L|_{N_s}$ has an orientation, 
it follows that $L|_{N_s}$ is the orientation line bundle of $N$. It follows that $L|_{N_s}$ gives no additional structure at all, and  this variant of the Smith homomorphism factors through $\Omega_{n-1}^\O\subset\Omega_{n-1}^\O(B\O(1))$. That is, we have a map
\begin{equation}
    \sm_\sigma\colon \Omega_n^\SO(B\O(1))\longrightarrow \Omega_{n-1}^\O.
\end{equation}
To the best of our knowledge, this was first written down by Komiya~\cite[\S 5]{Kom72}.
Other examples in the literature show the same phenomenon: if one places a tangential structure on $M$ in the sense of Lashof~\cite{Las63}, the Smith homomorphism lands in a bordism group whose degree and tangential structure are in general different from those of the domain.

Despite the variety of known examples, the general theory of the Smith homomorphism does not appear in the literature. The first objective of this paper is to tell the general story.

Our other major objective is to apply the Smith homomorphism to quantum physics.
Our inspiration for this paper came from \cite{HKT19}, where they study the physical process of defect anomaly matching as modelled by Smith homomorphisms.
Work of Freed-Hopkins-Teleman~\cite{FHT10}, Freed-Hopkins~\cite{FH16}, and Grady~\cite{Gra23} classifies various kinds of invertible field theories (IFTs) in terms of Anderson duals\footnote{See \S\ref{sss:IFT_bord} and~\cite[\S 5.3]{FH16} for more on Anderson duality and how we use it.} 
to bordism homology theories.\footnote{There are additional classification theorems for invertible field theories due to Yonekura~\cite{Yon19}, Rovi-Schoenbauer~\cite{RS22}, and Kreck-Stolz-Teichner (unpublished; see~\cite{StolzTalk}). All of these are closely related to~\cite{FHT10, FH16} and to each other, but for the purposes of our paper, we need the homotopical approach presented in ~\cite{FH16}.} Therefore, the Anderson dual of the Smith homomorphism is a map of invertible field theories.
As invertible field theories can be understood as anomalies of quantum field theories (see \cref{subsec:anomalies-and-IFT}), the Anderson-dualized Smith homomorphism provides an anomaly-matching formula expressing the anomalies of certain QFTs in terms of anomalies of lower-dimensional \term{defect theories}~\cite{HKT19,COSY19}.

However, in \cite{HKT19}, they noted in section 4.4 that they were missing a mathematical way to compute their homomorphisms of interest. The fiber sequence of spectra we studied to answer this question led to more interesting physics: in forming a fiber sequence, hence a long exact sequence, we found a tool to easily compute Smith homomorphisms. Furthermore,  two new maps in the long exact sequence also have interesting physical interpretations in the context of spontaneous symmetry breaking. In our companion paper~\cite{PhysSmith}, we provide detailed physical interpretations of the entire symmetry breaking long exact sequence (SBLES) and address many physical examples and applications. We provide a summary in \cref{LESIFTs} in this paper.

Next we outline the results of this paper. We make use of standard definitions in bordism theory, which we review in \S\ref{thom_background}. Fix a tangential structure\footnote{In the definition of a tangential structure, $B$ is a space and $\xi\colon B\to B\O$ is a fibration. There is no relation between this space $B$ and the classifying space functor $B$; in this paper, it should be clear from context which one we mean when we say $B$.}
$\xi\colon B\to B\O$ (see \cref{defn_tangstr}), a space $X$, a virtual vector bundle $V\to X$ of rank $r_V$, and a vector bundle $W\to X$ of rank $r_W$.

By an \term{$(X, V)$-twisted $\xi$-structure} on a virtual vector bundle $E\to M$, we mean the data of a map $f\colon M\to X$ and a $\xi$-structure on $E\oplus f^*(V)$ (\cref{VB_twist}); this is a tangential structure in its own right (\cref{shearing_lemma}), which we denote $\xi + (X, V)$. In particular, there is a notion of bordism of $(X, V)$-twisted $\xi$-manifolds, and by \cref{what_is_twisted_bordism}
the corresponding bordism groups are naturally isomorphic to the $\xi$-bordism groups of the Thom spectrum $X^{V-r_V}$ of the rank-zero virtual bundle $V - r_V\to X$:\footnote{This perspective on twisted bordism, and these results, are not new, and we are not sure who originally developed them. We follow the language and perspective of~\cite[\S 10]{DDHM22}, and note this is not the first approach to this material.}
\begin{equation}\label{tw_unshear}
    \Omega_n^{\xi + (X, V)} \xrightarrow{\cong} \Omega_n^\xi(X^{V - r_V}).
\end{equation}
Our first result is to define a Smith homomorphism $\sm_W$ associated to the data of $\xi$, $X$, $V$, and $W$. This Smith homomorphism will have type signature
\begin{equation}
    \sm_W\colon \Omega_n^\xi(X^{V-r_V}) \longrightarrow \Omega_{n-r_W}^\xi(X^{V\oplus W - r_V - r_W}).
\end{equation}
That is, passing through~\eqref{tw_unshear}, the Smith homomorphism passes from the bordism group of $n$-dimensional $(X, V)$-twisted $\xi$-manifolds to the bordism group of $(n-r_W)$-dimensional $(X, V\oplus W)$-twisted $\xi$-manifolds. We actually provide three different definitions of $\sm_W$:
\begin{enumerate}
    \item First, in \cref{Smith_homomorphism_intersection_defn} we define $\sm_W$ as the map sending the bordism class of a manifold $M$ with map $f\colon M\to X$ to the bordism class of the zero locus of a section of $f^*W\to M$ transverse to the zero section.
    \item\label{smith_2_intro} We then define the Smith homomorphism in \cref{Smith_map_vector_bundle} as the map of bordism groups induced by a map of Thom spectra $X^V\to X^{V\oplus W}$, itself induced by the map of total spaces of vector bundles $V\to V\oplus W$ sending $v\mapsto (v, 0)$.
    \item Our third definition, in \cref{smith_homomorphism_euler_definition}, defines $\sm_W$ as the cap product homomorphism with the Euler class of $W$ in (twisted) $\xi$-cobordism, following a construction of Euler classes in twisted generalized cohomology in \S\ref{ss:tw_Euler}.
\end{enumerate}
\begin{thm*}[\cref{two_smith_1,two_smith_2}]
The above three definitions are equivalent.
\end{thm*}

Each definition has its own advantages: the first and third allow for a direct comparison with preexisting special cases in the literature and the second is an essential ingredient for constructing long exact sequences.
Specifically, in \cref{math_Smithfibersequence}, we reprove the following well-known theorem.
\begin{fiberthm}
With $X$, $V$, $W$, and $\xi$ as above, the fiber of the map of spectra $X^V\to X^{V\oplus W}$ in definition~\ref{smith_2_intro} is the map $p\colon S(W)^{p^*V}\to X^V$, where $S(W)$ denotes the unit sphere bundle of $W$ and $p\colon S(W)\to X$ is the bundle projection map.
\end{fiberthm}
This is not a new result; in~\cite[Remark 3.14]{KZ18} it is attributed to James. Our contribution is to relate it to the fully general definition of the Smith homomorphism.

From the fiber sequence of spectra in \cref{math_Smithfibersequence}, we derive two long exact sequences. First, in bordism:
\begin{LEScor1}
Let $X$, $V$, $W$, and $\xi$ be as above. There is a long exact sequence of bordism groups: 
\begin{equation}\label{bordism_LES_eqn}\resizebox{1\hsize}{!}{
        $\dotsb \longrightarrow \Omega_n^\xi(S(W)^{p^*V-r_V}) \overset{p}{\longrightarrow} \Omega_n^\xi(X^{V-r_V}) \xrightarrow{\sm_W} \Omega_{n-r_W}^\xi (X^{V\oplus W - r_V - r_W}) \overset{\delta}{\longrightarrow} \Omega_{n-1}^\xi(S(W)^{p^*V-r_V}){\longrightarrow}\dotsb$}
\end{equation}
\end{LEScor1}
As we discuss in \cref{smith_is_gysin}, this is a generalized Gysin sequence.
As a computational tool, it turns out to be remarkably convenient. 
Different vector bundles $W$ can be combined to calculate bordism groups, often avoiding difficult spectral sequence calculations. For example, we use this idea in \cref{explicit_pin} to address an extension problem; other papers using this or closely related techniques to do computations include~\cite{HS13, DL23, Deb23, DDHM22, DNT24, DYY}.

Next, there are many examples of Smith homomorphisms in the literature, so we devote some time in this paper to explicating the Smith homomorphism for various choices of $\xi$, $X$, $V$, and $W$. One phenomenon that we address in \S\ref{periodicity_and_shearing} is that Smith homomorphisms come in ``families'': if you iterate $\sm_W$ with $V = 0,W,2W,3W,\dotsc$, often the domain and codomain recur with a finite period $p$ because the notions of $(X,kW)$-twisted and $(X, (k+p)W)$-twisted $\xi$-structures are equivalent. In \S\ref{periodicity_and_shearing}, we use the ``fake vector bundle twists'' of~\cite[\S 1]{DY23} to establish the following Smith families.
\begin{itemize}
    \item \Cref{unoriented_periodicity}: when $\xi = 
    \mathrm{id}\colon B\O\to B\O$ (unoriented bordism), the period $p=1$ for all $X$ and $W$: like in~\eqref{OG_Smith}, the tangential structure does not change.
    \item \Cref{periodicity_of_SO,MU_MSpinc}: for $\xi\colon BG\to B\O$ for $G = \SO$, $\Spin^c$, or $\U$, $p = 1$ if $W$ is orientable and $2$ if $W$ is unorientable.
    \item \Cref{periodicity_of_spin}: for $\xi\colon B\Spin\to B\O$, $p$ can be $1$, $2$, or $4$ depending on the first two Stiefel-Whitney classes of $W$.
\end{itemize}
These may be thought of as versions of James periodicity over bases other than $\mathbb S$; see \cref{james_example}. For some other common choices of $\xi$, including $\SU$-structure, string structure, and stable framing, the period is harder to determine, as we discuss in \S\ref{not_sharp}. We would be interested in learning tools for computing such periods, and suspect that these periodicities are related to the image of the $J$-homomorphism; see \cref{highconn}.

We also study examples where we fix $X$, $V$, and $W$, but let $\xi$ vary, recovering known Smith families from the literature.
\begin{enumerate}
    \item In \S\ref{taut_real_fams}, we let $X = \RP^\infty$, $V = k\sigma$, and $W = \sigma$, where $\sigma\to\RP^\infty$ is the tautological line bundle.
    \begin{itemize}
        \item In \cref{Z2_MO}, we apply this to unoriented bordism, where it reproduces Conner-Floyd's original Smith homomorphism~\eqref{OG_Smith}~\cite[Theorem 26.1]{CF64}.
        \item For oriented bordism, see \cref{SO_Z2_exm}, where we get Komiya's~\cite[\S 5]{Kom72} $2$-periodic family of Smith homomorphisms exchanging $\SO\times\Z/2$- and $\O$-bordism.
        \item On \spinc bordism (\cref{spinc_Z2_exm}), this recovers the Smith homomorphisms between $\Spin^c\times\Z/2$ and $\Pin^c$ bordism studied in~\cite{BG87a, HS13}.
        \item On spin bordism (\cref{spin_4periodic}), this Smith family is $4$-periodic, involving bordism for the groups $\Spin\times\Z/2$, $\Pin^-$, $\Spin\times_{\set{\pm 1}}\Z/4$ and $\Pin^+$. This family or subsets of it have been discussed in~\cite{Pet68, ABP69, Gia73a, Kre84, KT90, HS13, KTTW15, TY19, HKT19, WWZ19, BR23}.
        \item On string bordism (\cref{string_Z2}), this family has period $8$ and to our knowledge has not appeared in the literature. It would be interesting to study this family in more detail.
    \end{itemize}
    \item In \S\ref{CP_inf_twist}, we instead consider the tautological complex line bundle over $\CP^\infty$. For some tangential structures ($\O$, $\SO$, $\Spin^c$, $\U$) this family is $1$-periodic (\cref{cpx_triv_exms}).
    \begin{itemize}
        \item For spin bordism (\cref{spinc_spin_Smith}), this family is $2$-periodic, exchanging $\Spin\times\U(1)$ and \spinc bordism, recovering work of Kirby-Taylor~\cite[Corollary 6.12, Remark 6.14]{KT90}.
        \item In \cref{spinzn}, we pull $L$ back along $B\Z/n\to B\U(1)$, obtaining a $2$-periodic family exchanging bordism for the groups $\Spin\times\Z/n$ and $\Spin\times_{\set{\pm 1}}\Z/2n$ also studied in~\cite[Appendix E]{DDHM22}.
    \end{itemize}
\end{enumerate}
These are not the only examples we consider---see \S\ref{examplesofSmith} for more.

In \cref{LESIFTs}, we leverage our theory toward physical applications.
Let $I_\Z\Omega_\xi^*(\text{--})$ denote the generalized cohomology theory that is Anderson dual to $\xi$-bordism.  
As mentioned above, theorems of Freed-Hopkins~\cite{FH16} and Grady~\cite{Gra23} imply that $I_\Z\Omega_\xi^*(X)$ classifies deformation classes of reflection-positive, fully extended invertible field theories, as defined in~\cite[\S 5]{FH16}, on manifolds equipped with $\xi$-structures and maps to the space $X$. In particular, when $X = BG$, homotopy classes of maps to $X$ classify isomorphism classes of principal $G$-bundles.
Using the fiber sequence, we show the following.
\begin{LEScor2}
    Let $X$, $V$, $W$, and $\xi$ be as above. There is a long exact sequence of groups of deformation classes of reflection-positive, fully extended invertible field theories:
\begin{equation}\label{SBLES_eqn}\resizebox{1\hsize}{!}
        {$\dotsb \longrightarrow 
        I_\Z\Omega^{n-1}_\xi(S(W)^{p^*V-r_V}) 
        \xrightarrow{\mathrm{Ind}_W}
        I_\Z\Omega_\xi^{n-r_W}(X^{V\oplus W - r_V - r_W}) \xrightarrow{\mathrm{Def}_W} I_\Z\Omega^n_\xi(X^{V-r_V}) 
        \xrightarrow{\mathrm{Res}_W}
        \Omega^n_\xi(S(W)^{p^*V-r_V}) \longrightarrow\dotsb$}
\end{equation}
\end{LEScor2}

This is our mathematical model for the \term{symmetry-breaking long exact sequence} of~\cite{PhysSmith}, and we have labelled the maps as in that paper. The map dual to the Smith map $\sm_W$ is the defect anomaly matching map $\mathrm{Def}_W$ studied in \cite{HKT19} and \cite{COSY19}. Our framework allows for the study of more physical examples, as well as provides a way of explicitly computing each map and thus extracting more physical information. 
We study this in detail in \cite{PhysSmith}.

Finally, we have two appendices. In Appendix~\ref{appendix_bordism_LES}, we explicate a Smith long exact sequence from \cref{another_Wall_exm}, which interchanges \pinm and \pinp bordism, with the third term in the long exact sequence identified with a certain twisted spin bordism of $\RP^1$. In Appendix~\ref{s:eu_counter}, we provide an example computation of a Smith homomorphism that goes beyond cohomological approximations.

An interesting direction for future work is to investigate what happens in the absence of unitarity. The mathematical backbone of our work generalizes nicely to the nonunitary case: Freed-Hopkins-Teleman~\cite{FHT10} classify invertible topological field theories in the absence of a reflection positivity structure using unstable Madsen-Tillmann spectra, and the Smith long exact sequence generalizes to this case (see, e.g., \cref{GMTW_exm}). Anomalies of nonunitary theories are not so well-studied, but some examples appear in~\cite{CL21, HTY22, HSV25}, and the fact that the Smith long exact sequence generalizes suggests our methods do too.

From there one could ask: the appearance of Madsen-Tillmann spectra in the classification of invertible TFTs is due to theorems of Galatius-Madsen-Tillmann-Weiss~\cite{GMTW09}, Nguyen~\cite{Ngu17}, and Schommer-Pries~\cite{SP17} establishing Madsen-Tillmann spectra as classifying spectra for bordism (higher) categories. Can one lift the Smith homomorphism to a morphism of bordism categories? This is a question in pure mathematics whose affirmative answer would suggest a generalization of our methods to noninvertible TFTs, and therefore to the symmetry breaking of noninvertible symmetries of field theories, as studied in, e.g.,~\cite{LTLSB21, ABCGR23, DY23a, CHZ23, DAC23, CBNM25, KW25}.

\subsection*{Acknowledgements}
We would like to thank
Adrian Clough,
Dan Freed,
Mike Hopkins,
Theo Johnson-Freyd,
Yigal Kamel,
Justin Kulp,
Miguel Montero,
David Reutter,
Luuk Stehouwer,
Weicheng Ye,
Matthew Yu, and
the anomymous referees
for helpful discussions which improved our paper.

Part of this project was completed while AD, CK, and YLL visited the Perimeter Institute for Theoretical Physics for the 2022 workshop on Global Categorical Symmetries; research at Perimeter is supported by the Government of Canada through Industry Canada and by the Province of Ontario through the Ministry of Research \& Innovation. The workshop on Global Categorical Symmetries was supported by the Simons Collaboration on Global Categorical Symmetries. We thank both Perimeter and the Simons Collaboration for Global Categorical Symmetries for their hospitality.
CK and NPT are supported by NSF DGE-2141064 and SKD is supported by NSF DGE-2140743.

The authors declare that we have no competing interests.
We certify that the data supporting the conclusions of this article are available within the paper.

\section{Bordism and Thom spectra}\label{thom_background}

In this section, we review virtual bundles, tangential structures, and their 
Thom spectra. We also review the Pontrjagin-Thom theorem, which relates the homotopy groups of Thom spectra to bordism groups. 

\subsection{Virtual vector bundles and tangential structures}
Everything in this subsection is well-worn mathematics; see~\cite{Fre19, FH16, DY23} and the references therein for additional references for this material.

\begin{definition}
A \term{virtual vector bundle} $V\to X$ is the data of two vector bundles $V_1,V_2\to M$, which we think of as ``$V_1 - V_2$.''

An \term{isomorphism of virtual vector bundles} $f$ between $V = (V_1, V_2)$ and $W = (W_1, W_2)$ over a common base space $X$ is the data of vector bundles $E_1,E_2\to X$ and isomorphisms $f_1\colon V_1\oplus E_1\xrightarrow{\cong} W_1\oplus E_2$ and $f_2\colon V_2\oplus E_1\xrightarrow{\cong} W_2\oplus E_2$.
\end{definition}
The idea behind this definition of isomorphism is that we would like the following pairs of virtual vector bundles to be isomorphic.
\begin{enumerate}
    \item $(V_1, V_2)$ and $(W_1, W_2)$ when $V_1\cong W_1$ and $V_2\cong W_2$.
    \item $(V_1, V_2)$ and $(V_1\oplus E, V_2\oplus E)$: adding and subtracting $E$ should not change the isomorphism type of the vector bundle.
\end{enumerate}
A vector bundle $V$ defines a virtual vector bundle as the pair $(V, 0)$. In the future we will make this assignment implicitly.

Let $B\O$ denote the classifying space of the infinite-dimensional orthogonal group $\O\coloneqq\varinjlim_n\O_n$.
\begin{lemma}
The space $\Z\times B\O$ classifies virtual vector bundles: for a space $X$ with the homotopy type of a CW complex, the set $[X, \Z\times B\O]$ is naturally in bijection with the set of isomorphism classes of virtual vector bundles on $X$.
\end{lemma}
Projection $\Z\times B\O\to\Z$ onto the first component defines a numerical invariant of virtual vector bundles; this is the \term{rank} $\mathrm{rank}(V) \coloneqq \dim(V_1) - \dim(V_2)$. Thus $B\O$, thought of as $B\O\times\{0\}$, is the classifying space for rank-zero virtual vector bundles.
\begin{definition}
\label{defn_tangstr}
A \term{(stable) tangential structure} is a fibration $\xi\colon B\to B\O$, and given $\xi$, a \term{$\xi$-structure} on a (rank-zero virtual) vector bundle $V\to X$ is a lift of its classifying map $f_V\colon X\to B\O$ as in the diagram
\begin{equation}
\begin{tikzcd}
	& {B} \\
	X & {B\O,}
	\arrow["{f_V}"', from=2-1, to=2-2]
	\arrow["{\xi}", from=1-2, to=2-2]
	\arrow["{\widetilde f_V}", dashed, from=2-1, to=1-2]
\end{tikzcd}
\end{equation}
i.e.\ a map $\widetilde f_V\colon X\to B$ such that $f_V \simeq \xi\circ\widetilde f_V$.
\end{definition}
Two $\xi$-structures on $V$ are equivalent if the corresponding maps $\widetilde f_V, \widetilde f{}'_V\colon X\to B$ are homotopic. 

If $M$ is a manifold, a \term{$\xi$-structure} on $M$ means a $\xi$-structure on the virtual vector bundle defined by $TM$. One also sees \term{normal $\xi$-structures} on $M$, which are $\xi$-structures on $-TM$, the virtual vector bundle defined by the pair $(0, TM)$.
\begin{example}
For $\xi\colon B\SO\to B\O$, a $\xi$-structure is equivalent to an orientation. For $\xi\colon
B\Spin\to B\O$, a $\xi$-structure is equivalent to a spin structure.
\end{example}

If $M$ is a manifold with boundary, the outward unit normal vector field defines a trivialization of the normal
bundle to $\partial M\hookrightarrow M$, so $T(\partial M)\oplus\underline\R\cong TM|_{\partial M}$, and therefore
a $\xi$-structure on $M$ induces a $\xi$-structure on $\partial M$. It is therefore possible to define a notion of
bordism of manifolds with $\xi$-structure, as Lashof~\cite{Las63} did; we let $\Omega_n^\xi$ denote the set of
bordism classes of $n$-manifolds with $\xi$-structure, which becomes an abelian group under disjoint union.

Often, one studies groups $G$ with maps $\rho\colon G\to\O$, and lets $\xi\coloneqq B\rho\colon BG\to B\O$. In this
case it is common to denote $\Omega_*^\xi$ as $\Omega_*^G$ (e.g.\ $G = \O$, $\SO$, $\Spin$, $\Pin^\pm$, etc.).
\begin{remark}
\label{functorial_BO}
The category of tangential structures is the slice category $\cat{Top}_{/B\O}$, i.e.\ the objects are spaces with a
map to $B\O$, and the morphisms are maps which commute with the maps to $B\O$. Bordism groups are covariantly
functorial in this category.
\end{remark}

\subsection{Construction of Thom spectra}
First, recall the classical construction of a Thom space:
if $V\to X$ is a vector bundle, choose a Euclidean metric on $V$. Let $D(V)$ be the \term{disk bundle} of vectors in $V$ of norm at most $1$ and $S(V)$ be the \term{sphere bundle} of vectors of norm exactly $1$; write $\mathrm{Th}(X; V) \coloneqq D(V)/S(V)$.

\begin{example}
    Let $\underline{\R}^n\to X$ be a trivial bundle and let $X_+$ be the space $X$ with a disjoint basepoint. Then the Thom space is the $n$-fold suspension $\mathrm{Th}(X; \underline{\R}^n) \simeq \Sigma^n X_+$.
\end{example}

\begin{example}
    Let $X=\RP^n$ and let $V=\sigma$ be the tautological line bundle. Then the Thom space is $\mathrm{Th}(\RP^n;\sigma)\simeq \RP^{n+1}$.
\end{example}

\begin{proposition}\label{prop:X-compact-thom-one-point}
    If $X$ is compact, then $\mathrm{Th}(X;V)$ is the one point compactification of the disk bundle $D(V)$.
\end{proposition}

Let $\cat{Top}$ denote the $\infty$-category of spaces and $\cat{Top}_*$ denote the $\infty$-category of pointed spaces.\footnote{The reader does not need to be comfortable with $\infty$-categories to follow the results in this section. We use them because they simplify some of the language around Thom spectra and are heavily used by~\cite{ABGHR14a, ABGHR14b}, but one can intuitively replace ``$\infty$-category'' with ``category'' without losing the key insights of this section.} Given $X\in\cat{Top}$, there is an $\infty$-category $\pi_{\le\infty}(X)$, the \term{fundamental $\infty$-groupoid} of $X$, whose objects are points of $X$, $1$-morphisms are paths in $X$, $2$-morphisms are homotopies between paths, etc. $\pi_{\le\infty}(X)$ is often used as a categorical stand-in for $X$, in the sense that functors $\pi_{\le\infty}(X)\to\cat C$ for an $\infty$-category $\cat C$ are thought of as $\cat C$-valued local systems on $X$: for every point $x\in X$, we have an object $c(x)$ of $\cat C$, for every path $x\to y$  in $X$ we have an isomorphism $c(x)\xrightarrow{\cong} c(y)$, etc. For example, there is a functor $\kappa \colon \pi_{\le\infty}(B\O(d))\to \cat{Vect}$ describing the fibers of the tautological rank-$d$ vector bundle and the parallel transport with respect to a (noncanonical) connection. Thus, if $V \to X$ is a rank-$d$ real vector bundle (\emph{not} merely a virtual vector bundle), and $f_V\colon X \to B\O(d)$ is the classifying map, the functor $\kappa\circ \pi_{\le\infty}(f_V)$ from $\pi_{\le\infty}(X)$ to $\cat{Vect}$ can be throught of as describing how $V\to X$ is a local system of vector spaces on $X$.

The one-point compactification of a vector space $V$ is a sphere, and this defines a functor $\cat{Vect}\to\cat{Top}_*$; precomposing with $\kappa$, we obtain a functor $\kappa'$ from the fundamental $\infty$-groupoid of $B\O(d)$ to $\cat{Top}_*$. As before, this should be thought of as describing the fibers, paths, etc., in a local system of $d$-dimensional spheres over $B\O(d)$.

We will frequently use diagrams out of $\pi_{\le\infty}(X)$ for a space $X$ below; therefore, given an $\infty$-category $\cat C$, the notation $X\to\cat C$ refers to a functor $\pi_{\le\infty}(X)\to\cat C$.
\begin{proposition}
The Thom space $\mathrm{Th}(X; V)$ is naturally homotopy equivalent to the colimit of the functor
\begin{equation}
    X\overset{f_V}{\longrightarrow} B\O(d) \overset{\kappa'}{\longrightarrow} \cat{Top}_*.
\end{equation}
\end{proposition}
This process is analogous to taking the global sections of a sheaf to define sheaf cohomology, though that is dual to this (it is a limit, rather than a colimit).


We need a similar construction to $\mathrm{Th}(X;V)$ for virtual bundles on $X$. It has the structure of a \term{spectrum}. 
The reader unfamiliar with spectra is encouraged to think of them as
similar to topological spaces, in that one can take homotopy, (co)homology, and generalized (co)homology groups of
them. See Freed-Hopkins~\cite[\S 6.1]{FH16} or Beaudry-Campbell~\cite[\S 2]{BC18} for precise definitions
and~\cite[\S 10.3]{DDHM22} for a longer but still heuristic overview. We write $\Sp$ for the $\infty$-category of spectra.

We follow \cite{ABGHR14a} in the rest of this section.
By a \term{local system of spectra} over a space $X$ we mean a functor $\mathcal L$ from the fundamental $\infty$-groupoid of $X$ to spectra. We will usually denote this as $\mathcal L\colon X\to\cat{Sp}$. The fiber of a local system at a point $p\in X$ is obtained by composing $\mathcal L$ with the functor $\mathrm{pt}\to X$ given by inclusion at $p$; a functor out of $\mathrm{pt}$ is equivalent to a single spectrum, and we call this the fiber of $\mathcal L$ at $p$.
\begin{definition}[{\cite{DL59, ABGHR14a}}]
A \term{stable spherical fibration} is a local system of spectra valued in the full sub-$\infty$-category of spectra with objects $\Sigma^n \mathbb S$, $n\in\Z$.
\end{definition}
Here $\mathbb S$ denotes the \term{sphere spectrum}.

\begin{definition}
\label{dfnSV}
Let $X$ be a space and $V\to X$ be a vector bundle of rank $r$. Let $\mathbb S^V\to X$ denote the associated stable spherical fibration, whose fiber at a point $x\in X$ is the suspension spectrum of the one-point compactification of $V_x$.
\end{definition}
For a vector space $W$, we use $S^W$ to refer to the one-point compactification of $W$, which is noncanonically diffeomorphic to $S^{\mathrm{dim}(W)}$.

Now fix a base space $X$ and a virtual vector bundle $V \to X$, which is equivalently a map $V\colon X \to B\O \times \mb{Z}$. Associated to $X$ and $V$, there is a spectrum called the \term{Thom spectrum} $X^V$ constructed as follows.
There is a functor $J\colon B\O \times \mb{Z} \to \Sp$, generalizing the map $B\O(d) \to \cat{Top}_*$ above. $J$ maps into spectra now, instead of spaces, because for a virtual bundle $V_1 - V_2$, we want to assign the sphere $S^{V_1} \wedge S^{-V_2}$, 
but $S^{-V_2}$ doesn't make sense as a space, since spheres of negative dimension don't exist. However, the sphere spectrum $\mathbb S$ can be desuspended, and the Thom spectrum associated to a virtual bundle is defined as follows.

\begin{definition}
\label{vb_thom_ABGHR}
	Given a virtual bundle $V\colon X \to B\O \times \mathbb{Z}$, the \term{Thom spectrum} $X^V$ is the colimit (in spectra) of the composite $X \xrightarrow{V} BO \times \mathbb{Z} \xrightarrow{J} \Sp$.
\end{definition}

Here's the compatibility between the Thom space and Thom spectrum construction.
\begin{lemma}
\label{thom_from_space}
Let $V\colon X \to B\O(d)$ be a vector bundle and let $\xi\colon X \to B\O(d) \to B\O \times \mathbb{Z}$ be the corresponding virtual bundle. Then the Thom spectrum of $\xi$ is the suspension spectrum of the Thom space of $V$; i.e. $X^\xi \simeq  \Sigma^\infty_+ X^V$.
\end{lemma}
Here by $\Sigma_+^\infty$ we mean first taking the disjoint union with a single point, which we take as the basepoint, then taking the suspension spectrum.
\begin{proof}
    This follows from the fact that $\Sigma^{\infty}_+\colon \cat{Top} \to \Sp$ preserves colimits.
\end{proof}
Using \cref{thom_from_space}, one can directly check that the Thom spectrum of the trivial bundle $\underline\R^n\to X$ is homotopy equivalent to a suspension of the suspension spectrum $\Sigma^n \Sigma^{\infty}_+ X$.

\begin{lemma}[{\cite[Lemma 2.3]{Ati61}}]
\label{boxplus_lem}
Let $V\to X$ and $W\to Y$ be virtual vector bundles. Then the Thom spectrum of $V\boxplus W\to X\times Y$ is homotopy equivalent to $X^V\wedge Y^W$.
\end{lemma}
Here $\boxplus$ is the external direct sum, i.e.\ the direct sum of the pullbacks of $V$ and $W$ across the projection maps $X\times Y\to X$, resp.\ $X\times Y\to Y$.

One can often combine \cref{boxplus_lem} with the observation that Thom spectra of trivial bundles are suspensions to simplify Thom spectra appearing in examples. For example, $X^{V+\underline\R^n}$, often denoted $X^{V+n}$, is homotopy equivalent to $\Sigma^n X^V$. Since we are working with virtual vector bundles, $n$ may be any integer.

Let us discuss a variant for tangential structures.
\begin{definition}
\label{mads_till}
Let $\xi\colon B\to B\O$ be a tangential structure. Then its inverse (as a virtual vector bundle) $-\xi$ is often denoted $\xi^\perp$. Equivalently, $\xi^\perp$ is the composition of $\xi$ with the map $-1\colon B\O\to B\O$, which is the inverse map in the $E_\infty$-structure on $B\O$ induced by direct sum. Therefore $\xi^\perp$ is also a tangential structure; its Thom spectrum $B^{-\xi}$ is called a \term{Madsen-Tillmann spectrum}~\cite{MT01, MW07} and is often denoted $\mathit{MT\xi}$. If $B\to B\O$ is obtained from a family of Lie group homomorphisms $H(n)\to\O(n)$ in the (co)limit $n\to\infty$, $\mathit{MT\xi}$ is often written $\mathit{MTH}$.

Likewise, the Thom spectrum of the pullback of $-V_n\to B\O(n)$ across a map $\xi_n\colon B_n\to B\O(n)$ is denoted $\mathit{MT\xi}_n$; if $B = BH(n)$ for a Lie group $H(n)$, this is often written $\mathit{MTH}(n)$.
\end{definition}

$\mathit{MT\xi}$ has two key properties.
\begin{theorem}[{Pontrjagin~\cite{Pon50, Pon55}, Thom~\cite{ThomThesis}}]
\label{PTthm}
There is a natural isomorphism $\pi_n(\mathit{MT\xi})\xrightarrow{\cong}
	\Omega_n^\xi$.\footnote{It is most common to define Thom spectra and bordism in terms of the stable normal
	bundle, rather than the tangent bundle; the resulting spectra are written $\mathit{M\xi}$. The spectra
	$\mathit{MT\xi}$ and $\mathit{M\xi}$ coincide for the tangential structures $\O$, $\SO$, $\Spin^c$ and $\Spin$,
	but not in general: $\mathit{MTPin}^\pm\simeq\mathit{MPin}^\mp$.  By composing with the map $-1\colon B\O\to
	B\O$, one can pass between normal bordism and tangential bordism and therefore pass between our definition and
	the standard one.}
\end{theorem}
Pontrjagin and Thom considered a few specific choices of $\xi$; Lashof~\cite{Las63} first considered general tangential structures.
\begin{theorem}[Thom isomorphism~\cite{ThomThesis}]
\label{thomiso}
Let $A$ be a commutative ring. Then there is a natural\footnote{Naturality
	here is for maps of tangential structures as in \cref{functorial_BO}; this map typically does not commute with the action of cohomology
	operations.} isomorphism $H^*(B; A_{w_1})\xrightarrow{\cong} H^*(\mathit{MT\xi}; A)$, where $A_{w_1}$ denotes the
	pullback by $\xi$ of the orientation local system on $B\O$.
\end{theorem}
In \cref{thomiso}, the use of twisted cohomology can be avoided by assuming $A = \Z/2$ or by choosing
an orientation of the virtual vector bundle classified by the map $\xi$.


\section{Maps of spectra inducing Smith homomorphisms}

This section is the technical heart of the paper---we provide a general definition of the Smith homomorphism, then lift it to a map $\mathrm{sm}$ of bordism spectra. The map of spectra has been studied, though its identification with the Smith homomorphism is new; using this, we can write down the cofiber of $\mathrm{sm}$ (\cref{the_cofiber_sequence}) and therefore obtain Smith long exact sequences of bordism groups and Anderson-dualized bordism groups (\cref{bordism_LES_cor,IZ_LES_cor}).

\subsection{$(X, V)$-twisted tangential structures}\label{math_section_twisted_tangential_structures}
Twisted tangential structures are an important ingredient in the Smith homomorphism---they determine its domain and codomain. We take this subsubsection to define them and point out why they arise in the Smith homomorphism setting.

Throughout this subsubsection, we fix a topological space $X$, a vector bundle $V\to X$ of rank $r$, and a tangential structure $\xi\colon B\to B\O$.
\begin{definition}
\label{VB_twist}
Let $W\to Y$ be a vector bundle. An \term{$(X, V)$-twisted $\xi$-structure} on $W$ is the data of a map $f\colon Y\to X$ and a $\xi$-structure on $W\oplus f^*(V)$.
\end{definition}
There is a space of $(X, V)$-twisted $\xi$-structures on $W$, and just like for tangential structures, we will think of two such structures as the same if they lie in the same connected component.

Twisted $\xi$-structures provide a convenient way to describe a more complicated tangential structure in terms of a simpler one.
\begin{example}
Recall that a \spinc structure on an oriented vector bundle $W\to Y$ is the data of a complex line bundle $L\to Y$ and an identification $w_2(L) = w_2(W)$. The data of $L$ is equivalent to a map $Y\to B\U(1)$ such that $L$ is the pullback of the tautological complex line bundle $S\to B\U(1)$.
 The identification $w_2(L) = w_2(W)$ is equivalent by the Whitney sum formula to $w_2(W\oplus L) = 0$.

Choosing a spin structure on $W\oplus L$ first provides an orientation of $W\oplus L$, which since $L$ is canonically oriented by its complex structure is equivalent to an orientation of $W$; then it additionally provides an identification $w_2(W\oplus L) = 0$. Therefore the data of a \spinc structure on $W$ is equivalent to the data of $L$ and a spin structure on $W\oplus L$, meaning that a \spinc structure is equivalent to a $(B\U(1), S)$-twisted spin structure.
\end{example}
In a similar way, one can show that if $\sigma\to B\Z/2$ is the tautological real line bundle, \pinm structures are equivalent to $(B\Z/2, \sigma)$-twisted spin structures, \pinp structures are equivalent to $(B\Z/2, 3\sigma)$-twisted spin structures, and \pinc structures are equivalent to $(B\Z/2, \sigma)$-twisted \spinc structures.

It turns out that all of these twisted tangential structures can also be ``untwisted'' into ordinary tangential structures.
\begin{lemma}[Shearing]\label{shearing_lemma}
Let $T\to B\O$ denote the tautological rank-zero virtual vector bundle and $\zeta\colon B\times X\to B\O$ be classified by the rank-zero virtual vector bundle $\xi^*(T)\boxplus (V-r)$. Then $(X, V)$-twisted $\xi$-structures are equivalent to $\zeta$-structures.
\end{lemma}
The proof is given in~\cite[Lemma 10.18]{DDHM22} for $\xi = \Spin$; the general case is completely analogous. Invoking the Pontrjagin-Thom theorem, we then learn:
\begin{corollary}
\label{what_is_twisted_bordism}
There is a notion of bordism of manifolds with $(X, V)$-twisted $\xi$-structures, corresponding to the Thom spectrum $\mathit{MT\xi}\wedge X^{V-r_V}$; thus the bordism groups of these manifolds are $\Omega_*^\xi(X^{V-r_V})$.
\end{corollary}
Here we use the fact that the Thom spectrum functor sends external direct sums to smash products, which is \cref{boxplus_lem}.
\begin{lemma}
\label{presmith}
Suppose $X$ is a closed smooth manifold with a $\xi$-structure and $M\subset X$ is an embedded submanifold such that the image of the mod $2$ fundamental class of $M$ in $H_*(X;\Z/2)$ is Poincaré dual to $e(V)\in H^r(X;\Z/2)$. Then $M$ has a canonical $(X, V)$-twisted $\xi$-structure.
\end{lemma}
\begin{proof}
Because the homology class of $M$ is Poincaré dual to the mod $2$ Euler class of $V$, the normal bundle to $M\hookrightarrow X$ is isomorphic to $V|_M$. Choose a Riemannian metric on $X$; this is a contractible choice, so will not change the connected component of the data we obtain, so as discussed above different choices of metric lead to the same $(X, V)$-twisted $\xi$-structure in the end.

Using the Levi-Civita connection induced by the metric, we may split the short exact sequence of vector bundles over $M$,
\begin{equation}
    \shortexact{}{}{TM}{TX|_M}{\nu},
\end{equation}
thereby obtaining an isomorphism $TM\oplus V|_M\cong TX|_M$. Since $TX$ has a $\xi$-structure, this implies $TM\oplus V|_M$ has a chosen $\xi$-structure, i.e. that we have put a $(X, V)$-twisted $\xi$-structure on $M$.
\end{proof}

\subsection{Smith homomorphisms induced by maps of Thom spectra}
\label{smith_defn_round_1}

We will now apply the previous discussions of Thom spectra and shearing to understand a class of homomorphisms between bordism groups called \textit{Smith homomorphisms}. These map between bordism groups of manifolds of different dimensions and with different tangential structures.

Fix a tangential structure $\xi\colon B\to B\O$ such that its bordism spectrum $\mathit{MT\xi}$ is a ring spectrum (e.g.\ $\O$, $\SO$, $\Spin^c$, $\Spin$). Fix also a virtual vector bundle $V\to X$ of rank $r_V$ and $W\to X$ a vector bundle of rank $r_W$.
\begin{definition}\label{Smith_homomorphism_intersection_defn}
    The \term{Smith homomorphism} associated to $\xi$, $V$, and $W$ is the homomorphism
\begin{equation}
    \sm_W\colon \Omega_n^\xi(X^{V-r_V}) \longrightarrow \Omega_{n-r_W}^\xi(X^{V\oplus W - r_V - r_W})
\end{equation}
that sends a closed $n$-manifold $[M]$ to the bordism class $[N]$, where $N\subset M$ is the submanifold defined as follows: pull back $W$ from $X$ to $M$ and choose a section $s\colon M\to f^*W$ transverse to the zero section. Then, $N\coloneqq s^{-1}(0)$ is an $(n-r_W)$-dimensional manifold whose mod $2$ homology class is Poincaré dual to $e(W)$, hence by \cref{presmith} has a $(X, V\oplus W)$-twisted $\xi$-structure, and we define $\sm_W([M]) \coloneqq [N]$.
\end{definition}

\begin{proposition}[{\cite[\S 4.2]{HKT19}}]
    The bordism class $[N] \in \Omega_{n-r_W}^\xi(X^{V\oplus W - r_V - r_W})$ is independent of the choice of section.
\end{proposition}

\begin{example}
    Let $\xi\colon B\Spin \to B\O$, $X=B\Z/2$, $V=0$, and $W=\sigma\to B\Z/2$, where $\sigma$ is the tautological line bundle. The corresponding Smith homomorphism is
    \begin{equation}
        \Omega_n^\Spin(B\Z/2) \xrightarrow{\sm_\sigma} \Omega_{n-1}^\Spin((B\Z/2)^{\sigma-1}).
    \end{equation}
    After shearing (\cref{shearing_lemma}), we recognize this as
    \begin{equation}
        \Omega^{\Spin\times \Z/2}_n \xrightarrow{\sm_\sigma} \Omega_{n-1}^{\Pinm}.
    \end{equation}
    Letting $V=0$, $\sigma$, $2\sigma$, and $3\sigma$ produces the maps in the four-periodic family discussed in \cref{spin_4periodic}.
\end{example}

Later, in \cref{examplesofSmith}, we thoroughly discuss the history of Smith maps and present many more examples. For the rest of this section, we discuss how Smith homomorphisms are induced by maps of Thom spectra.
Let $X$ be a topological space and $V$ be a rank-$r$ real vector bundle on $X$.
We abuse notation and also denote the associated classifying map by $V\colon X \to B\O(r)$.
The inclusion $0 \hookrightarrow W$ induces a zero section map $X \to X^W$. More generally, we have the following.
\begin{definition}\label{Smith_map_vector_bundle}
    Let $V$ and $W$ be vector bundles on $X$. Let $S^V \to S^{V \oplus W}$ be the map of finite-dimensional spheres over $X$ induced by the zero section map on $W$.
    The \textit{Smith map} associated to $X$, $V$, and $W$ is the map of Thom spaces
    \begin{equation}
        \sm_W\colon \mathrm{Th}(X; V) \to \mathrm{Th}(X; V \oplus W)
    \end{equation} 
    formed as the colimit of the map of spheres.
\end{definition}

\begin{definition}\label{Smith_map_virtual_bundle}
    In the case that we have a virtual bundle $V$, the zero section map induces a map of stable spherical fibrations $\mathbb{S}^{V} \to \mathbb{S}^{V \oplus  W} \simeq \mathbb{S}^{V} \wedge \mathbb{S}^W$ over $X$.
    Taking the colimit, we get a map of Thom spectra
    \begin{equation}
    \label{spectral_sm_map}
        \sm_W\colon X^V \to X^{V \oplus W}
    \end{equation}
    which we also call a \textit{Smith map}.
\end{definition} 
\begin{proposition}
\label{two_smith_1}
The map on $\xi$-bordism groups induced by the map~\eqref{spectral_sm_map} of spectra is equal to the Smith homomorphism as defined in \cref{Smith_homomorphism_intersection_defn}.
\end{proposition}
This follows by unpacking the Pontrjagin-Thom isomorphism.

\section{Euler classes and Smith homomorphisms}

In this section, we develop an alternate definition of the Smith homomorphism via the Euler class.

\subsection{Euler classes in generalized cohomology}
\label{ss:tw_Euler}
Fix $\xi \colon X \to B\O$ a tangential structure and $W \colon X \to B\O(r_W)$ a vector bundle on $X$.
We would like to describe the Smith homomorphism on $\xi$-bordism groups as taking a manifold $(M, p\colon M \to X)$ with $\xi$-structure to a smooth representative of the Poincaré dual of $e(p^* W)$, where $e(p^* W)\in H^{r_W}(M;\Z_{w_1(W)})$ is the Euler class of $W$.
This, however, is \emph{not} true in general, as we show in Appendix~\ref{s:eu_counter}---we need to upgrade what we mean by the Euler class.

We will define an Euler class living in twisted cobordism. More generally, for $\mathcal R$ an $\mathbb{E}_1$ ring spectrum, we define a $\mathcal R$-valued Euler class in the $\mathcal R$-cohomology of $X^{-W}$. In the case we have an untwisting, given by a $\mathcal R$-orientation on $W$, we will see in \cref{euler_is_Thom} that the \emph{untwisted} Euler class is the pullback of the Thom class $U^{\mathcal R}(W) \in \mathcal R^{r_W}(\mathrm{Th}(X;W))$ along the $0$ section $X \to \mathrm{Th}(X;W)$ (e.g.\ in~\cite[\S 13]{Bec70}), so that our definition deserves to be called an Euler class; we also generalize to the twisted setting where there is no Thom class.

Recall the setup of \cref{Smith_map_virtual_bundle}.
Let $0$ be the vector bundle over $X$ of rank zero. The zero section gives a map $0 \to W$ of vector bundles over $X$. Therefore we get a map of stable spherical fibrations 
\begin{subequations}
\begin{equation}
    z\colon \mathbb S^0\longrightarrow\mathbb S^W,
\end{equation}
i.e.\ a fiberwise map of spectra. Because $0$ is the trivial rank-zero vector bundle, $\mathbb S^0$ is the constant stable spherical fibration $\underline{\mathbb S}$ with fiber $\mathbb S$.

Apply the duality $\mathrm{Map}(\text{--}, \mathbb S)$ fiberwise to obtain another map
\begin{equation}
\label{zvee}
    z^\vee\colon \mathbb S^{-W}\longrightarrow \mathbb S^0.
\end{equation}
Because the codomain of $z^\vee$ is constant as a functor $X\to\cat{Sp}$, there is an induced map of spectra:
\begin{equation}\label{Euler_appearance}
    e^{\mathbb S}(W):  X^{-W} = \operatorname{colim}_X \mathbb S^{-W} \to \mathbb S.
\end{equation}
\end{subequations}
\begin{definition}
The class $e^{\mathbb S}(W)$ is called the \term{stable cohomotopy Euler class} of $W$. Usually, we will interpret generalized cohomology of $X^{r_W-W}$ as the $(-W)$-twisted generalized cohomology of $X$, meaning $e^{\mathbb S}(W)$ is an element of the degree-$r_W$ $(-W)$-twisted stable cohomotopy of $X$.
\end{definition}
\begin{remark}
    This cohomology class of $e^{\mathbb S}(W)$ lives in $(\mathbb S)^0(X^{-W})$. By the Pontrjagin-Thom isomorphism, this is equivalent to the twisted cobordism group $\Omega_{\mathrm{fr}}^0(X, -W)$.
\end{remark}
\begin{definition}
\label{twisted_Euler_class}
Let $\mathcal R$ be a ($\mathbb{E}_1$)-ring spectrum, so that there is a unique ring map $1_{\mathcal R}\colon \mathbb S\to\mathcal R$. The \term{$\mathcal{R}$-cohomology Euler class of $W$}, denoted $e^{\mathcal R}(W)$, is the composition $1_{\mathcal R}\circ e^{\mathbb S}(W)$. As in the previous definition, we interpret this as an element of the degree-$r_W$ $(-W)$-twisted $R$-cohomology of $X$.
\end{definition}

Now we see how the Euler class and Smith homomorphism are related:
\begin{proposition}\label{prop:euler-and-smith}\hfill
    \begin{enumerate}
        \item\label{ES1}     Let $0$ be the trivial rank $0$ vector bundle on $X$; then $e^{\mathbb{S}}(0)\colon \Sigmainfty X \to \mathbb S$ is the infinite suspension of the crush map $X \to *$.
        \item\label{ES2} Let $W$ be a vector bundle on $X$ and $\sm_W\colon X^{-W} \to X$ be the Smith map. Then
        $e^{\mathbb{S}}(W) = (\sm_W)^* (e^{\mathbb{S}}(0))$.
    \end{enumerate}
\end{proposition}
\begin{proof}
    For part~\ref{ES1}: $0$ defines the trivial stable spherical fibration on $X$, which factors through a point. Therefore the Euler class of $0$ is the pullback of the Euler class of the trivial bundle over a point. 

    For part~\ref{ES2}: this follows from the fact that $e^{\mathbb{S}}(W)\colon X^{-W} \to \mathbb{S}$ factors through
    \begin{equation*}
        X^{-W} \xrightarrow{\sm_W} X \xrightarrow{ e^{\mathbb{S}}(0)} 0.\qedhere
    \end{equation*}
\end{proof}
We immediately learn that Smith maps pull back Euler classes.
\begin{corollary}
Given a virtual vector bundle $V$ and a vector bundle $W$, let $\mathrm{sm}_W$ denote the Smith homomorphism $\mathrm{sm}_W\colon X^{-V\oplus -W} \to X^{-V}$. Then
\begin{equation}
    \mathrm{sm}_W^*(e^{\mathbb{S}}(V)) = e^{\mathbb{S}}(V \oplus W).
\end{equation}
\end{corollary}
We can thus recover the Smith homomorphism from capping with the twisted Euler class.
\begin{proposition}\label{prop:smith-is-cap-with-euler}
    For any virtual bundle $V$ on $X$, the Smith map $X^{V} \to X^{V \oplus W}$ can be defined as the following composition:
    \begin{equation}\label{Smith_Smith_WTS}
        X^{V} \simeq X^{(V \oplus W) \oplus - W} \xrightarrow{\Delta} (X \times X)^{(V \oplus W) \boxplus -W} \simeq X^{V \oplus W} \wedge X^{-W} \xrightarrow{e^{\mathbb S}(W)} X^{V \oplus W}.
    \end{equation}
    The map $X^{(V \oplus W) \oplus - W} \xrightarrow{\Delta} (X \times X)^{(V \oplus W) \boxplus -W}$ is induced by 
    the diagonal map $\Delta\colon X \to X \times X$.
\end{proposition}
\begin{proof}
The Euler map for the trivial rank $0$ vector bundle 
   \begin{equation}
   \label{zero_Euler}
    X^0 \simeq \Sigma^\infty_{+} X \xrightarrow{e^{\mathbb S}(0)} \mathbb S.
    \end{equation}
is the counit for the $\mathbb{E}_\infty$-coalgebra structure on $\Sigma^\infty_{+} X$. By \cref{prop:euler-and-smith}, the Euler class $e^{\mathbb S}(W)$ factors through~\eqref{zero_Euler} as
\begin{equation}
    X^{-W} \longrightarrow X^{-W \oplus W} \simeq \Sigma^\infty_{+} X \xrightarrow{e^{\mathbb S}(0)}  \mathbb S.
\end{equation}
This implies that~\eqref{Smith_Smith_WTS} can be written as 
   \begin{equation}\label{big_diag_1}
    \begin{tikzcd}[column sep=0.5cm]
        X^{V} \simeq X^{(V \oplus W) \oplus - W} \ar[r, "\Delta"] \ar[rrd, "\phi"',bend right=8] & (X \times X)^{(V \oplus W) \boxplus {-W}}   \ar[r] & (X \times X)^{(V \oplus W) \boxplus 0} \ar[r, "\simeq"] & X^{V \oplus W} \wedge \Sigma^\infty_{+} X \ar[r, "e^{\mathbb S}(W)"] & X^{V \oplus W} \\ 
        & & X^{V \oplus W} \ar[u, "\Delta"]  &
    \end{tikzcd}
   \end{equation}
   Since the map $X^{V \oplus W} \to (X \times X)^{(V \oplus W) \boxplus 0}\simeq X^{V \oplus W} \wedge \Sigma^\infty_{+} X $ comes from the comodule structure of $X^{V\oplus W}$ over $\Sigma_+^\infty X$,
 the composite $X^{V \oplus W} \to (X \times X)^{(V \oplus W) \boxplus 0} \to X^{V \oplus W}$ is the identity map. Therefore it is sufficient to show that the map $\phi$ in~\eqref{big_diag_1} is homotopy equivalent to the spectral Smith map $\mathrm{sm}_W$, and this follows by restricting to the diagonal in the map $(X \times X)^{(V \oplus W) \boxplus {-W}}  \to  (X \times X)^{(V \oplus W) \boxplus 0} $ along the top of~\eqref{big_diag_1}, which is induced from $\mathrm{id} \boxplus \mathrm{sm}_W$.
\end{proof}

We see that the Euler class records all the ``Smith'' information about $W$. 
We will therefore refer to the Smith homomorphism as capping with the Euler class or as the map of Thom spectra interchangeably.

The dual version of \cref{prop:smith-is-cap-with-euler} also holds.
\begin{proposition}
\label{coho_smith}
Let $\mathcal R$ be a ring spectrum. Then the pullback map on $\mathcal R$-cohomology $\mathrm{sm}_W^*\colon \mathcal R^*(X^{V\oplus W})\to \mathcal R^*(X^{V})$ is equal to the cup product with $e^{\mathcal R}(W)$.
\end{proposition}

\begin{remark}
The long exact sequence of field theories we shall discuss in \cref{LESIFTs} is cohomological in nature: it is given by applying $I_\Z\mathit{MT\xi}$-cohomology to $\mathrm{sm}_W$. However, \cref{coho_smith} does not apply: the Smith homomorphism there cannot be described as taking the product with an $I_\Z\mathcal R$-Euler class. This is because if $\mathcal R$ is a ring spectrum, $I_\Z\mathcal R$ usually admits no ring spectrum structure. However, $I_\Z\mathcal R$ is an $\mathcal R$-module, so we do learn from \cref{coho_smith} that this Smith homomorphism is the cup product with $e^{\mathcal R}(W)$ using the $\mathcal R$-module structure. For example, when we study fermionic invertible phases, we will typically choose $\mathcal R = \MTSpin$.
\end{remark}

Let us review the standard story that ``the Euler class is the pullback of the Thom class along the zero section.'' First we review orientations and Thom classes. For simplicity, we will define them only for vector bundles, though the story generalizes to virtual bundles and much more. 

\begin{definition}\label{def:R-orientation-on-vector-bundles}
    Let $W$ be a vector bundle of rank $n$ on $X$. Fix $\mathcal{R}$ an $\mathbb{E}_1$-algebra in spectra and let $\cat{Mod}_R$ be the $\infty$-category of $\mathcal{R}$-module spectra. An \term{$\mathcal{R}$-orientation} of $W$ is a natural isomorphism $\phi$ of functors between 
    \begin{equation}
    \mathcal{R}^W\colon X \xrightarrow{W} B\O(n) \to \cat{Sp} \xrightarrow{\text{--}\wedge \mathcal{R}} \cat{Mod}_{\mathcal R}
    \end{equation}
    and the constant functor valued in $\Sigma^n \mathcal{R}$. An $\mathcal{R}$-orientation of a manifold $M$ means an $\mathcal R$-orientation of $TM$.
    \end{definition}
    \begin{remark}
    The map $z^\vee$ from~\eqref{zvee} is similar to an orientation on $-W$, in the sense of Ando-Blumberg-Gepner-Hopkins-Rezk, except that $z^\vee$ is in general non-invertible and between different suspensions of the sphere spectrum.
\end{remark}
    An $\mathcal{R}$-orientation $\phi$ on $W$ induces an equivalence
\begin{equation}
    \colim_{X} \mathcal{R}^W \simeq \Sigma^{\infty}_+ \mathrm{Th}(X;W) \wedge \mathcal{R} \simeq X \wedge \Sigma^n \mathcal{R} \simeq \Sigma^n \Sigma^\infty_+ X \wedge \mathcal{R}.
\end{equation}
\begin{definition}
    The composite 
    \begin{equation}
        U\colon \Sigma^{\infty}_+ \mathrm{Th}(X;W)  = X^W \to \Sigma^{\infty}_+ \mathrm{Th}(X;W) \wedge \mathcal{R} \simeq \Sigma^n \Sigma^\infty_+ X \wedge \mathcal{R} \to \Sigma^n \mathcal{R}
    \end{equation}
    is the \term{Thom class}. Often we think of $U$ through its
    homotopy class, which lives in $\mathcal{R}^n(\mathrm{Th}(X;W))$.
\end{definition}

Given a $\mathcal{R}$-orientation on $W$, we can also define the (untwisted) Euler class of $W$. This is a standard definition (e.g.~\cite[\S 13]{Bec70}).
\begin{definition}
\label{untwisted_Euler}
    Given an $\mathcal{R}$-orientation, we have a natural isomorphism of functors $X\to\cat{Mod}_R$
    \begin{equation}
        R^{-W} \simeq \Sigma^{-n} \underline{\mathcal{R}},
    \end{equation}
    where $\Sigma^{-n} \underline{\mathcal{R}}$ is the constant functor valued in $\Sigma^{-n}\mathcal R$. The composite
    \begin{equation}
    \Sigma^{-n} X \longrightarrow \Sigma^{-n} X \wedge \mathcal{R} \simeq X^{-W}\wedge \mathcal{R} \xrightarrow{\mathrm{sm}_W} X \wedge \mathcal{R} \to \mathcal{R}
    \end{equation}
    is called the (untwisted) \term{Euler class} of $W$. 
\end{definition}
Unlike the twisted Euler class, this untwisted Euler class depends
on the $\mathcal{R}$-orientation.

Finally, we can prove that our definition of the Euler class, \cref{twisted_Euler_class}, coincides with the more standard \cref{untwisted_Euler} when they overlap (i.e.\ when there is an $\mathcal R$-orientation chosen on $V$).
\begin{lemma}
\label{euler_is_Thom}
     Suppose $W$ is $\mathcal R$-oriented, and let $U\in \mathcal R^r(\mathrm{Th}(X;W))$ denote the Thom class. Then $e^{\mathcal R}(W) = z_W^*U$, where $z_W \colon X\to \mathrm{Th}(X;W)$ is the inclusion as the zero section.
\end{lemma}

\begin{proof}
    After suspending, the zero section map becomes the Smith map. Therefore it suffices to show that the following diagram commutes.
    \begin{equation}
        \begin{tikzcd}
           \Sigmainfty X \ar[r, "\text{--}\wedge\mathcal R"] \ar[d, "\textrm{sm}_W"] & \Sigmainfty X \wedge \mathcal{R} \ar[r, "\simeq"] \ar[d, "\textrm{sm}_W\wedge \mathrm{id}_{\mathcal{R}}"] & \Sigma^n X^{-W} \wedge \mathcal{R} \ar[d, "\mathrm{sm}_W"] \\ 
           \Sigmainfty \mathrm{Th}(X;W) \simeq X^W \ar[r, "\text{--}\wedge\mathcal R"] & X^W \wedge \mathcal{R} \ar[r, "\simeq"] &\Sigma^n X \wedge \mathcal{R}.
        \end{tikzcd}
    \end{equation}
    Here the equivalences in the right square are the ones induced by the orientation $\phi$.

    The left-hand square commutes because smashing with $\mathcal R$ is a functor. The right-hand square commutes because the following diagram commutes in
    $\cat{Fun}(X, \cat{Mod}_R)$:
    \begin{equation}
    \begin{tikzcd}
    \underline{\mathcal{R}} \ar[r, "(\phi \wedge \mathcal R^{-W})"] \ar[d, "z^\vee \wedge \mathcal{R}"] & \Sigma^n \mathcal{R}^{-W} \ar[d, "z^\vee \wedge \mathcal{R} \wedge \mathcal R^{W}"] \\
    \mathcal{R}^W \ar[r, "\phi"] & \Sigma^n  \underline{\mathcal{R}},
    \end{tikzcd}
    \end{equation}
    which follows from naturality. Recall that $z^\vee: \underline{\mathcal R} \to \mathcal{R}^W$ is the map of spherical fibrations over $X$ that induces the Smith map.
\end{proof}

\subsection{Smith homomorphisms defined via Atiyah-Poincaré dual of the generalized Euler classes}

Now equipped with the theory of Euler classes, we can give another alternate definition of the Smith homomorphism. Fix $\xi\colon B\to B\O$, $V\to X$ of rank $r_V$, and $W\to X$ of rank $r_W$ as in \cref{Smith_homomorphism_intersection_defn}.
Recall that by \cref{what_is_twisted_bordism}, a class $c\in\Omega_n^\xi(X^{V-r_V})$ can be represented by a closed $n$-manifold $M$ with an $(X,V)$-twisted $\xi$-structure, which includes the data of a map $f\colon M\to X$.

In this subsubsection, we assume that $\mathit{MT\xi}$ is a ring spectrum.
\begin{definition}\label{smith_homomorphism_euler_definition}
    The \term{Smith homomorphism} associated to $\xi$, $V$, and $W$ is the homomorphism
\begin{equation}
    \sm_W\colon \Omega_n^\xi(X^{V-r_V}) \longrightarrow \Omega_{n-r_W}^\xi(X^{V\oplus W - r_V - r_W})
\end{equation}
sending the class $[M]$ to the Poincaré dual of the cobordism Euler class 
$e^{\mathit{MT\xi}}(f^*W)$.
\end{definition}

To show this, we first recall Atiyah duality. 
There is the standard notion of duals in any symmetric monoidal category $\cat{C}$~\cite{Lin78, DP80, DM82}. Here for $\cat{C}$ we take the homotopy category of spectra, which is symmetric monoidal with respect to the smash product $\wedge$. If $A, B$ have duals $A^\vee, B^\vee$, then a morphism $f\colon A \to B$ induces a dual morphism, which we write as $f^\vee\colon B^\vee \to A^\vee$.

\begin{theorem}[{Atiyah duality~\cite[Proposition 3.2 and Theorem 3.3]{Ati61a}}]
\label{atiyah_duality}
Let $M$ be a compact manifold; then $(M/\partial M)^\vee\simeq M^{-TM}$. If $M$ is closed and $V\to M$ is a virtual vector bundle, then $(M^V)^\vee \simeq M^{-TM-V}$.
\end{theorem}

Furthermore, dual spectra provide isomorphisms between 
homology and cohomology groups: let $X$ be a spectrum with a dual $X^{\vee}$; then, for any spectrum $\mathcal{R}$, we have a canonical isomorphism
\begin{equation}\label{eq:dual-iso}
    \mathcal{R}_{*}(X) \overset\cong\longrightarrow \mathcal{R}^{-*}(X^\vee). 
\end{equation}
We call two classes $\alpha\in\mathcal R_*(X)$ and $\beta\in\mathcal R^{-*}(X^\vee)$ \term{Atiyah-Poincaré dual} if $\alpha\mapsto \beta$ under the isomorphism~\eqref{eq:dual-iso}.

Furthermore, this is functorial: given a map $f\colon X \to Y$ of dualizable spectra, let $f^\vee\colon Y^\vee \to X^\vee$ be the dual map. 
We have a commutative square:
\begin{equation}\label{eq:dual-commuting}
    \begin{tikzcd}
        \mathcal{R}_{*}(X) \ar[r, "f_*"] \ar[d, "\simeq"] & \mathcal{R}_*(Y) . \ar[d, "\simeq"]\\
        \mathcal{R}^{-*}(X^\vee) \ar[r, "(f^\vee)^*"] & \mathcal{R}^{-*}(Y^\vee).
    \end{tikzcd}
\end{equation}

Let $\Omega^{\mathrm{fr}}_*(X)$ denote the stably framed bordism of $X$, i.e.\ the bordism groups of manifolds with a map to $X$ and a trivialization of the stable tangent bundle (or equivalently, the stable normal bundle). The Pontrjagin-Thom theorem identifies these bordism groups with the stable homotopy groups of $X$. We learn a neat fact:
\begin{lemma}
    Let $M$ be a closed compact $d$-dimensional manifold. Then
    $M$ defines a canonical class in $\Omega_d^{\mathrm{fr}}(M, -TM) = \mathbb{S}_0(M^{-TM})$. This is the Atiyah-Poincaré dual to the Euler class for the trivial bundle $e^{\mathbb{S}}(0) \in \mathbb{S}^0(M)$.
\end{lemma}

\begin{proof}
    The Euler class is represented by a map $e\colon M_+ \to S^0$ taking $+$ to the basepoint of $S^0$ and the entirety of $M$ to the other point. On the other hand, consider an embedding $\iota\colon M \to \R^N$ and let $\nu$ be the normal bundle. 
    Then $\Sigmainfty \mathrm{Th}(M; \nu) \simeq \Sigma^{-N}M^{-TM}$.
     By the Pontrjagin-Thom construction, the tautological class $[M] \in \Omega^{\mathrm{fr}}_d(M, -TM)$ comes from the collapse map 
    $S^N = (\mathbb{R}^N)^+ \to \mathrm{Th}(M; \nu)$, where $(-)^+$ denotes the one point compactification.

    The result follows from the finite-dimensional description of the evaluation and co-evaluation map of $M$ and $M^{-TM}$ \cite{Ati61}: we have an evaluation map $S^N \to M_+ \wedge \mathrm{Th}(M; \nu)$, representing $\mathbb S \to M \wedge M^{-TM}$. The composite $S^N \to  M_+ \wedge \mathrm{Th}(M; \nu) \xrightarrow{e} S^0 \wedge \mathrm{Th}(M; \nu) = \mathrm{Th}(M; \nu)$ is precisely the Pontrjagin-Thom collapse map.
\end{proof}

Now we see how Atiyah duality interacts with Smith homomorphisms on compact manifolds:
\begin{lemma}\label{lem:atiyah-dual-euler}
    Fix a closed compact manifold $M$. Given a virtual bundle $V\to M$ and a vector bundle $W\to M$, the Atiyah dual $ (\sm_W)^\vee$ of the Smith map
    \begin{equation}
       \sm_W\colon M^V \longrightarrow M^{V\oplus W}
    \end{equation}
    is the Smith map associated to $-TM-V - W$:
    \begin{equation}
       \sm_{W}\colon M^{-TM- V -W} \longrightarrow M^{-TM-V}.
    \end{equation}
    \end{lemma}
\begin{proof}
    Let us do the case $V = 0$; the general case follows in the same way. First we give a space-level description of the Atiyah dual map. Consider the manifold with boundary $D_M(W)$, the disk bundle of $W$. Its tangent bundle is $T(D_M(W)) \cong TM \oplus W$, where we are implicitly pulling back $W$ to $D_M(W)$. Now consider an embedding $\mu_D\colon D_M(W) \to \mathbb{R}^N$. Then $M$, sitting as the zero section, also gets an embedding $\mu_M\colon M \to D_M(W)  \to \mathbb{R}^N$. 
    
    Let $\nu_D$, resp.\ $\nu_M$ be the normal bundle of $\mu_D$, resp.\ $\mu_M$. As virtual bundles,
    \begin{subequations}
    \label{eq:normal-bundle-eq}
    \begin{align}
       \nu_D &\cong \mathbb{R}^N - TM - W\\
       \nu_M &\cong \mathbb{R}^N - TM.  
    \end{align}
    \end{subequations}
    Note that $\nu_M = \nu_D \oplus W$.
    Now let $N_D(\mu)$ be a tubular neighborhood of $D_M(W)$ and $N_M(\mu)$ the same for $M$. Observe that $N_D(\mu)$ and $N_M(\mu)$ are diffeomorphic to $\mu_D$, resp.\ $\mu_M$.
    
    Using the standard Pontrjagin-Thom collapse argument, the open embedding $i\colon N_M(\mu) \to N_D(\mu)$ induces a map of one-point compactifications $i^+\colon N_D(\mu)^+\to N_M(\mu)^+$. By \cref{prop:X-compact-thom-one-point}, we can write this as $\mathrm{Th}(D_M(W); \nu_D)\simeq \mathrm{Th}(M; \nu_D) \to \mathrm{Th}(M;\nu_D)$. Recall that $D_M(W)$ is homotopy equivalent to $W$.
    
    After passing to spectra, \cref{eq:normal-bundle-eq} gives a map
    \begin{equation}
        \Sigma^n M^{-TM-W} \longrightarrow \Sigma^n M^{-TM}.
    \end{equation}
    This is the Atiyah dual map of the Smith map.

    To see this is the Smith map for $-TM-V-W$ as claimed, notice that the composite $\mathrm{Th}(M;\nu_D) \to \mathrm{Th}(D_M(W);\nu_D)\to \mathrm{Th}(M;\nu_D \oplus W)$ is induced by the inclusion of disk bundles, a.k.a.\ the Smith homomorphism on Thom spaces, which suspends to the Smith map on Thom spectra.
\end{proof}

\begin{lemma}\label{lem:name-it-something}
    Let $W$ be a rank $r_W$ vector bundle on a closed compact $d$-manifold $M$, and let $[M] \in \Omega^{\mathrm{fr}}_d(M, -TM)$ be the tautological class. Then $(\mathrm{sm}_W)_*([M]) \in \Omega^{\mathrm{fr}}_d(M, -TM + W) = \mathbb{S}_0(M^{-TM+W})$ is the Atiyah-Poincaré dual of the Euler class $e^{\mathbb S}(W) \in \Omega^{r_W-d}_{\mathrm{fr}}(M, -W)$.
\end{lemma}

\begin{proof}
    By \cref{eq:dual-commuting}, ${\mathrm{sm}_W}_*([M])$ is Atiyah-Poincaré dual to $((\mathrm{sm}_W)^\vee)^*(e^{\mathbb{S}}(0))$, where $0$ denotes the zero vector bundle. By \cref{lem:atiyah-dual-euler}, $(\mathrm{sm}_W)^\vee$ is still $\mathrm{sm}_W$. By \cref{prop:euler-and-smith}, $(\mathrm{sm}_W)^*(e^{\mathbb{S}}(0))$ is $e^{\mathbb{S}}(W)$. 
\end{proof}

Now we can collect our prize: we show that \cref{Smith_map_virtual_bundle,smith_homomorphism_euler_definition} are equivalent definitions of the Smith homomorphism. In other words, the Smith homomorphism as we first defined it is the same as the map taking the Poincaré dual of the Euler class, as it is often described in the literature.
\begin{corollary}
\label{two_smith_2}
Let $V\to X$ be a virtual vector bundle and $W\to X$ be a rank $r_W$ vector bundle. Choose a bordism class in $\Omega^{\mathrm{fr}}_d(X, V)$ (i.e.\ $(X, V)$-twisted framed bordism) and let $M$ be a closed manifold representative of that class. Let $[N] \in \Omega^{\mathrm{fr}}_d(M, -TM\oplus W)$ be the Atiyah-Poincaré dual of the Euler class $e^{\mathbb{S}}(W|_M)$. Then the image of $[N]$ in $\Omega^{\mathrm{fr}}_d(X, V \oplus W)$ is $\mathrm{sm}_W([M])$. 
\end{corollary}

\begin{proof}
    Since $M$ has a $(X, V)$-twisted framing, the map $M\to X$ Thomifies to a map $f\colon M^{-TM} \to X^V$. The Smith map is functorial, so we get a commutative square:
    \begin{equation}
        \begin{tikzcd}
            M^{-TM} \ar[r, "f"] \ar[d, "\mathrm{sm}_{W|_M}"] & X^V \ar[d, "\mathrm{sm}_W"]\\ 
            M^{-TM \oplus W} \ar[r, "f"] & X^{V \oplus W}.
        \end{tikzcd}
    \end{equation}
    Furthermore, $[M] \in  \Omega^{\mathrm{fr}}_d(X, V)$ is the $f$-pushforward of the tautological class in $\Omega^{\mathrm{fr}}_d(M, -TM)$. The result now follows from \cref{lem:name-it-something}.
\end{proof}

\begin{remark}
    This tells us that given a bordism class represented by $M$, $\mathrm{sm}_W([M])$ is represented by a manifold $N$ that is Atiyah-Poincaré dual (in the bordism homology theory) to the twisted cobordism Euler class of $M$.
\end{remark}

\begin{remark}
Sometimes, instead of working with twisted cobordism Euler classes, one can pass to simpler cohomology theories, such as twisted ordinary cohomology. For example, Hason-Komargodski-Thorngren~\cite{HKT19} use exclusively $\Z$-cohomology Euler classes in their discussion of the Smith homomorphism. In Appendix~\ref{s:eu_counter} we explain why, in the generality considered in this paper, ordinary cohomology does not suffice.

In this remark we discuss another interesting example, similar in spirit, which illustrates why cobordism is the right level of generality.\footnote{We thank an anonymous referee for suggesting this example.} Consider the low-energy limit of four-dimensional QCD with two flavors. The fields of this theory on a spin $4$-manifold $X$ include a smooth map $u\colon X\to\SU_2$. Given such a map $u$, and a point $x\in\SU_2$, the preimage $u^{-1}(x)\subset X$ is generically a finite disjoint union of circles, which we think of as the worldlines of skyrmions in the theory.

Given such a circle $\gamma\colon S^1\hookrightarrow u^{-1}(X)$, Witten~\cite{Wit83} showed how to determine whether the corresponding skyrmion is a boson or a fermion; his analysis has been revisited and extended by Freed~\cite[\S 4]{Fre08} and Lee-Ohmori-Tachikawa~\cite[\S 2.5]{LOT}. The circle $\mathrm{Im}(\gamma)$ has a spin structure induced by comparing the spin structure on $X$ and the unique spin structure on $\SU_2$, and we may consider the class $[\mathrm{Im}(\gamma)]\in\Omega_1^\Spin\cong\Z/2$; the skyrmion is a boson if this class is $0$, and a fermion if the class equals $1$. Lee-Ohmori-Tachikawa (\textit{ibid.}) show that $\widetilde\Omega_4^\Spin(\SU_2)\cong\Z/2$, and the assignment
\begin{equation}
\label{SU2_exm}
    \begin{aligned}
        \widetilde\Omega_4^\Spin(\SU_2) &\overset\cong\longrightarrow \Omega_1^\Spin\cong\Z/2\\
        (X, u) &\longmapsto [u^{-1}(x)]
    \end{aligned}
\end{equation}
is an isomorphism (where we assume $u$ and $x$ are generic so that $u^{-1}(x)$ is indeed a closed $1$-manifold).

Real vector bundles on $\SU_2$ are classifed by $[\SU_2, B\O_n]\cong \pi_3(B\O_n) = 0$ for all $n$, since $\SU_2\simeq S^3$, and therefore the map~\eqref{SU2_exm} cannot arise from a Smith homomorphism. However, it is similar in spirit: the worldline of the skyrmion is not enough to determine its statistics, and what one really needs is the spin bordism class.

\end{remark}

\section{The Smith fiber sequence}\label{math_Smithfibersequence}

In this section we extend the Smith map into a fiber sequence, which allows us to derive a long exact sequence of bordism groups and, dually, a long exact sequence of field theories.

For any vector bundle $E\to X$ of rank $r$, let $E$ also denote the classifying map $X\to B\O(r)$. Which of these two things we mean by $E$ will be clear from context.
In this section, we will write $S_X(E)$ and $D_X(E)$ for the sphere, resp.\ disk bundles of $E$, because there will be places where it will help to remember which base space we work over.

The following result is not new; see, e.g.~\cite[Remark 3.14]{KZ18}, where it is attributed to James.
\begin{theorem}
\label{the_cofiber_sequence}
    Let $V, W$ be real vector bundles over $X$. Then there is a cofiber sequence in pointed spaces:
    \begin{equation}
        S_X(W)^{p^*V} \to X^V \to X^{W \oplus V},
    \end{equation}
    where $p\colon S_X(W)\to X$ is the bundle map.
    Similarly, if $V$ is a virtual bundle, we have a (co)fiber sequence in spectra:
    \begin{equation}
    \label{the_cof_seq}
        S_X(W)^{p^*V} \to X^V \to X^{V \oplus W}.
    \end{equation}
\end{theorem}

\begin{proof}
    We will do the case where $V$ is an actual vector bundle; the virtual bundle case is analogous. Given an $r$-dimensional vector space $W$, we have a cofiber sequence in pointed spaces:
    \begin{equation}
    \label{sphere_cof_step}
        S(W)_+ \to D(W)_+ \simeq S^0 \to S^W.
    \end{equation}
    Now since $\mathrm{Aut}(W) \cong \O(r)$ acts on $W$, we can upgrade~\eqref{sphere_cof_step} to a cofiber sequence of spaces with $\O(r)$-actions; equivalently, \eqref{sphere_cof_step} is a cofiber sequence of functors $B\O(r) \to \cat{Top}_*$. Pulling back to $X$ via the map $X\to B\O(r)$ classifying $W$, we get a cofiber sequence of functors $X\to\cat{Top}_*$. Now smash with $S^V$: we get a cofiber sequence of the form 
    \begin{equation}
    \label{third_step_cofiber}
        S(W)_+ \wedge S^V \to D(W)_+ \wedge  S^V \to  S^{W} \wedge  S^V \simeq S^{V \oplus W}.
    \end{equation}
    This cofiber sequence is in the category of functors $X\to\cat{Top}_*$.
    
    Since taking the colimit over $X$ preserves cofiber sequences, it is sufficient to show that the colimit of~\eqref{third_step_cofiber} over $X$ is
    \begin{equation}
        S_X(W)^{p^*V} \longrightarrow X^V \longrightarrow X^{V \oplus W}.
    \end{equation}
For the last term $S^{V \oplus W}$ in~\eqref{third_step_cofiber}, this follows directly from the definition of the Thom spectrum (\cref{vb_thom_ABGHR}).

For $\colim_X(D(W)_+ \wedge  S^V)$, note that $D(W)_+ \simeq S^0$, so $D(W)_+ \wedge  S^V \simeq  S^V$, so \cref{vb_thom_ABGHR} once again tells us the colimit is $X^V$. It also follows that the map $X^V\to X^{V\oplus W}$ on colimits is the Smith map.

Lastly, the colimit of $S(W)_+$ over $X$ is the associated sphere bundle $S_X(W)$. It follows that the colimit of $S(W)_+ \wedge  S^V$ over $X$ is equivalent to the colimit of (the pullback of) $ S^V$ over $S_X(W)$, which is $S_X(W)^{p^*V}$. 
\end{proof}
\begin{remark}
    Everything here is functorial, so given a map $Y \to X$, we get maps between cofiber sequences and therefore a map of long exact sequences of homotopy groups.
\end{remark}
\begin{corollary}
\label{bordism_LES_cor}
    Applying $\pi_*$ to the fiber sequence, we get a long exact sequence of bordism groups:
    \begin{equation}
    \label{smith_the_LES}
        \cdots \longrightarrow \Omega^{\xi}_k(S_X(W)^{p^*V}) \longrightarrow \Omega^{\xi}_k(X^V) \longrightarrow \Omega^{\xi}_{k-r}(X^{V + W -r}) \longrightarrow \Omega^{\xi}_{k-1}(S_X(W)^{p^*V}) \longrightarrow \cdots
    \end{equation}
\end{corollary}
Though written as bordism groups of Thom spectra, these are also twisted $\xi$-bordism groups thanks to \cref{what_is_twisted_bordism}.
We work through an explicit example long exact sequence on the level of manifold generators in \cref{appendix_bordism_LES}.

\begin{remark}\label{rem:G-transitive-on-sphere}
Suppose $X = BG$ for a compact Lie group $G$ and that $W\to X$ is the vector bundle associated to an orthogonal representation of $G$ such that $G$ acts transitively on the unit sphere in $W$. Then the sphere bundle has a particularly simple form: if $G_v$ is the stabilizer subgroup for a point $v\in S(W)$, then the bundle map $p\colon S_X(W)\to X$ is homotopy equivalent to the map $BG_v\to BG$ induced by the inclusion $G_v\hookrightarrow G$. Thus, as we use in \cite{PhysSmith} and allude to in \cref{LESIFTs}, the obstruction for an invertible field theory to be in the image of the Anderson-dualized Smith homomorphism is its restriction from manifolds with $G$-bundles (and some sort of tangential structure) to manifolds with $G_v$-bundles (and the corresponding tangential structure).
\end{remark}
\begin{remark}[Smith and Gysin long exact sequences]
\label{smith_is_gysin}
The reader looking at the type signatures of~\eqref{smith_the_LES} and~\eqref{math_SBLES_2} might notice that they resemble Gysin sequences: long exact sequences involving (co)homology groups of the base space and total space of a sphere bundle, especially because one of the maps can be interpreted as a product with an Euler class. And indeed, if one takes ordinary homology or cohomology, the Smith long exact sequence becomes the Gysin long exact sequence, as can be verified by comparing the three maps in the long exact sequence.

Thus, the Smith long exact sequence can be thought of as the generalization of the Gysin long exact sequence to arbitrary vector bundle twists of generalized cohomology theories.
\end{remark}
\subsection{Comparison with Conner-Floyd's long exact sequence}
Suppose $\xi_1\colon B_1\to B\O$ and $\xi_2\colon B_2\to B\O$ are tangential structures and $\eta\colon \xi_1\to\xi_2$ is a map of tangential structures, i.e.\ a map $B_1\to B_2$ such that $\xi_2\circ\eta = \xi_1$. The map $\eta$ induces a map of Thom spectra $\mathit{MT\xi}_1\to\mathit{MT\xi}_2$ and hence also maps of bordism groups; we will also denote both of these maps by $\eta$.

Conner-Floyd~\cite[\S 16]{CF66a} concoct a long exact sequence
\begin{equation}\label{CFLES}
    \dotsb\longrightarrow
    \Omega_k^{\xi_1} \overset\eta\longrightarrow
    \Omega_k^{\xi_2} \overset j\longrightarrow
    \Omega_k^{\xi_2/\xi_1} \overset\partial\longrightarrow
    \Omega_{k-1}^{\xi_1} \longrightarrow
    \dotsb
\end{equation}
where $\Omega_k^{\xi_2/\xi_1}$ is the bordism group of $k$-dimensional $\xi_2/\xi_1$-manifolds (see below).
\begin{definition}
A \term{$\xi_2/\xi_1$-manifold} is a compact manifold $M$ equipped with the data of
\begin{enumerate}
    \item a $\xi_2$-structure $\mathfrak x$ on $M$,
    \item a $\xi_1$-structure $\mathfrak x_\partial$ on $\partial M$, and
    \item an identification of $\xi_2$-structures $\eta(\mathfrak x_\partial)\overset\simeq\to \mathfrak x|_{\partial M}$.
\end{enumerate}
\end{definition}
Thus a $\xi_2/\xi_1$-structure makes precise the notion of a $\xi_2$-manifold equipped with a $\xi_1$-structured boundary.

Conner-Floyd (\textit{ibid.}) defined a notion of bordism for $\xi_2/\xi_1$-manifolds, and showed that the corresponding bordism groups $\Omega_k^{\xi_2/\xi_1}$ fit into the long exact sequence~\eqref{CFLES},\footnote{Conner-Floyd do not work in this level of generality, only looking at a few tangential structures; nevertheless, their arguments go through in general. Some other examples with more tangential structures appear in~\cite{Sto68, Ale75, Mit75, RST77, Sto85, Lau00, Bun15, BHS20, TY23b, TY23a, Deb23, JFY24}.} where $j$ regards a closed $\xi_2$-manifold as having an empty boundary and $\partial$ takes the boundary of a $\xi_2/\xi_1$-manifold. Like with the Smith long exact sequence, one can verify by hand that~\eqref{CFLES} is exact at each entry.

Where this story meets ours is that Smith long exact sequences are special cases of~\eqref{CFLES}. Specifically, given the data $(\xi, X, V, W)$ that we used to construct the Smith long exact sequence in \cref{bordism_LES_cor}, let $\xi_1$ be $(S(W), V)$-twisted $\xi$-structure, $\xi_2$ be $(X, V)$-twisted $\xi$-structure, and let $\eta$ be induced from the bundle map $S(W)\to X$. Then two-thirds of the Conner-Floyd and Smith long exact sequences coincide, so by the five lemma all three must (homotopically, they are both the homotopy groups of equivalent cofiber sequences of spectra):
\begin{equation}\begin{tikzcd}
	\dotsb & {\Omega_k^\xi(S(W)^V)} & {\Omega_k^\xi(X^V)} & {\Omega_{k-r}^\xi(X^{V+W-r})} & {\Omega_{k-1}^\xi(S(W)^V)} & \dotsb \\
	\dotsb & {\Omega_k^{\xi_1}} & {\Omega_k^{\xi_2}} & {\Omega_k^{\xi_2/\xi_1}} & {\Omega_k^{\xi_1}} & \dotsb
	\arrow[from=1-1, to=1-2]
	\arrow[from=1-2, to=1-3]
	\arrow[from=1-3, to=1-4]
	\arrow[from=1-4, to=1-5]
	\arrow[from=1-5, to=1-6]
	\arrow[from=2-1, to=2-2]
	\arrow["\eta", from=2-2, to=2-3]
	\arrow["j", from=2-3, to=2-4]
	\arrow["\partial", from=2-4, to=2-5]
	\arrow[Rightarrow, no head, from=1-2, to=2-2]
	\arrow[Rightarrow, no head, from=1-3, to=2-3]
	\arrow["\varphi", color={rgb,255:red,214;green,92;blue,92}, from=1-4, to=2-4]
	\arrow[Rightarrow, no head, from=1-5, to=2-5]
	\arrow[from=2-5, to=2-6]
\end{tikzcd}\end{equation}
That is, the map $\varphi$ from $(X, V\oplus W)$-twisted $\xi$-bordism to $\xi_2/\xi_1$-bordism (i.e.\ $(X, V)$-twisted $\xi$-manifolds with an $(S(W), V)$-twisted $\xi$-structure on the boundary) must be an equivalence. But we can do better: we will provide a geometric reason for this equivalence.
\begin{itemize}
    \item The map $\varphi$ is: given an $(X, V\oplus W)$-twisted manifold $M$, with structure map $f\colon M\to X$, $\varphi(M)$ is the disk bundle of $f^*W$; the boundary $S(f^*W)$ has a canonical reduction of structure across $S(W)\to X$, which is to say the map $f\colon D(f^*W)\to M\to X$, when pulled back to $S(f^*W)$, canonically lifts across $S(W)\to X$. Therefore the output of $\varphi$ is indeed a $\xi_2/\xi_1$-manifold.
    \item Going backwards, given a $\xi_2/\xi_1$-manifold $M$, it is possible to construct a bordism of $\xi_2/\xi_1$-manifolds from $M$ to a tubular neighborhood of a manifold representative of the Poincaré dual of the Euler class of $W$. The $\xi_1$-structure on the boundary means this Poincaré dual does not intersect $\partial M$; now we have replaced $M$ with a disk bundle, so we can directly invert $\varphi$ by restricting to the zero section.
\end{itemize}
The reader can then check that $j$ and $\partial$ coincide with their analogues in the Smith long exact sequence.

\section{Periodicity of twists and shearing}
\label{periodicity_and_shearing}
The goal of this section is to provide tools for working with twists of tangential structures. We are interested in collections of similar twists over the same base space; this provides an organizing principle for different Smith homomorphisms that we will use many times in the next section.
\subsection{Families of Smith homomorphisms}
\begin{definition}
\label{smfam}
Fix a space $X$, a virtual vector bundle $V\to X$ of rank $r_V$, a vector bundle $W\to X$ of rank $r_W$, and a tangential structure $\xi$. The \term{family of Smith homomorphisms} associated to this data is the set of Smith homomorphisms
\begin{equation}
\label{Smith_fam}
    \sm_W\colon \Omega^\xi_n(X^{V-r_V + k(W-r_W)}) \longrightarrow \Omega^\xi_{n-r_W}(X^{V - r_V + (k+1)(W-r_W)})
\end{equation}
for $k\in\Z$, i.e.\ the Smith homomorphisms from $(X, V \oplus kW)$-twisted $\xi$-bordism to $(X, V\oplus (k+1)W)$-twisted $\xi$-bordism.

If there is some $\ell\in\Z$ and an identification of $(X, V\oplus kW)$-twisted $\xi$-structures with $(X, V\oplus (k+\ell)W)$-twisted tangential structures for all $k$ that commutes with the Smith homomorphisms~\eqref{Smith_fam}, we say this Smith family is \term{periodic} with period the smallest positive such $\ell$.
\end{definition}
This definition may seem too specific to be very applicable, but we will soon see many examples of periodic families.

The main new
result in this section is \cref{cofiber_periodicity}, providing a way to calculate the periodicity of a family of Smith
homomorphisms.
We also review the theory of shearing in and around \cref{shearing_lem}, which is a convenient way to split the Thom spectra for a wide class of twisted bordism theories, and is an essential step in identifying the terms in Smith long exact
sequences. Our perspective on shearing follows~\cite[\S 1]{DY23}, so see there for some more details; see also~\cite{FH16, Bea17, Ste21, BLM23, DDHM22} for additional approaches.
\begin{definition}
Let $\xi\colon B\to B\O$ be a tangential structure. \term{Two-out-of-three data} for $\xi$ is the data of:
\begin{itemize}
	\item for each pair of $\xi$-structured virtual vector bundles $V,W\to X$, a natural $\xi$-structure on
	$V\oplus W$; and
	\item for each $\xi$-structured virtual vector bundle $V\to X$, a natural $\xi$-structure on $-V\to X$.
\end{itemize}
\end{definition}
The reason for this name is that, given this data, a $\xi$-structure on any two of $V$, $W$, and $V\oplus W$
induces a $\xi$-structure on the third. 
\begin{example}
The tangential structures $\O$, $\SO$, $\Spin^c$, $\Spin$, $\String$, $\U$, $\SU$, and $\mathrm{Sp}$ all have
two-out-of-three data. $\Pin^{\pm}$ and $\Pin^c$ do not.
\end{example}
If $M$ and $N$ are manifolds, $T(M\times N)\cong p_1^*(TM)\oplus p_2^*(TN)$, where $p_1$ and $p_2$ are the
projections of $M\times N$ onto $M$, resp.\ $N$, so two-out-of-three data induces a ring structure on
$\Omega_*^\xi$ given by the direct product of manifolds. More abstractly, this data makes $B$ into a grouplike
$E_\infty$-space and $\xi$ into an $E_\infty$-map, where $B\O$ has the direct sum $E_\infty$-structure. This
implies by work of Lewis~\cite[Theorem IX.7.1]{LMSM86} (see also \cite{May77, ABG18}) that $\mathit{MT\xi}$ is an $E_\infty$-ring spectrum.

For $R$ an $E_\infty$-ring spectrum, May~\cite[\S III.2]{May77} defines a grouplike $E_\infty$-space $\GL_1(R)$,
and Ando-Blumberg-Gepner-Hopkins-Rezk~\cite{ABGHR14a, ABGHR14b} associate to a map $f\colon X\to B\GL_1(R)$, which
we call a \term{twist} of $R$, a Thom spectrum $Mf\in\cat{Mod}_R$. The $f$-twisted $R$-homology groups of $X$ are
by definition the homotopy groups of $Mf$~\cite[Definition 2.27]{ABGHR14a}. Homotopy-equivalent twists induce
equivalent Thom spectra. All of this generalizes our discussion around \cref{vb_thom_ABGHR}, for which $R = \mathbb S$.
\begin{example}[Vector bundle twists]
\label{vb_thom}
We have been using (rank-zero virtual) vector bundles to define twists of bordism theories, and these two notions
of twist are compatible: rank-zero virtual vector bundles $V\to X$ are classified by maps $f_V\colon X\to B\O$, and
the $J$-homomorphism is a map $B\O\to B\GL_1(\mathbb S)$; then, if $\xi$ is a tangential structure with
two-out-of-three data, the unit map $e\colon \mathbb S\to\mathit{MT\xi}$ induces a map $e\colon B\GL_1(\mathbb S)\to
B\GL_1(\mathit{MT\xi})$. The Thom spectrum for $(X, V)$-twisted $\xi$-bordism, as we characterized it in \cref{what_is_twisted_bordism},
is naturally equivalent to the Thom spectrum $M(e\circ J\circ f_V)$ of the map
\begin{equation}
\label{vb_factor}
	X\overset{f_V}{\longrightarrow} B\mathrm O\overset{J}{\longrightarrow} B\GL_1(\mathbb S)\overset{e}{\longrightarrow}
	B\GL_1(\mathit{MT\xi}).
\end{equation}
This is a combination of theorems of Lewis~\cite[Chapter IX]{LMSM86} and
Ando-Blumberg-Gepner-Hopkins-Rezk (see~\cite[Corollary 3.24]{ABGHR14a} and~\cite[\S 1.2]{ABGHR14b}).
\end{example}
\begin{theorem}[{May-Quinn-Ray~\cite[Lemma IV.2.6]{May77}}]
\label{beardsley}
There is a canonical null-homotopy of the map
\begin{equation}
	e\circ J\circ \xi\colon B\to B\GL_1(\mathit{MT\xi}),
\end{equation}
and therefore $e\circ J$ factors through the cofiber $B\O/B$, 
and in fact through $B\GL_1(\mathbb S)/B$.
\end{theorem}
In other words, the homotopy type of the Thom spectrum for $(X, V)$-twisted $\xi$-bordism depends only on
the image of $V$ in $B\O/B$. And the key slogan is that the orders of elements in $[X, B\O/B]$ control the
periodicity of families of Smith homomorphisms for twisted $\xi$-bordism; the group structure on $[X, B\O/B]$ uses
the fact that $B\O/B$ is the cofiber of a map of grouplike $E_\infty$-spaces, and $B\O$ is connected, so $B\O/B$ is also a grouplike
$E_\infty$-space (see~\cite[\S 1.2.2]{DY23}). Thus homotopy classes of maps into $B\O/B$ naturally
form abelian groups.
\begin{definition}[{Bhattacharya-Chatham~\cite[Definition 2.9]{BC22}}]
The \term{$\mathit{MT\xi}$-orientation order} of a virtual vector bundle $V\to X$, written $\Theta(V, \mathit{MT\xi})$, is the smallest positive integer $k$ such that $V^{\oplus k}$ is $\mathit{MT\xi}$-oriented, or infinity if no such $k$ exists.
\end{definition}
Equivalently, $\Theta(V, \mathit{MT\xi})$ is the order of the classifying map of $e\circ J\circ f_V\colon X\to B\GL_1(\mathit{MT\xi})$, where $f_V\colon X\to B\O$ is the classifying map of $V$. By \cref{beardsley}, $e\circ J\circ f_V$ factors through $[X, B\O/B]$, so $\Theta(V, \mathit{MT\xi})$ divides the exponent of $[X, B\O/B]$. We will use this fact below to make quick estimates of orientation orders.
\begin{proposition}
\label{cofiber_periodicity}
Let $V\to X$ be a vector bundle.
If $\epsilon\coloneqq \Theta(V, \mathit{MT\xi})$ is finite, the Smith homomorphism family of $(X, kV)$-twisted $\xi$-bordism is $\epsilon$-periodic.
\end{proposition}
This bound is not sharp, as we discuss in \S\ref{not_sharp}.
\begin{proof}
The image
$\overline{f_V}$ of the classifying map $f_V\colon X\to B\O$ in $[X, B\O/B]$ satisfies $(k+\epsilon)\overline{f_V}
= k\overline{f_V}$. Since the homotopy type of the Thom spectrum for $(X, W)$-twisted $\xi$-bordism only depends on
$\overline{f_W}\in[X, B\O/B]$, this implies that the notions of $(X, kV)$-twisted $\xi$-bordism and $(X,
(k+\epsilon)V)$-twisted spin bordism coincide, so the Smith family $\set{(X, kV): k\in\Z}$ is $\epsilon$-periodic.
\end{proof}
\subsection{Examples of periodic Smith families}
Though \cref{cofiber_periodicity} seems abstract, it lends itself readily to examples.
\begin{example}[Unoriented bordism families are $1$-periodic]
\label{unoriented_periodicity}
\Cref{cofiber_periodicity} implies that when $\xi = \mathrm{id}\colon B\O\to B\O$, the periodicity of a Smith family of
$(X, kV)$-twisted unoriented bordism divides the exponent of $[X, B\O/B\O] = 0$. In other words, all Smith families
of twisted unoriented bordism are $1$-periodic.

We will see some examples of Smith families for unoriented bordism in \cref{Z2_MO,cpx_triv_exms,quater_exm}.
\end{example}
\begin{example}[Oriented bordism families are $2$-periodic]
\label{periodicity_of_SO}
Because $B\SO$ is the fiber of $w_1\colon B\O\to K(\Z/2, 1)$, and the Whitney sum formula implies $w_1$ is a map of
$E_\infty$-spaces, the cofiber $B\O/B\SO$ is equivalent to $K(\Z/2, 1)$ as grouplike $E_\infty$-spaces. Thus for
all spaces $X$, $[X, B\O/B\SO]$ is annihilated by $2$, so all Smith families for twisted oriented bordism are
$2$-periodic (or $1$-periodic).

We will see some examples of Smith families for oriented bordism in \cref{SO_Z2_exm,Wall_exm,cpx_triv_exms,another_Wall_exm,quater_exm}.
\end{example}
\begin{example}[Complex and \spinc bordism families are $2$-periodic]
\label{MU_MSpinc}
If $V$ is a real vector bundle, then $V\oplus V$ has a canonical complex structure (think of this bundle as
$V\oplus iV$), and therefore also a canonical \spinc structure. Therefore for any map $f\colon X\to B\O$, $2f$
lifts to $B\U$ and to $B\Spin^c$. Therefore the image of the map $[X, B\O]\to [X, B\O/B\U]$ has exponent $2$ (and
likewise for $\Spin^c$), so by \cref{cofiber_periodicity} all Smith families of complex and \spinc bordism are at
most $2$-periodic.

For examples of Smith families for \spinc bordism, see \cref{spinc_Z2_exm,cpx_triv_exms,quater_exm} and Footnote~\ref{spinu_footnote}.
\end{example}
\begin{example}[Spin bordism families are $4$-periodic]
\label{periodicity_of_spin}
$B\O/B\Spin$ is not equivalent to a product of Eilenberg-Mac Lane spaces even as an $E_1$-space~\cite[Lemma
1.37]{DY23}, so we cannot reuse the strategy of~\eqref{periodicity_of_SO}. However, there is a cofiber sequence of
grouplike $E_\infty$-spaces (heuristically, an extension of abelian $\infty$-groups)~\cite[\S 1.2.3]{DY23}
\begin{equation}
\label{non_split_Spin}
	K(\Z/2, 2)\longrightarrow B\O/B\Spin \longrightarrow K(\Z/2, 1),
\end{equation}
inducing a long exact sequence on $[X, \bl]$. Since $[X, K(\Z/2, 2)]$ and $[X, K(\Z/2, 1)]$ both have exponent at
most $2$ for any $X$, exactness implies $[X, B\O/B\Spin]$ has exponent at most $4$. Thus using
\cref{cofiber_periodicity} we conclude that all twisted spin bordism Smith families are at most $4$-periodic; in
fact, \cref{spin_4periodic} has period exactly $4$, which implies~\eqref{non_split_Spin} does not split. One could
also argue $4$-periodicity similarly to \cref{MU_MSpinc}.

If we restrict to oriented vector bundles, we can do better, as periodicity is controlled by maps into
$B\SO/B\Spin\simeq K(\Z/2, 2)$ (the argument is similar to $B\O/B\SO\simeq K(\Z/2, 1)$ from
\cref{periodicity_of_SO}). Therefore we conclude that twisted spin Smith families using an oriented vector bundle
are $2$-periodic.

We will discuss several examples of $1$-, $2$-, and $4$-periodic Smith families for spin bordism in \cref{spin_4periodic,spinc_spin_Smith,Zk_spin,another_Wall_exm,quater_exm,spinh_exm}.
\end{example}
\begin{remark}
Periodicity for spin bordism also implies periodicity for $\ko$ and $\KO$. Bhattacharya-Chatham~\cite[Main Theorem 1.5]{BC22} generalize this periodicity of $\KO$-orientability to higher real $K$-theories $\mathit{EO}_\Gamma$.
\end{remark}
\begin{example}[Families of twisted string structures]
\label{string_not_periodic}
The space $B\Spin$ is $3$-connected, with $\pi_4(B\Spin)\cong\Z$; if one kills this homotopy group by taking the $4$-connected cover, one gets a space $B\String$, and the corresponding tangential structure is called a \term{string structure}~\cite[Definition 5.0.3]{ST04} (see also~\cite[\S 1]{Gia71}). The generator of $H^4(B\Spin;\Z)$ is not the first Pontrjagin class $p_1$, but rather is a class $\lambda$ with $2\lambda = p_1$~\cite[Theorem 1.2]{Tho62}. Thus a string structure on a spin vector bundle is equivalent data to a trivialization of $\lambda$.

As grouplike $E_\infty$-spaces, $B\O/B\String$ is an extension of $B\O/B\Spin$ by $B\Spin/B\String\simeq K(\Z, 4)$
(see~\cite[\S 1.2.4]{DY23}); since $B\O/B\Spin$ is itself an extension of $K(\Z/2, 1)$ by $K(\Z/2, 2)$, if $X$ is a
$3$-connected space, $[X, B\O/B\String]\cong H^4(X;\Z)$. Thus for a general space $X$, \cref{cofiber_periodicity} provides no information on Smith families for twisted
string bordism: it reports that the period is at most infinity. We will nevertheless prove in \cref{twstr_cor} that all twisted string Smith families have finite period, though our proof does not provide an effective computation of the period.

In special cases, though, \cref{cofiber_periodicity} allows us to provide sharper bounds: for example, because $H^*(B\Z/2;\Z)$ is
$2$-torsion in positive degrees and $[B\Z/2, B\O/B\Spin]$ has exponent $4$, the long exact sequence associated to
the cofiber sequence $K(\Z, 4)\to B\O/B\String\to B\O/B\Spin$ implies $[B\Z/2, B\O/B\String]$ has exponent at most
$8$, implying that all Smith families of $(B\Z/2, V)$-twisted string bordism are at most $8$-periodic; an
$8$-periodic example appears in~\cref{string_Z2}.
\end{example}
\subsection{Examples of twisted bordism}
In this subsection, we discuss how to use the perspective we have been developing to concretely identify examples of twists of
$\xi$-bordism for the tangential structures $\SO$, $\Spin^c$, and $\Spin$.
\begin{lemma}[{Shearing~\cite[\S 1.2]{ABGHR14b}}]
\label{shearing_lem}
If a twist $f\colon X\to B\GL_1(M\xi)$ factors through a map $g_V\colon X\to B\O$ classifying a rank-zero virtual
vector bundle $V\to X$ as in~\eqref{vb_factor}, then $Mf\simeq \mathit{MT\xi}\wedge X^V$.
\end{lemma}
We will use this lemma as follows: first, for the four tangential structures $\xi\colon BG\to B\O$ mentioned above,
we compute the homotopy type of $B\O/BG$ and understand the map $B\O\to B\O/BG$, to recognize when a map $X\to
B\O/BG$ comes from a (virtual rank-zero) vector bundle $V\to X$. In that situation, \cref{shearing_lem} describes
the corresponding twisted $\xi$-bordism groups as $\Omega_*^\xi(X^V)$, so we can use the Smith homomorphism tools
we developed in this paper.
\begin{example}[Twists of oriented bordism]
\label{oriented_shearing}
Recall from \cref{periodicity_of_SO} that $B\O/B\SO\simeq K(\Z/2, 1)$; the argument there implies the map
$B\O\to B\O/B\SO\xrightarrow{\simeq} K(\Z/2, 1)$ is the first Stiefel-Whitney class. Given a map $a\colon X\to
B\O/B\SO$, the Thom spectrum of the corresponding twist $f_a\colon X\to B\GL_1(\MTSO)$ of $\MTSO$ is the bordism
spectrum whose homotopy groups are the bordism groups of manifolds $M$ with a map $\phi\colon M\to X$ and a
trivialization of $w_1(M) - \phi^*(a)$.\footnote{\label{_normal}Strictly speaking, what one trivializes is
$w_1(\nu) - \phi^*(a)$, where $\nu\to M$ is the stable normal bundle, but there is a canonical identification of
$w_1(M)$ and $w_1(\nu)$.  This nuance will matter for spin structures.}

Every class $a\in H^1(X;\Z/2)$ is
the first Stiefel-Whitney class of some line bundle $L_a\to X$, so for any twist $f\colon X\to B\GL_1(\MTSO)$
described by a map $f_a\colon X\xrightarrow{a} K(\Z/2, 1)\simeq B\O/B\SO\to B\GL_1(\MTSO)$, there is a homotopy
equivalence
\begin{equation}
	Mf\xrightarrow{\simeq} \MTSO\wedge X^{L_a - 1}.
\end{equation}
For example, unoriented bordism is an example of such a twist: every manifold $M$ has a canonical map to $K(\Z/2,
1)$, given by $w_1(M)$, and $w_1(M) - w_1(M)$ has a canonical trivialization. Therefore unoriented bordism is
twisted oriented bordism for the twist $K(\Z/2, 1)\xrightarrow{\simeq} B\O/B\SO$, and \cref{shearing_lem} implies
$\MTO\simeq \MTSO\wedge (K(\Z/2, 1))^{\sigma - 1}$, where $\sigma\to B\Z/2\simeq K(\Z/2, 1)$ is the tautological
line bundle; this is a theorem of Atiyah~\cite[Proposition 4.1]{Ati61}.

For another example of how to use \cref{shearing_lem}, let $\mathcal W$ denote the Thom spectrum for the notion of
bordism of manifolds $M$ equipped with a lift of $w_1(M)$ to a class $\alpha\in H^1(M;\Z)$. The class $\alpha$ is
equivalent to a map $\phi\colon M\to B\Z = S^1$, and $\alpha = \phi^*x$, where $x\in H^1(S^1;\Z)$ is the generator;
rephrased in this way, the condition that $\alpha\bmod 2 = w_1(M)$ is equivalent to a trivialization of $w_1(M) -
\phi^*(x\bmod 2)$. Therefore $\mathcal W$-bordism is twisted oriented bordism for $(S^1, x\bmod 2)$, and as $x\bmod
2$ is $w_1$ of the Möbius bundle $\sigma\to S^1$, we learn from \cref{shearing_lem} that $\mathcal W\simeq
\MTSO\wedge (S^1)^{\sigma-1}$. This is also due to Atiyah~\cite[\S 4]{Ati61}.
\end{example}
\begin{example}[Twists of \spinc bordism]
\label{spinc_shearing}
There is an equivalence of spaces, but not $E_1$-spaces, $B\O/B\Spinc\simeq K(\Z/2, 1)\times K(\Z,
3)$~\cite[Proposition 1.20, Lemma 1.30]{DY23}, and the map $B\O\to B\O/B\Spin^c$ is picked out by $(w_1,
\beta(w_2))$, where $\beta$ is the integral Bockstein. The fact that $\beta(w_2)$ is not linear in the direct sum
of vector bundles is why this decomposition of $B\O/B\Spin^c$ does not respect the $E_1$-structure.

Given data $a\in H^1(X;\Z/2)$ and $c\in H^3(X;\Z)$, if $Mf_{a,c}$ is the Thom spectrum for the corresponding twist
\begin{equation}
	f_{a,c}\colon X\overset{(a, c)}{\longrightarrow} K(\Z/2, 1)\times K(\Z, 3)\simeq B\O/B\Spin^c\longrightarrow
	B\GL_1(\MTSpin^c),
\end{equation}
then the homotopy groups of $Mf_{a,c}$ are the bordism groups of manifolds $M$ with maps $\phi\colon M\to X$ and
trivializations of $w_1(M) - \phi^*(a)$ and $\beta(w_2(M)) - \phi^*(c)$; the proof is essentially the same as
Hebestreit-Joachim's~\cite[Corollary 3.3.8]{HJ20} (Footnote~\ref{_normal} still applies: what appears is the stable
normal bundle, but the characteristic classes are the same). If there is a (rank-zero, virtual) vector bundle $V\to
X$ with $w_1(V) = a$ and $\beta(w_2(V)) = c$, then \cref{shearing_lem} implies $Mf_{a,c}\simeq \MTSpin^c\wedge X^V$
and we can invoke the Smith homomorphism on $V$.

For example, a \pinc structure on a manifold $M$ is a trivialization of $\beta(w_2(M))$ (i.e.\ the \spinc
condition without the trivialization of $w_1$). Thus a \pinc structure is equivalent to a twisted \spinc structure
where $X = B\Z/2$, $a$ is the generator of $H^1(B\Z/2;\Z/2)$, and $c = 0$: as in \cref{oriented_shearing}, $w_1(M)$
gives us a canonical map to $B\Z/2$, there is a canonical trivialization of $w_1(M) - w_1(M)$, and $c = 0$
means this twisted \spinc condition does not modify $\beta(w_2)$. So this twisted \spinc condition is that $\beta(w_2) =
0$ and $w_1$ is arbitrary, i.e.\ a \pinc structure. And if $\sigma\to B\Z/2$ is the tautological line bundle,
$w_1(\sigma) = a$ and $\beta(w_2(\sigma)) = 0 = c$, so \cref{shearing_lem} implies $\MTPin^c\simeq \MTSpin^c\wedge
(B\Z/2)^{\sigma-1}$, reproving a theorem of Bahri-Gilkey~\cite{BG87a, BG87b}.

Other examples of twists of \spinc bordism which can be realized by vector bundles include the spin-$\U(2)$ bordism
of Davighi-Lohitsiri~\cite{DL20, DL21} and the tangential structure corresponding to Stehouwer's alternate class AI
fermionic groups~\cite[\S 2.2]{Ste21}.

Not every choice of $(a, c)$ can be realized by a vector bundle; for example, $\beta(w_2)$ is always $2$-torsion,
but $c$ need not be. There are also examples with $2$-torsion $c$, as a consequence of work of
Gunawardena-Kahn-Thomas~\cite[\S 2]{GKT89}.
\end{example}
\begin{example}[Twists of spin bordism]
\label{twists_of_spin}
The most commonly studied examples of twisted $\xi$-bordism in mathematical physics are twists of spin bordism. The
story is closely analogous to \cref{spinc_shearing}, with $K(\Z, 3)$ replaced with $K(\Z/2, 2)$, and the map
$B\O\to B\O/B\Spin\simeq K(\Z/2, 1)\times K(\Z/2, 2)$ is $(w_1, w_2)$. Given classes $a\in H^1(X;\Z/2)$ and $b\in
H^2(X;\Z/2)$, the homotopy groups of the Thom spectrum of the corresponding twist $f_{a,b}\colon X\to
B\GL_1(\MTSpin)$ are the bordism groups of manifolds $M$ with maps $\phi\colon M\to X$ and trivializations of
$w_1(\nu) - \phi^*(a)$ and $w_2(\nu) - \phi^*(b)$~\cite[Corollary 3.3.8]{HJ20}, where $\nu\to M$ is the stable
normal bundle. Now, unlike in Footnote~\ref{_normal}, the distinction between $TM$ and $\nu$ matters: $w_1(TM) =
w_1(\nu)$, but $w_2(TM) + w_1(TM)^2 = w_2(\nu)$, providing a formula for the nontrivial transition from tangential
to normal data. If $a = w_1(V)$ and $b = w_2(V)$ for a rank-zero virtual vector bundle $V\to X$, \cref{shearing_lem}
implies $Mf_{a,b}\simeq \MTSpin\wedge X^V$.  See~\cite[\S 1.2.3]{DY23} for more information.

Many commonly studied tangential structures arise as vector bundle twists of spin structures.
\begin{enumerate}
	\item A \pinm structure is a trivialization of $w_2(M) + w_1(M)^2$, with no condition on $w_1$. Thus this is
	equivalent to a trivialization of $w_2(\nu)$. Like in \cref{oriented_shearing,spinc_shearing}, we can ask for a
	map $\phi\colon M\to B\Z/2$ and a trivialization of $w_1(\nu) - \phi^*(a)$, where $a\in H^1(B\Z/2;\Z/2)$ is the
	generator, and this is no data at all; then we also want to impose $w_2(\nu) = 0$. So \pinm bordism is
	the Thom spectrum of the twist $f_{a,0}\colon B\Z/2\to B\GL_1(\MTSpin)$. The classes $a$ and $0$ are $w_1$ and
	$w_2$ of $\sigma\to B\Z/2$, so we learn that $\MTPin^-\simeq\MTSpin\wedge (B\Z/2)^{\sigma-1}$, a splitting first
	written down by Peterson~\cite[\S 7]{Pet68}.
	\item A \pinp structure is a trivialization of $w_2(M)$, with no condition on $w_1$. Switching to the stable
	normal bundle, we want a trivialization of $w_2(\nu) + w_1(\nu)^2$. Just as for \pinm structures, pick a map
	$\phi\colon M\to B\Z/2$ and ask for a trivialization of $w_1(\nu) - \phi^*(a)$, which is no data; then we want
	to trivialize $w_2(\nu) + \phi^*(a^2)$. Thus \pinp bordism is the Thom spectrum of the twist $f_{a,a^2}\colon
	B\Z/2\to B\GL_1(\MTSpin)$. The classes $a$ and $a^2$ are $w_1$, resp.\ $w_2$ of the virtual vector bundle
	$-\sigma$, so \cref{shearing_lem} tells us $\MTPin^+\simeq \MTSpin\wedge (B\Z/2)^{1-\sigma}$, a result of
	Stolz~\cite[\S 8]{Sto88}.\footnote{As $[B\Z/2, B\O/B\Spin]$ has exponent $4$ by \cref{periodicity_of_spin},
	$[1-\sigma] = [3\sigma-3]$, so the reader who prefers to avoid virtual vector bundles can write $\MTPin^+\simeq
	\MTSpin\wedge (B\Z/2)^{3\sigma-3}$.}
	\item A \spinc structure is data of a trivialization of $w_1(TM)$ and a class $c_1\in H^2(M;\Z)$ such that
	$c_1\bmod 2 = w_2(TM)$; in this case there is no difference between $w_2(TM)$ and $w_2(\nu)$. This is a twisted
	spin structure where $X = B\U(1) = K(\Z, 2)$, $a = 0$, and $b$ is the generator of $H^2(K(\Z, 2); \Z/2)
	\cong\Z/2$. As $0$, resp.\ $b$ are the first and second Stiefel-Whitney classes of the tautological complex
	line bundle $L\to B\U(1)$, \cref{shearing_lem} implies $\MTSpin^c\simeq \MTSpin\wedge (B\U(1))^{L-2}$, which is
	known due to Bahri-Gilkey~\cite{BG87a, BG87b}.
	\item A spin-$\Z/2k$ structure on a manifold $M$ is data of a principal $\Z/k$-bundle $P\to M$ together with
	trivializations of $w_1(M)$ and $w_2(M) - w_2(V_P)$, where $V$ is the standard one-dimensional complex
	representation of $\Z/k$ as rotations and $V_P\to M$ is the associated complex line bundle to $P$. Thus,
	analogous to the \spinc argument above, this structure is a twisted spin structure for $X = B\Z/k$, $a = 0$,
	and $b = w_2(V)$, and \cref{shearing_lem} implies $\mathit{MT}\paren{\Spin\text{-}\Z/2k}\simeq \MTSpin\wedge
	(B\Z/k)^{V - 2}$, reproving a theorem of Campbell~\cite[\S 7.9]{Cam17}.
	\item A \spinh structure is data of a trivialization of $w_1(M)$ and a rank-$3$ oriented vector bundle $E\to M$
	and a trivialization of $w_2(M) - w_2(E)$. Again tangential vs.\ normal does not matter here, and one can use
	the same line of reasoning to show that \spinh structures are twisted spin structures for $X = B\SO_3$, $a =
	0$, and $b = w_2$. As these are $w_1$, resp.\ $w_2$ of the tautological vector bundle $V\to B\SO_3$,
	\cref{shearing_lem} tells us $\MTSpin^h\simeq \MTSpin\wedge (B\SO_3)^{V - 3}$, which is due to
	Freed-Hopkins~\cite[\S 10]{FH16}.
\end{enumerate}
There are many more examples of vector bundle twists of spin bordism, including the examples in, e.g.,~\cite{FH16,
Guo18, DL20, GOPWW20, WW20, DL21, Ste21, DDHM22}. But one can find twists of spin bordism not described by vector
bundle twists, even in physically motivated examples: see~\cite[Theorem 4.2]{DY22} for an example where $X =
B\SU_8/\set{\pm 1}$, with a few more examples given in~\cite[\S 3.1]{DY23}. The Smith-theoretic techniques in our
paper do not apply in those situations.
\end{example}
\begin{example}[James periodicity as Smith periodicity]
\label{james_example}
James periodicity~\cite{Jam59} is a classical result in homotopy theory that the homotopy types of the \term{stunted projective spaces} $\RP^n_k\coloneqq \RP^n/\RP^k$ (here $k < n$) are periodic, with periodicity dependent on $n$ and $k$. There are also results for the analogously defined stunted complex and quaternionic projective spaces $\CP_k^n\coloneqq \CP^n/\CP^k$ and $\HP_k^n\coloneqq \HP^n/\HP^k$. These periodicities can be thought of in terms of periodic Smith families for framed bordism---or conversely, the periodicities in the previous several examples can be thought of as generalizations of James periodicity over other ring spectra than $\mathbb S$.

\Cref{cofiber_periodicity} is the engine behind our periodicity results; its key idea is that vector bundles inducing equivalent maps to $B\GL_1(R)$ have equivalent $R$-module Thom spectra. For framed bordism, where $R = \mathbb S$, we therefore should look at the image of the homomorphism $[X, B\O]\to [X, B\GL_1(\mathbb S)]$; following Atiyah~\cite[\S 1]{Ati61a}, this image is typically denoted $J(X)$. Atiyah (\textit{ibid.}, Lemma 2.5) proves that if $V,W\to X$ have equal images in $J(X)$, then $X^V\simeq X^W$.\footnote{See Held-Sjerve~\cite[Theorem 1.2]{HS73} for a partial converse to this result.} Therefore we can obtain framed bordism Smith periodicities, or equivalences of Thom spectra, by calculating the groups $J(X)$. Atiyah (\textit{ibid.}, Proposition 1.5) shows that when $X$ is a finite CW complex, $J(X)$ is a finite group, implying the existence of many framed Smith families.

For James periodicity specifically, choose $F\in\{\R, \C, \mathbb{H}\}$. Stunted projective spaces are Thom spectra: if $L\to F\mathbb P^k$ denotes the tautological (real, complex, or quaternionic) line bundle, there is an equivalence $\Sigma^\infty F\mathbb P_k^n\simeq (F\mathbb P^k)^{(n-k)L}$~\cite[Proposition 4.3]{Ati61a}, reducing the proof of James periodicity to the computation of the order of $L$ in $J(F\mathbb P^k)$. For example, for $F = \R$ Adams calculates the order of $L$ in $J(\RP^k)$ in~\cite[Theorem 7.4]{Ada62} and~\cite[Example 6.3]{AdamsJ2} to be $2^{\phi(k)}$, where $\phi(k)$ is the number of integers $s$ with $0 < s \le k$ and $s\equiv 0$, $1$, $2$, or $4\bmod 8$. Therefore for all $k$ and $n$, there is a homotopy equivalence
\begin{equation}
    \Sigma^\infty\RP_k^{n+2^{\phi(k)}} \overset\simeq\longrightarrow \Sigma^\infty \Sigma^{2^\phi(k)}\RP_k^n.
\end{equation}
(and in fact this is true even before applying $\Sigma^\infty$~\cite{Mah65}). Additional computations in $J$-groups of $F\mathbb P^k$ are done by Adams-Walker~\cite{AW65}, Lam~\cite{Lam72}, Federer-Gitler~\cite{FG73, FG77}, Sigrist~\cite{Sig75}, Walker~\cite{Wal81}, Crabb-Knapp~\cite{CK88}, Dibağ~\cite{Dib99, Dib03}, Obiedat~\cite{Obi01}, and Randal-Williams~\cite[\S 5.3]{RW23}.
\end{example}
\begin{remark}
There are many other tangential structures $\xi$ that one can study twists of.
See~\cite{SSS09, Sat10, Sat11, Sat11b, Sat12, SSS12, Sat15, SW15, SW18, LSW20, SY21, BC22, DY23} for more examples.
\end{remark}
\subsection{The bound in \cref{cofiber_periodicity} is not sharp}
\label{not_sharp}
In this subsection, we discuss ways in which \cref{cofiber_periodicity} loses information, yielding Smith families with lower-than-expected periodicity. This in particular occurs for twisted string structures.

\Cref{cofiber_periodicity} estimates the periodicity of a Smith family for a tangential structure $\xi$ and vector bundle $V\to X$ in terms of the minimal positive integer $k$ such that $V^{\oplus k}$ has a $\xi$-structure. As we have seen above in \cref{unoriented_periodicity,,periodicity_of_SO,,MU_MSpinc,,periodicity_of_spin}, $k$ is finite in many examples of interest, including $\O$-, $\SO$-, $\Spin^c$-, $\Spin$-, and $\U$-structures. However, $k$ is not always finite.
\begin{lemma}\label{nostring_chern}\hfill
\begin{enumerate}
    \item\label{nostring} Let $V\to X$ be a real vector bundle whose rational first Pontrjagin class $p_1(V)\in H^4(X;\mathbb Q)$ is nonzero. Then for $k\ne 0$, $V^{\oplus k}$ does not admit a string structure.
    \item\label{noSU} Let $V\to X$ be a complex vector bundle whose rational first Chern class $c_1(V)\in H^2(X;\mathbb Q)$ is nonzero. Then for $k\ne 0$, $V^{\oplus k}$ does not admit an $\SU$-structure.
\end{enumerate}
\end{lemma}
\begin{proof}
For part~\eqref{nostring}, if $E\to M$ is a string vector bundle, then $\lambda(E) = 0$ implies $p_1(E) = 0\in H^4(M;\Z)$ (since $2\lambda = p_1$), which implies the image of $p_1(E)$ in $H^4(M;\mathbb Q)$ is also $0$, so it suffices to show $p_1(V^{\oplus k})$ has nonzero image in $H^q(X;\mathbb Q)$ for $k\ne 0$. Since $p_1(V)$ is nonzero in this group, it is in particular nontorsion, and the Whitney sum formula implies that in $\mathbb Q$-cohomology $p_1(V^{\oplus k}) = k p_1(V)$, so it is also nonzero. Part~\eqref{noSU} is analogous, using that the complete obstruction for lifting from a $\U$-structure to an $\SU$-structure is $c_1\in H^2(B\U;\Z)$.
\end{proof}
There are many bundles satisfying the hypotheses of \cref{nostring_chern}: for example, both are true for the tautological complex line bundle over $B\U(1)$.
\begin{definition}[{Lashof~\cite[\S 3]{Las63}}]
\label{U6_defn}
Let $B\U\ang 6$ denote the $5$-connected cover of $B\U$, and let $\xi\ang 6\colon B\U\ang 6\to B\O$ be the composition of the covering map $B\U\ang 6\to B\U$ and the map $B\U\to B\O$ forgetting the complex structure. We will refer to $\xi\ang 6$-structures as \term{$U\ang 6$-structures}.
\end{definition}
A $U\ang 6$-structure induces both an $\SU$-structure and a string structure; the former because the $5$-connected covering map always factors through the $3$-connected cover, which for $B\U$ is $B\SU\to B\U$, and the latter because the map $B\U\ang 6\to B\O$ must factor through the $5$-connected cover of $B\O$, which is $B\String$.

One can construct two-out-of-three data for $U\ang 6$-structures using the $5$-connected covers of the maps in the two-out-of-three data for $B\U$. Since this data is constructed in this universal way, it is compatible with the two-out-of-three data we have already used for $\U$, $\SU$, $\String$, etc.
\begin{proposition}[{Bauer~\cite[Lemma 2.1]{Bau03}}]
\label{bauer}
Let $\xi$ be a tangential structure with two-out-of-three data admitting a map $\xi\ang 6\to\xi$ compatible with two-out-of-three data. If $L\to B\U(1)$ denotes the tautological complex line bundle, then there is a homotopy equivalence
\begin{equation}
    \mathit{MT\xi}\wedge (B\U(1))_+ \overset\simeq\longrightarrow \mathit{MT\xi}\wedge (BU(1))^{24L^* - 48}.
\end{equation}
\end{proposition}
Bauer's statement in~\cite[Lemma 2.1]{Bau03} is only a corollary of this, obtained by base-changing from $\MU\ang 6$ to $\mathit{tmf}$, but in his proof he proves the version we provide here. One can also pass from $L^*$ to $L$ by pulling back along the complex conjugation map $B\U(1)\to B\U(1)$.
\begin{corollary}
With $\xi$ as in \cref{bauer}, in particular including string and $\SU$-structures, every Smith family for $\xi$-structures and a complex line bundle is at most $24$-periodic. In particular, \cref{cofiber_periodicity} is not sharp.
\end{corollary}
What went wrong? In this specific case, we didn't use all of the available information in \cref{cofiber_periodicity}. Looking into Bauer's proof, one learns that the map $[B\U(1), B\O/B\String]\to[B\U(1), B\GL_1(\MTString)]$ from \cref{beardsley} sends the class of $L^*$, which is \emph{infinite-order} in $[B\U(1), B\O/B\String]$, to a \emph{finite-order} class in $[B\U(1), B\GL_1(\MTString)]$. Specifically, because this twist came from a vector bundle, it also factors through $[B\U(1), B\GL_1(\mathbb S)/B\String]$ (see \cref{beardsley}), and the image of $L^*$ in this group has finite order.
\begin{proposition}
\label{string_twists_are_finite}
For any space $X$ homotopy equivalent to a CW complex with finitely many cells in each degree, let $J_*\colon [X, B\O]\to [X, B\GL_1(\mathbb S)/B\String]$ denote the map induced by the $J$-homomorphism $B\O\to B\GL_1(\mathbb S)$ followed by taking the cofiber of $B\String\to B\O\overset J\to B\GL_1(\mathbb S)$. Then all classes in $\mathrm{Im}(J_*)$ have finite order.
\end{proposition}
\begin{corollary}\label{twstr_cor}
Every twisted string structure Smith family over a base space $X$ as in \cref{string_twists_are_finite} has finite periodicity.
\end{corollary}
\begin{proof}[Proof of \cref{string_twists_are_finite}]
Because $B\O$ and $B\GL_1(\mathbb S)/B\String$ are grouplike $E_\infty$-spaces,\footnote{It is not immediately obvious that $B\GL_1(\mathbb S)/B\String$ is a grouplike $E_\infty$-space, but one can prove it by modifying the argument in~\cite[Proof of Proposition 1.20]{DY23}.}
one can prove the proposition by lifting the map $B\O\to B\GL_1(\mathbb S)/B\String$ to the equivalent data of a map of spectra $j_{/\mathit{bstring}}\colon \mathit{bo}_0\to \mathit{bg\ell}_1(\mathbb S)/\mathit{bstring}$.\footnote{The name $\mathit{bo}_0$ instead of $\mathit{bo}$ is because it is traditional to use $\mathit{bo}$ to refer to the spectrum corresponding to the grouplike $E_\infty$-space $\Z\times B\O$, i.e.\ the spectrum $\ko$.}

The map $j_{/\mathit{bstring}}$ was induced from a map of grouplike $E_\infty$-spaces that factored through $B\GL_1(\mathbb S)$, and therefore $j_{/\mathit{bstring}}$ factors through $\mathit{bg\ell}_1(\mathbb S)$. The homotopy groups of this spectrum are torsion, which follows from May's definition~\cite[\S III.2]{May77} of $\GL_1(\mathbb S)$ and the triviality of the positive-degree rational stable homotopy groups of the sphere~\cite{Ser53}. Thus the rationalization $\mathit{bg\ell}_1(\mathbb S)\wedge H\mathbb Q\simeq 0$. For any class $x\in (\mathit{bo}_0)^0(X)$, if $J_*(x)$ has infinite order, its image in the rationalized $\mathit{bg\ell}_1(\mathbb S)/\mathit{bstring}$-cohomology of $X$ must be nonzero: because $X$ has finitely many cells in each dimension, its generalized cohomology groups for any finite-type spectrum (including all spectra appearing in this proof) are finitely generated, so infinite-order elements persist through rationalization. Rationally, though, $J_*$ passes through the zero spectrum.
\end{proof}
\begin{remark}
One way to interpret this phenomenon is that, even though the first Pontrjagin class $p_1\colon B\O\to K(\Z, 4)$ descends to a map $p_1\colon B\O/B\String\to K(\Z, 4)$, this map does not extend to a rationally nontrivial map out of $B\GL_1(\mathbb S)/B\String$. In the language of \cite{DY23}, ``a fake vector bundle with respect to twisted string structures has a first Pontrjagin class, but a fake spherical fibration does not.''
\end{remark}
\begin{remark}[$\O\ang n$-families' periodicity and Bernoulli numbers?]
\label{highconn}
\Cref{string_twists_are_finite,twstr_cor} generalize \textit{mutatis mutandis} to tangential structures further up the Whitehead tower of $B\O$. To wit, given a natural number $n$, let $\xi\colon B\O\ang n\to B\O$ be the $(n-1)$-connected covering map. This defines a tangential structure commonly called a \term{$\O\ang n$-structure}, and the two-out-of-three data for $B\O$ pull back by the universal property of the $(n-1)$-connected cover to define two-out-of-three data for $\O\ang n$-structures. (Compare \cref{U6_defn}.) Thus \cref{smfam}, \cref{vb_thom}, and \cref{beardsley} define vector bundle twists, non-vector-bundle twists, and Smith families for twisted $\O\ang n$-structures just like for string structures.

For $n > 4$, the homotopy groups of $B\O/B\O\ang n$ are not all torsion, which is downstream from the isomorphism $\pi_4(B\O)\cong\Z$~\cite[\S24, \S 25]{Ste51}. Therefore, like for $B\O/B\String$, the order of a twist in $[X, B\O/B\O\ang n]$ is not a particularly good estimate for the value of the corresponding Smith family's period. Indeed, \cref{string_twists_are_finite,twstr_cor} and their proofs generalize directly from $B\String$ to $B\O\ang n$, showing all Smith families of twisted $\O\ang n$-structures have finite order, provided they are over spaces homotopy equivalent to CW complexes with finitely many cells in each dimension. One can also generalize this whole story to the limiting case as $n\to\infty$, which is the tangential structure $E\O\to B\O$, i.e.\ a stable framing.

Our proof did not give any estimates on the orders of these Smith families, just finiteness. It would be interesting to bound or compute these orders; for example, one could investigate the map of Atiyah-Hirzebruch spectral sequences induced by $j_{/\mathit{bo}\ang n}$, the generalization of $j_{/\mathit{bstring}}$ in the proof of \cref{string_twists_are_finite}, or generalize Bauer's proof in~\cite[Lemma 2.1]{Bau03}. We suspect that sharp estimates for periodicity for Smith families of twisted $\O\ang n$-structures will have formulas involving Bernoulli numbers, because of their appearance in Adams' seminal work computing the image of the $J$-homomorphism in $\pi_*(\mathbb S)$~\cite{Ada63, AdamsJ2, Ada65a, Ada66}. We would be interested in learning whether this is the case.

For $n\le 16$, twisted $\O\ang n$-structures as defined here recover familiar twists of familiar tangential structures.
\begin{enumerate}
    \setcounter{enumi}{-1}
    \item $\O\ang 0$- and $\O\ang 1$-structures are canonically equivalent to $\O$-structures, i.e.\ no data, and so this story is vacuous.
    \setcounter{enumi}{1}
    \item $\O\ang 2$-structures are equivalent to $\SO$-structures, and this story recovers the twists in \cref{oriented_shearing}.
    \item $\pi_3(B\O) = 0$~\cite[\S IV]{Car36}, and $B\O\ang 3$ and $B\O\ang 4$ are equivalent to $B\Spin$. This story recovers the notion of twisted spin structure we discussed in \cref{twists_of_spin}.
    \setcounter{enumi}{4}
    \item $\pi_k(B\O)$ vanishes for $k=5$~\cite[Remarks 24.11]{Ste51}, $k=6$~\cite[3.72]{Eck51}, and $k=7$~\cite[Proposition 19.5]{BS53}, so $B\O\ang 5$, $B\O\ang 6$, $B\O\ang 7$, and $B\O\ang 8$ all coincide, and are $B\String$. This story recovers the standard story of twists of string bordism that we mentioned in \cref{string_not_periodic}.
    \setcounter{enumi}{8}
    \item Sati-Schreiber-Stasheff~\cite[Definition 1]{SSS09} call an $\O\ang 9$-structure a \term{fivebrane structure}, and in a sequel paper~\cite[\S 2.3]{SSS12}, they introduce twisted fivebrane structures over a space $X$ given by data of a map $X\to K(\Z, 8)$. It is possible to show that their definition is a special case of ours. Specifically, similarly to the identification $K(\Z, 4)\simeq B\Spin/B\String$ producing a map from the $K(\Z, 4)$-twists of string bordism to the more general group of $B\O/B\String$ twists of string bordism~\cite[(1.45)]{DY23}, the characteristic class $\tfrac 16 p_2\colon B\String\to K(\Z, 8)$ induces an equivalence of grouplike $E_\infty$-spaces $B\String/B\O\ang 9\to K(\Z, 8)$, and this equivalence leads to a map $K(\Z, 8)\simeq B\String/B\O\ang 9\to B\O/B\O\ang 9$ carrying Sati-Schreiber-Stasheff's twisted fivebrane structures to a subgroup of the fake vector bundle twists of fivebrane structure.
    \item Sati~\cite[Definition 2.4]{Sat15} calls $\O\ang{10}$-structures ``$2$-orientations.'' and studies their twists in (\textit{ibid.}, Definition 5.1). Like for twisted fivebrane structures, Sati's twists are classified by maps to $K(\Z/2, 9)\simeq B\O\ang{10}/B\O\ang 9$, and map to the twists we considered via the map $B\O\ang{10}/B\O\ang 9\to B\O/B\O\ang 9$.
    \setcounter{enumi}{12}
    \item As $\pi_{11}(B\O) = 0$~\cite{Bot59}, $\O\ang{11}$- and $\O\ang{12}$-structures coincide. Sati refers to this as a ``$2$-spin structure''~\cite[Definition 2.5]{Sat15}, and in (\textit{ibid.}, Definition 5.2) introduces twisted $2$-spin structures corresponding to $K(\Z/2, 10)\simeq B\O\ang{11}/B\O\ang{10}\to B\O/B\O\ang{10}$.
    \item As $\pi_k(B\O) = 0$ for $k = 13$, $14$, and $15$~\cite{Bot59}, $B\O\ang{13}=\dotsb= B\O\ang{16}$. Sati~\cite[Definition 3.1]{Sat15} names this tangential structure a \term{ninebrane structure}, and produces twisted ninebrane structures classified by the fractional Pontrjagin class $(1/240)p_3\colon B\O\ang{12}\to K(\Z, 12)$ (\textit{ibid.}, Definition 5.3). As in the previous cases, the map $K(\Z, 12)\simeq B\O\ang{13}/B\O\ang{11}\to B\O/B\O\ang{11}$ sends Sati's twists to ours.
\end{enumerate}
Passing to the infinite limit, we obtain an interpretation of maps to $B\O/E\O$, i.e.\ to $B\O$, as classifying vector bundle twists of framed bordism. These twists of framed bordism are studied in~\cite{Cru03, FSS24}.
\end{remark}

\section{Examples of Smith fiber sequences}\label{examplesofSmith}
In this section, we implement the discussion from the previous section for some commonly studied vector bundles. We
find many previously studied Smith homomorphisms, and also identify a few other well-known cofiber sequences,
including Wood's sequences, Wall's sequence, and the cofiber sequences associated to the Hopf maps
and to transfer maps, as Smith homomorphisms (\cref{Wall_exm,transfer_exm}). We include this Pokédex of examples in part to illustrate what kinds
of Smith cofiber sequences are out there; in part to make contact with preexisting literature; and in part to
illustrate how to put theorems such as \cref{the_cofiber_sequence} into practice to explicitly write down Smith cofiber
sequences.
\subsection{Twisting by real line bundles}
\label{taut_real_fams}
Our first family of examples use the tautological line bundle $\sigma\to B\Z/2$; its sphere bundle is the
tautological $\Z/2$-bundle $E\Z/2\to B\Z/2$, whose total space is contractible. Therefore by \cref{the_cofiber_sequence}, for any
$k\in\Z$, we have a cofiber sequence
\begin{equation}
\label{smith_Z2}
	\mathbb S\longrightarrow (B\Z/2)^{k(\sigma-1)}\xrightarrow{\sm_\sigma} \Sigma
	(B\Z/2)^{(k+1)(\sigma-1)},
\end{equation}
where $\sm_\sigma$ is the Smith homomorphism associated to $\sigma$. When $k = 0$, this is especially nice: the
middle spectrum is $\Sigma_+^\infty B\Z/2\simeq \mathbb S\vee \Sigma B\Z/2$ and the map $\mathbb S\to\mathbb
S\vee \Sigma B\Z/2$ is the inclusion of the first factor of the wedge sum, leading to a Smith \emph{isomorphism}
$\sm_\sigma\colon \Sigma^\infty B\Z/2\overset\simeq\to (B\Z/2)^\sigma$. This equivalence is well-known; see Kochman~\cite[Lemma 2.6.5]{Koc96} for a proof.
\begin{remark}
The Thom spectrum $(B\Z/2)^{k\sigma}$ is often denoted in the homotopy theory literature by $\RP^\infty_k$, so that $(B\Z/2)^{k(\sigma-1)}$ can be identified with its desuspension $\Sigma^{-k} \RP^\infty_k$. One justification for this notation stems from~\eqref{smith_Z2}: suspending it $k$ times gives a cofiber sequence
\begin{equation}
    \mathbb{S}^k \longrightarrow \RP^\infty_k \longrightarrow \RP^\infty_{k+1},
\end{equation}
which exhibits $\RP^\infty_{k+1}$ as the spectrum obtained by crushing the bottom cell of $\RP^\infty_k$.
\end{remark}
\begin{example}
\label{Z2_MO}
Smash~\eqref{smith_Z2} with $\MTO$. As every virtual bundle has a unique $\MTO$-orientation, this cofiber sequence
simplifies to
\begin{equation}
\label{MO_Smith_map}
	\MTO\longrightarrow \MTO\wedge (B\Z/2)_+ \xrightarrow{\sm_\sigma} \MTO\wedge \Sigma (B\Z/2)_+.
\end{equation}
This was the first Smith homomorphism studied; it was defined and named the Smith homomorphism by
Conner-Floyd~\cite[Theorem 26.1]{CF64}. Thom's celebrated calculation of $\Omega_*^\O$~\cite{ThomThesis} implies that $\MTO$ is a sum of shifts of $H\Z/2$; on each of these copies, the Smith map~\eqref{MO_Smith_map} is the cap product with the
nonzero element of $H^1(B\Z/2;\Z/2)$.

Stong~\cite[Proposition 5]{Sto69} and Uchida~\cite{Uch70} study related examples, where one
smashes~\eqref{MO_Smith_map} with spaces $X$; they identify the fiber $\MTO\wedge X$ and show that the long exact
sequence of homotopy groups splits. Their papers are among the earliest examples identifying the Smith long exact
sequence.\footnote{At the time, it was common to think of $\Omega_*^\O(B\Z/2)$ as the bordism groups of manifolds
$M$ equipped with a free involution $\tau$, rather than manifolds with a principal $\Z/2$-bundle; Stong and
Uchida's results are phrased in that language. To pass between these perspectives, rewrite $(M, \tau)$ as the
principal $\Z/2$-bundle $M\to M/\tau$; in the other direction, take the deck transformation involution of the total
space of a principal $\Z/2$-bundle.}
\end{example}
\begin{example}
\label{SO_Z2_exm}
Smash~\eqref{smith_Z2} with $\MTSO$. Since $\sigma$ is not orientable, but $2\sigma$ is oriented (see \cref{periodicity_of_SO}),
we obtain a $2$-periodic series of codimension-$1$ Smith homomorphisms between the
oriented bordism of $B\Z/2$ and $(B\Z/2, \sigma)$-twisted oriented bordism. The latter can be identified with
unoriented bordism: a $(B\Z/2, \sigma)$-twisted orientation on $V$ is data of a line bundle on $L$ and an
orientation of $V\oplus L$, which is no data at all: this identifies $L\cong\Det(V)^\ast\cong\Det(V)$ up to a
contractible space of choices, and $V\oplus\Det(V)$ is canonically oriented. So every vector bundle has a canonical
$(B\Z/2, \sigma)$-twisted orientation.

Therefore by \cref{the_cofiber_sequence} we obtain a $2$-periodic sequence of codimension-$1$ Smith homomorphisms:
\begin{subequations}
\label{SO(2)_periodic}
\begin{gather}
	\MTSO\longrightarrow \MTSO\wedge (B\Z/2)_+ \xrightarrow{\sm_\sigma} \Sigma\MTO\\
	\MTSO\longrightarrow \MTO \xrightarrow{\sm_\sigma} \Sigma\MTSO\wedge (B\Z/2)_+.
\end{gather}
\end{subequations}
These maps are obtained by taking smooth representatives of Poincaré duals of $w_1$ either of the manifold (when
the domain is $\Omega_\ast^\O$) or of the principal $\Z/2$-bundle (when the domain is $\Omega_\ast^\SO(B\Z/2)$).
See \cite[\S IV.B]{PhysSmith} for the physical interpretation of the corresponding long exact sequence of Anderson dual groups.

These Smith homomorphisms were first introduced by Komiya~\cite[\S 5]{Kom72}; see also Shibata~\cite[Proposition
2.1]{Shi73}. See Córdova-Ohmori-Shao-Yan~\cite[Appendix A]{COSY19}, Hason-Komargodski-Thorngren~\cite[\S
4.4]{HKT19}, and Fidkowski-Haah-Hastings~\cite{FHH20} for applications of these Smith homomorphisms to physics. The
splitting of the $k = 0$ case of~\eqref{smith_Z2} implies a homotopy equivalence $\MTSO\wedge B\Z/2\xrightarrow{\cong}
\Sigma\MTO$, a theorem of Atiyah~\cite[Proposition 4.1]{Ati61}.
\end{example}
\begin{example}
\label{spin_4periodic}
Some of the coolest examples of this kind come about by smashing~\eqref{smith_Z2} with $\MTSpin$. As we discussed in \cref{periodicity_of_spin}, the periodicity of this family is $1$, $2$, or $4$; a Whitney sum
formula calculation shows that $k\sigma$ is spin iff $k$ is a multiple of $4$, and therefore this Smith family is $4$-periodic. The corresponding $(B\Z/2, k\sigma)$-twisted spin
bordism groups can be identified with $H$-bordism for certain Lie groups $H$, as discussed in \cref{twists_of_spin}; specifically,
\begin{enumerate}
	\item a $(B\Z/2, \sigma)$-twisted spin structure is equivalent to a \pinm structure;
	\item a $(B\Z/2, 2\sigma)$-twisted spin structure is equivalent to an $H$ structure, where $H =
	\Spin\times_{\set{\pm 1}} \Z/4$; and
	\item a $(B\Z/2, 3\sigma)$-twisted spin structure is equivalent to a \pinp structure.
\end{enumerate}
Using \cref{the_cofiber_sequence} once again, the $4$-periodic sequence of codimension-$1$ Smith homomorphisms takes the form
\begin{subequations}
\label{spin_Z2_Smith}
\begin{gather}
    \label{spin_pinm}
	\MTSpin\longrightarrow \MTSpin\wedge (B\Z/2)_+ \xrightarrow{\sm_\sigma} \Sigma\MTPin^-\\
        \label{pinm_spinZ4}
	\MTSpin\longrightarrow \MTPin^- \xrightarrow{\sm_\sigma} \Sigma \mathit{MT}(\Spin\times_{\set{\pm
	1}}\Z/4)\\
	\label{pinp_spinc2}
	\MTSpin\longrightarrow \mathit{MT}(\Spin\times_{\set{\pm 1}}\Z/4) \xrightarrow{\sm_\sigma} \Sigma
	\MTPin^+\\
	\label{pinp_spin_z2}
	\MTSpin\longrightarrow \MTPin^+\xrightarrow{\sm_\sigma} \Sigma \MTSpin\wedge (B\Z/2)_+,
\end{gather}
\end{subequations}
with each $\sm_\sigma$ obtained by taking a smooth representative of a Poincaré dual of $w_1$ of the manifold or of a
associated principal $\Z/2$-bundle, like in~\eqref{SO(2)_periodic}.

The splitting of the $k = 0$ Smith homomorphism
in~\eqref{smith_Z2} gives us an equivalence $\MTSpin\wedge B\Z/2\simeq\MTPin^-$, a theorem of Peterson~\cite[\S
7]{Pet68}.

This family of Smith homomorphisms has been discussed in the literature before. The piece involving
$\Spin\times\Z/2$ and $\Pin^-$ was used by Peterson~\cite[\S 7]{Pet68} and
Anderson-Brown-Peterson~\cite{ABP69}, who say that it was already ``well-known.'' The long exact sequence
corresponding to~\eqref{pinp_spinc2} appears in~\cite[Theorem 3.1]{Gia73}, where it is attributed to Stong. The
Smith homomorphism $\sm_\sigma$ in~\eqref{pinp_spin_z2} appears in Kreck~\cite[\S 4]{Kre84}. The long exact sequence induced by~\eqref{pinm_spinZ4} is used by Botvinnik-Rosenberg~\cite[\S 2]{BR23}, who also discuss~\eqref{pinp_spinc2} and~\eqref{pinp_spin_z2}. The composition of two
maps in~\eqref{spin_Z2_Smith} in a row to go between \pinp and \pinm bordism appears in Kirby-Taylor~\cite[Lemma
7]{KT90pinp}. We work out the corresponding long exact sequences, as well as some physical consequences for all four Smith homomorphisms in \cite[\S IV.C]{PhysSmith}.

The full family appears more recently in work of Hambleton-Su~\cite[\S 4.C]{HS13},
Kapustin-Thorngren-Turzillo-Wang~\cite[\S
8]{KTTW15}, Tachikawa-Yonekura~\cite[\S 3.1]{TY19}, Hason-Komargodski-Thorngren~\cite[\S 4.4]{HKT19}, and
Wan-Wang-Zheng~\cite[\S 6.7]{WWZ19}. Ekholm~\cite{Ekh98} produces the $4$-periodic sequence of tangential structures in a different setting but does not discuss the Smith homomorphism.
\end{example}
\begin{example}
\label{string_Z2}
As we discussed in \cref{string_not_periodic}, for a general vector bundle $V\to X$, there is no guarantee that $kV$ has a string
structure. However, on $B\Z/2$, $k\sigma$ has a string structure iff $k\equiv 0\bmod 8$, so there is an
eight-periodic family of codimension-$1$ Smith homomorphisms between bordism groups of manifolds with $(B\Z/2,
k\sigma)$-twisted string structures for various $k$.\footnote{\label{w4_foot}To prove the claimed fact about string structures on
$k\sigma$, first use the Whitney sum formula to show that $w_1(k\sigma)$, $w_2(k\sigma)$, and $w_4(k\sigma)$ all
vanish iff $k \equiv 0\bmod 8$. The reduction mod $2$ map $H^4(B\Z/2;\Z)\to H^4(B\Z/2;\Z/2)$ is an isomorphism, so
the string obstruction $\lambda(k\sigma)$ vanishes iff its mod $2$ reduction does, and $\lambda\bmod 2 = w_4$.}

In \cref{spin_4periodic}, each $(B\Z/2, k\sigma)$-twisted spin structure for $k\in\Z/4$ turned out to be equivalent to a $G[k]$-structure for
four Lie groups $G[k]$: $G[0]\cong \Spin\times\Z/2$, $G[1]\cong\Pin^-$, $G[2]\cong \Spin\times_{\set{\pm 1}}\Z/4$, and $G[3]\cong\Pin^+$ (see also~\cite[(3.1)]{TY19}). An analogous result is true for the eight twisted string structures, but in the world of $2$-groups, because the string group is
a Lie $2$-group~\cite{SP11}. We will see that for each $k\in\Z/8$, there is a Lie $2$-group
$\mathbb G[k]$ and a map $\xi\colon B\mathbb G[k]\to B\O$ such that $\mathbb G[k]$-structures on a smooth manifolds
are naturally equivalent to $(B\Z/2, k\sigma)$-twisted string structures. Though it is difficult to describe $\mathbb G[k]$ explicitly for most $k$, we will realize it as an extension of $G[k\bmod 4]$ by $B\T$: such extensions of a
compact Lie group $G$ by $B\T$ are classified by $H^4(BG;\Z)$~\cite{SP11, Wei22}, and we will find classes in this cohomology group classifying all eight $\mathbb G[k]$ $2$-groups.

Recall from \cref{twists_of_spin} that if $V\to X$ is a pin\textsuperscript{$\pm$} vector bundle, then $V\mp\Det(V)$ has a canonical spin structure, meaning that the string obstruction class $\lambda(V\mp\Det(V))\in H^4(X;\Z)$ is defined. Let
\begin{equation}
    \lambda^\pm \coloneqq \lambda(V \mp\Det(V))\in H^4(B\Pin^\pm;\Z).
\end{equation}
The inclusion $j\colon \Z/2\hookrightarrow\T$ pulls back the tautological complex line bundle $L\to B\T$ to $2\sigma\to B\Z/2$. Therefore a $(B\Z/2, 2\sigma)$-twisted spin structure induces a $(B\T, L)$-twisted spin structure---or in the language of \cref{twists_of_spin}, a $\Spin\times_{\set{\pm 1}}\Z/4$ structure induces a \spinc structure via $j$. Let $\lambda^c\in H^4(B\Spin^c;\Z)$ be the class defined in~\cite[Definition 2.6]{DY24},\footnote{The class $\lambda^c$ is called $q_2$ in~\cite{Dua18} and $\widehat p_2$ in~\cite[\S 2.7]{CN19}. In the notation of~\cite[(3.10)]{CY20}, $\lambda^c = c_1(L)^2 - p_c$, and in~\cite[Construction 2]{Dev22}, $\lambda^c = c_1(L)^2 -(c_1^2-p_1)/2$, where $L$ is the determinant line bundle.}
namely $\lambda(V\oplus L)$, where $L$ is the determinant bundle of the \spinc structure. Pulling back by $j$, we will regard $\lambda^c$ as a class in $H^4(B(\Spin\times_{\set{\pm 1}}\Z/4);\Z/2)$.

Thus, essentially by definition, for $k = -1,0,1,2\bmod 8$, the obstruction to lifting from a $G[k\bmod 4]$-structure to a $\mathbb G[k]$-structure is respectively $\lambda^+(V)$, $\lambda(V)$, $\lambda^-(V)$, and $\lambda^c(V)$, where $V\to BG[k]$ is the tautological bundle. This gives the first half of \cref{string_twists_table}.

Let $e^\Z\in H^1(B\Z/2;\Z_{w_1(\sigma)})$ denote the twisted $\Z$-cohomology Euler class as defined in \cref{twisted_Euler_class} for $\mathcal R = H\Z$. Čadek~\cite[Lemma 1]{Cad99} shows that $(e^\Z)^{2\ell}$ is an untwisted class, and moreover is the unique nonzero element of $H^{2\ell}(B\Z/2;\Z)$. Thus, since $4\sigma\to B\Z/2$ is spin but not string (see Footnote~\ref{w4_foot}), the class $\lambda(4\sigma)$ is defined and equals $e^\Z(\sigma)^4$. The string obstruction $\lambda$ satisfies a Whitney sum formula~\cite[Lemma 1.6]{Deb23} for pairs of spin vector bundles,\footnote{See also Johnson-Freyd and Treumann~\cite[\S 1.4]{JFT20} for a proof sketch and Jenquin~\cite[Corollary 4.9]{Jen05} for a similar result in a generalized cohomology theory.} so for $k = 3,4,5,6\bmod 8$,
\begin{equation}
    \lambda(V + k\sigma) = \lambda(V + (k-4)\sigma) + \lambda(4\sigma) = \lambda(V + (k-4)\sigma) + e^\Z(\sigma)^4,
\end{equation}
so by adding $e^\Z(\sigma)^4$ to the obstruction classes for the first four entries in \cref{string_twists_table}, we get the obstruction classes for the remaining entries.

In some cases it is possible to describe $\mathbb G[k]$ more explicitly. For example, when $k \equiv 0\bmod 8$, we get $\mathbb G[0] = \String\times\Z/2$, as the obstruction class is trivial on the $\Z/2$ factor. When $k\equiv 4\bmod 8$, $\mathbb G[4] = \String\times_{B\T} \cat{sLine}$, where $\cat{sLine}$ is the abelian Lie $2$-group of Hermitian super lines. This can be proven in a similar manner to~\cite[Corollary 10.23]{DDHM22}, which establishes an analogous result for twisted spin structures. Here we use that $\cat{sLine}$ is the extension of $\Z/2$ by $B\T$ classified by $e^\Z(\sigma)^4$; since $H^4(B\Z/2;\Z)\cong\Z/2$, this follows as soon as we know that $\cat{sLine}$ is a nonsplit central extension of $\Z/2$, which is straightforward.
\begin{table}[h!]
    \begin{tabular}{c c c}
    \toprule
        $k\bmod 8$ & twisted spin structure & twisted string structure obstruction class\\
    \midrule
    $-1$ & $\Pin^+$ & $\lambda^+(V)$\\
    $\phantom{-}0$ & $\Spin\times \Z/2$ & $\lambda(V)$\\
    $\phantom{-}1$ & $\Pin^-$ & $\lambda^-(V)$\\
    $\phantom{-}2$ & $\Spin\times_{\set{\pm 1}}\Z/4$ & $\lambda^c(V)$\\
    $\phantom{-}3$ &$\Pin^+$ & $\lambda^+(V) + e^\Z(\sigma)^4$\\
    $\phantom{-}4$ &$\Spin\times\Z/2$ & $\lambda(V) + e^\Z(\sigma)^4$\\
    $\phantom{-}5$ &$\Pin^-$ & $\lambda^-(V) + e^\Z(\sigma)^4$\\
    $\phantom{-}6$ &$\Spin\times_{\set{\pm 1}}\Z/4$ & $\lambda^c(V) + e^\Z(\sigma)^4$\\
    \bottomrule
    \end{tabular}
\caption{Summary of the eight twisted string structures over $B\Z/2$, as described in \cref{string_Z2}. Given $k\in\Z$, the middle column gives the group $G$ such that a $(B\Z/2, k\sigma)$-twisted spin structure is a $G$-structure, as in \cref{twists_of_spin}. The rightmost column gives the obstruction to lifting from a $(B\Z/2, k\sigma)$-twisted spin structure to a $(B\Z/2, k\sigma)$-twisted string structure as a class in $H^4(BG;\Z)$. Here $V$ is the tauological bundle pulled back from $B\O$ and $\sigma$ is the tautological bundle pulled back from $B\Z/2$. Notation for the characteristic classes is described in \cref{string_Z2}. As $H^4(BG;\Z)$ classes correspond to $2$-group central extensions by $B\T$, this implicitly describes a $2$-group $\mathbb G[k]$ such that a $\mathbb G[k]$-structure is a $(B\Z/2, k\sigma)$-twisted string structure.}
\label{string_twists_table}
\end{table}
\end{example}
\begin{example}
\label{spinc_Z2_exm}
If one smashes~\eqref{smith_Z2} with $\MTSpin^c$, one obtains a very similar story to \cref{SO_Z2_exm}: twice any
vector bundle is complex, hence \spinc, and $(B\Z/2, \sigma)$-twisted \spinc bordism is naturally identified with
\pinc bordism, as we discussed in \cref{spinc_shearing}. So taking Poincaré duals of $w_1$ as in \cref{SO_Z2_exm} defines a $2$-periodic sequence of
codimension-$1$ Smith homomorphisms
\begin{subequations}
\begin{gather}
	\MTSpin^c\longrightarrow \MTSpin^c\wedge (B\Z/2)_+\xrightarrow{\sm_\sigma} \Sigma\MTPin^c\\
	\MTSpin^c\longrightarrow \MTPin^c\xrightarrow{\sm_\sigma} \Sigma\MTSpin^c\wedge
	(B\Z/2)_+.
\end{gather}
\end{subequations}
To our knowledge, these long exact sequences first appear in Hambleton-Su~\cite[\S 4.C]{HS13}.

We also obtain an equivalence $\MTSpinc\wedge B\Z/2\overset\simeq\to \Sigma\MTPin^c$, which was first observed by
Bahri-Gilkey~\cite[\S 3]{BG87a}. See Shiozaki-Shapourian-Ryu~\cite[\S E.1]{SSR17b} and Kobayashi~\cite[\S IV]{Kob21} for applications in
condensed-matter physics and~\cite{DYY} for an application of a closely related Smith long exact sequence.
\end{example}
\begin{example}
\label{Wall_exm}
Pull back~\eqref{smith_Z2} along the map $B\Z\to B\Z/2$, i.e.\ $S^1 = \RP^1\hookrightarrow\RP^\infty$. The sphere
bundle of $\sigma\to\RP^1$ is not contractible: it is the double cover $S^1\to\RP^1$, and its Thom space is
$\RP^2$. Therefore we obtain from \cref{the_cofiber_sequence} a cofiber sequence $\Sigma_+^\infty S^1\to \Sigma^\infty\RP^2\to
\Sigma_+^{1+\infty}\RP^1$, which is a rotated version of the multiplication-by-$2$ cofiber sequence
\begin{equation}
\label{smith_BZ}
	\mathbb S\overset{2}{\longrightarrow} \mathbb S\longrightarrow \Sigma^{-1+\infty}\RP^2.
\end{equation}
The same story applies to the complex, quaternionic, and octonionic Hopf fibrations: their cofibers are the
respective projective planes $\Sigma^{-2+\infty}\CP^2$, $\Sigma^{-4+\infty}\HP^2$, and
$\Sigma^{-8+\infty}\mathbb{OP}^2$, and in each case the map to the cofiber is a Smith homomorphism for the
tautological line bundle over the respective projective line (which is a sphere). In the case of the complex Hopf
fibration, after smashing with $\ko$ or $\KO$, one obtains the Wood cofiber sequences~\cite{Woo63}
$\Sigma\KO\xrightarrow{\eta} \KO\to\KU$ and $\Sigma\ko \xrightarrow{\eta} \ko\to\ku$ as rotated versions of Smith
cofiber sequences.

Smash~\eqref{smith_BZ} with $\MTSO$ and you obtain Wall's cofiber sequence~\cite[Theorem 3]{Wal60}
\begin{equation}
\label{Wall_seq}
	\MTSO\overset{2}{\longrightarrow} \MTSO\longrightarrow \mathcal W,
\end{equation}
where $\mathcal W$ is the Thom spectrum whose homotopy groups are the bordism groups of manifolds with an integral
lift of $w_1$. This follows from Atiyah's identification of $\mathcal W\simeq
\Sigma^{-1}\MTSO\wedge\RP^2$~\cite[\S 4]{Ati61}, but it is also easy to directly check that an integral lift of
$w_1$ is equivalent data to a $(\RP^1, \sigma)$-twisted orientation, using that $\RP^1$ is a $B\Z$.

It is also interesting to smash~\eqref{smith_BZ} with $\MTSpin$; we work out the induced long exact sequence of bordism groups in low degrees in \cref{fig:spintimes2LES}, and this long exact sequence also appears in~\cite{DYY}.
\end{example}
\begin{example}
\label{transfer_exm}
Let $\pi\colon E\to B$ be a principal $\Z/2$-bundle and $L\coloneqq E\times_{\Z/2}\R\to B$ be the associated line
bundle.  Then we have a Smith homomorphism $\sm_L\colon B^{-L}\to \Sigma_+^\infty B$. The fiber is the Thom
spectrum of the pullback of $L$ to its sphere bundle; the sphere bundle is $E$ and $\pi^\ast(L)$ is trivial, so \cref{the_cofiber_sequence} gives us a cofiber sequence
\begin{equation}
\label{transfer_Smith}
	B^{-L}\xrightarrow{\sm_L} \Sigma_+^\infty B\overset{\tau}{\longrightarrow} \Sigma_+^\infty E.
\end{equation}
\begin{lemma}
\label{transfer_Smith_lem}
The map $\tau$ in~\eqref{transfer_Smith} is the Becker-Gottlieb transfer~\cite{Rou72, KP72, BG75} for $\pi$.
\end{lemma}
\begin{proof}
It suffices to work universally with the
Smith cofiber sequence $(B\Z/2)^{-\sigma}\to \Sigma_+^\infty B\Z/2\to \Sigma_+^\infty E\Z/2$, i.e.\
$(\RP^\infty)^{-\sigma}\to \Sigma_+^\infty\RP^\infty\to\mathbb S$, and to show that the latter map is the transfer
for $E\Z/2\to B\Z/2$.

This transfer map admits the following description: consider the map of $\Z/2$-spectra\footnote{This fact, and our
argument using it, works for both Borel and genuine $\Z/2$-spectra.} $f\colon \mathbb S\to
\Sigma^{1-\sigma}(\Z/2)_+$, whose cofiber is $\mathbb S^{-\sigma}$. Upon taking homotopy orbits, we obtain a map
$f_{h\Z/2}\colon \Sigma_+^\infty \RP^\infty\to \mathbb S$, and this is the transfer map.

If $G$ is a finite group and $V\in\mathit{RO}(G)$, there is a natural equivalence of spectra $(\mathbb
S^V)_{hG}\simeq (BG)^V$.\footnote{One quick way to prove this uses the Ando-Blumberg-Gepner-Hopkins-Rezk approach
to Thom spectra~\cite{ABGHR14a, ABGHR14b}: both $(\mathbb S^V)_{hG}$ and $(BG)^V$ are both the colimit of the
$\mathrm{pt}/G$-shaped diagram whose value on $\mathrm{pt}$ is $\mathbb S^V$ and whose value on the morphism set
$G$ encodes the $G$-action on $\mathbb S^V$~\cite[Theorem 1.17]{ABGHR14a}. It is also possible to prove this more classically by working with
Thom spaces.} And taking homotopy orbits of $G$-spectra preserves cofiber sequences, so the fiber of the transfer
$f_{h\Z/2}$ is the map $(\RP^\infty)^{-\sigma}\to \Sigma_+^\infty \RP^\infty$ given by the ``inclusion'' of virtual
representations $-\sigma\hookrightarrow 0$, which is the Smith homomorphism we began with.
\end{proof}
In the case that $B$ is a finite CW complex, one can prove \cref{transfer_Smith_lem} more classically by adapting
Cusick's calculation~\cite[Corollary 2.11]{Cus85} identifying the cofibers of transfer maps for double covers.
\end{example}

\begin{remark}
For another example along the lines of~\eqref{transfer_Smith}, Morisugi~\cite[Theorem 1.3]{Mor09} shows that the
cofibers of certain Smith homomorphisms over compact Lie groups can be described as Becker-Schultz transfer
maps~\cite[\S 4]{BS74}. And Uchida~\cite{Uch69}, motivated by the study of immersions, works out the Smith long
exact sequences of a few special cases of \cref{transfer_exm}, where $E = B\O(k)\times B\O(k)$ and $B = B(\O(1)\ltimes
(\O(k)^{\times 2}))$, where $\O(1)$ acts on $\O(k)^{\times 2}$ by swapping the two factors.
\end{remark}
\begin{remark}
The ease of modifying the Smith long exact sequence by a vector bundle twist suggests that \cref{transfer_exm} could be generalized to some sort of twisted transfer map. The relevant twisted transfer maps have been constructed by Kashiwabara-Zare~\cite{KZ18}.
\end{remark}
\subsection{Twisting by complex line bundles}\label{CP_inf_twist}
Now we consider the analogous family of examples arising from the tautological complex line bundle $L\to B\T$. Its
sphere bundle is $E\T\to B\T$, which is contractible, so just like in~\eqref{smith_Z2}, we have for any $k\in\Z$ a
cofiber sequence
\begin{equation}
\label{smith_T}
	\mathbb S\longrightarrow (B\T)^{k(L-2)}\xrightarrow{\sm_L} \Sigma^2 (B\T)^{(k+1)(L-2)}.
\end{equation}
Again, when $k = 0$, this sequence splits, yielding another Smith isomorphism $\Sigma^\infty B\T\xrightarrow{\cong}
(B\T)^L$. This equivalence is well-known, e.g.~\cite[Example 2.1]{Ada74}.
\begin{example}
\label{cpx_triv_exms}
Let $G$ be one of $\O$, $\SO$, $\Spin^c$, or $\U$; then the tautological line bundle over $B\T$ has a
$G$-structure, and $\mathit{MTG}$ is an $E_\infty$-ring spectrum and we can make sense of $G$-orientations. The
$G$-orientation on $L$ untwists the Thom spectrum, so smashing~\eqref{smith_T} with $\mathit{MTG}$ has a similar
effect to \cref{Z2_MO}: the result is a cofiber sequence
\begin{equation}
	\mathit{MTG}\longrightarrow \mathit{MTG}\wedge (B\T)_+\xrightarrow{\sm_L} \Sigma^2
	\mathit{MTG}\wedge (B\T)_+.
\end{equation}
For $G = \U$, this Smith homomorphism was first studied by Conner-Floyd~\cite[\S 5]{CF66}, and also appears in Ray~\cite[\S 3]{Ray72}.
\begin{lemma}
\label{U_smith_split}
For $G = \O$, $\SO$, $\Spin^c$, or $\U$, there is an $\mathit{MTG}$-module equivalence
\begin{equation}
	\mathit{MTG}\wedge (B\T)_+ \simeq \bigvee_{k\ge 0} \Sigma^{2k}\mathit{MTG}.
\end{equation}
\end{lemma}
\begin{proof}
The zeroth step is splitting off the basepoint: $\mathit{MTG}\wedge (B\T)_+\simeq \mathit{MTG}\vee
\mathit{MTG}\wedge (B\T)$. As noted above, $\Sigma^\infty B\T\simeq (B\T)^L$, and we have a Thom isomorphism
$\mathit{MTG}\wedge (B\T)^L\simeq \mathit{MTG}\wedge \Sigma^2 (B\T)_+$ of $\mathit{MTG}$-modules. We are now in the same situation as at the beginning of the proof, but shifted up by $2$, and we carry on in a similar way.
\end{proof}
\end{example}
\begin{example}
\label{spinc_spin_Smith}
Smash~\eqref{smith_T} with $\MTSpin$; the bundle $L\to B\T$ is oriented but not spin, so $2L$ is spin, and
therefore we obtain a $2$-periodic, codimension-$2$ family of Smith homomorphisms between the spin bordism of $B\T$
and $(B\T, L)$-twisted spin bordism. A $(B\T, L)$-twisted spin structure is equivalent data to a \spinc structure, as we discussed in \cref{twists_of_spin},
so this Smith family takes the form
\begin{subequations}
\label{spinc_both}
\begin{gather}
\label{first_spinc}
	\MTSpin\longrightarrow \MTSpin^c\xrightarrow{\sm_L} \Sigma^2 \MTSpin\wedge (B\T)_+\\
	\MTSpin\longrightarrow \MTSpin\wedge (B\T)_+\xrightarrow{\sm_L} \Sigma^2 \MTSpin^c.
\end{gather}
\end{subequations}
The long exact sequence arising from~\eqref{first_spinc} was identified by Kirby-Taylor~\cite[Corollary 6.12,
Remark 6.14]{KT90}. The splitting of~\eqref{smith_T} when $k = 0$ leads to an equivalence $\MTSpin\wedge
B\T\simeq\Sigma^2\MTSpin^c$, a theorem due to Stong~\cite[Chapter XI]{Sto68}. We discuss the physical interpretation of \eqref{spinc_both} in \cite[\S IV.A]{PhysSmith}.
\end{example}

It would be interesting to study analogues of this example for \pinc or pin\textsuperscript{$\tilde c\pm$} bordism
and applications to invertible phases. Kirby-Taylor~\cite[Remark 6.15]{KT90} consider two additional analogues
of~\eqref{first_spinc}, one of which is related to a Smith long exact sequence for $G$-bordism where
$G\coloneqq\Spin\times_{\set{\pm 1}}\O(2)$ (see~\cite[\S 3.3]{DYY25}). Stehouwer~\cite[\S 4]{Ste21} computes $G$-bordism groups in low dimensions, and $G$-bordism also appears in~\cite{DDHM21, DDHM22, DYY}. In addition, Hambleton-Kreck-Teichner~\cite[\S 2]{HKT94} study a \pinm and \pinc analogue of~\cref{spinc_spin_Smith}.
\begin{example}\label{spinzn}
Pull back~\eqref{smith_T} along the inclusion $\Z/k\hookrightarrow\T$, giving us Smith homomorphisms
$(B\Z/k)^{k(L-2)}\to \Sigma^2 (B\Z/k)^{(k+1)(L-2)}$, where $L$ is the complex line bundle induced by the rotation
representation of $\Z/k$ on $\C$. Recall from \cref{the_cofiber_sequence} the fiber sequence
\begin{equation}
	S(V_2)^{V_1} \to X^{V_1} \to X^{V_1\oplus V_2}.
\end{equation}
For this example, we start with $X=B\Z/n$, $V_2=i^*L-2$ (for $L$ as in the previous example and $2$ the trivial
complex line bundle), and $V_1 = k(i^*L-2)$. We can compute the sphere bundle $S(i^*L)$ by fitting it into a
pullback square:
\begin{equation}
\begin{gathered}
    \begin{tikzcd}
    S(i^*L) \arrow[r,] \arrow[d, "p"] \arrow[dr, phantom, "\usebox\pullback" , very near start, color=black]
    & S(L)\simeq * \arrow[d]\\
    B\Z/n \arrow[r]  & B\T.
\end{tikzcd}
\end{gathered}
\end{equation}
As noted above, $S(L)$ is contractible as it is the total space of the universal fibration. Therefore, the other three corners of the square form a fiber sequence. To compute the fiber of $B\Z/n\to B\T$, we notice that applying the classifying space functor to the short exact sequence $\Z/n\hookrightarrow \T \xrightarrow{\times n} \T$ gives a fibration $B\Z/n\to B\T\to B\T$. Then, recognizing the map $B\T\to B\T$ as the classifying map for a principal $\T$-bundle over $\T$ with total space $B\Z/n$, we conclude that the fiber of the map $B\Z/n\to B\T$ is exactly $\T$. So, $S(i^*L)\simeq S^1$.

Next, we need to pull back $V_1$ along the projection $p\colon S(i^*L)\to B\Z/n$. We have that $p^*(k(i^*L-1))\cong \bigoplus_k p^*(i^*L)$. Since $L$ is oriented as a real vector bundle, its pullbacks are as well, so $p^*i^*L$ is oriented when considered as a real vector bundle over $S^1$, and thus it is the trivial 2-plane bundle.

Therefore, we recognize the Thom spectrum $S(i^*L)^{kp^*(i^*L)}$ as
\begin{align*}
    S(i^*L)^{kp^*(i^*L)} &\simeq (S^1)^{k} \\
    &\simeq \Sigma^{2k} \textnormal{Th}(S^1;0) \\
    &\simeq \Sigma^{2k}(\Sigma_+^\infty S^1) \\
    &\simeq \Sigma^{2k}(\Sigma^\infty S^1\oplus \Sigma^\infty S^0) \\
    &\simeq \Sigma^{2k+1}\mathbb{S} \vee \Sigma^{2k}\mathbb{S}.
\end{align*}
Thus for each $k\ge 0$ we have a Smith cofiber sequence
\begin{equation}
	\Sigma^{2k+1}\mathbb{S} \vee \Sigma^{2k}\mathbb{S} \longrightarrow (B\Z/n)^{k\cdot i^*L} \longrightarrow
	B\Z/n^{(k+1)i^*L}.
\end{equation}
Finally, we place $V_1$ in virtual dimension zero by taking $V_1=k(i^*L-2)$, to be consistent with the other
examples in this section, and obtain the cofiber sequence
\begin{equation}
\label{Zk_smith}
	\Sigma\mathbb{S}\vee \mathbb{S} \longrightarrow (B\Z/n)^{k(i^*L-2)} \xrightarrow{\sm_L} \Sigma^2
	(B\Z/n)^{(k+1)(i^*L-2)}.
\end{equation}
\end{example}
\begin{example}
\label{Zk_spin}
Smash~\eqref{Zk_smith} with $\MTSpin$. Like in \cref{spinc_spin_Smith}, $i^*L$ is oriented but not spin, and
$2i^*L$ is spin, so we obtain a $2$-periodic, codimension-$2$ family of Smith homomorphisms between the spin
bordism of $B\Z/n$ and $(B\Z/n, i^*L)$-twisted spin bordism. Campbell~\cite[\S 7.9]{Cam17} identifies the latter as
bordism for the tangential structure $\Spin\times_{\set{\pm 1}}\Z/2n$, explicitly giving us Smith cofiber sequences
\begin{subequations}
\label{spinzk_Smith}
\begin{gather}
	\Sigma\MTSpin\vee \MTSpin\longrightarrow \mathit{MT}(\Spin\times_{\set{\pm 1}}\Z/2k)
	\xrightarrow{\sm_{i^*L}} \Sigma^2 \MTSpin\wedge (B\Z/k)_+\\
	\Sigma\MTSpin\vee \MTSpin\longrightarrow \MTSpin\wedge (B\Z/k)_+ \xrightarrow{\sm_{i^*L}}
	\mathit{MT}(\Spin\times_{\set{\pm 1}}\Z/2k).
\end{gather}
\end{subequations}
$\Spin\times_{\set{\pm 1}}\Z/2k$ bordism appears in the mathematical physics literature in~\cite{Bel99, Bla00, BBC17, Cam17, GEM18,
Hsi18, Jan18, Li19, GOPWW20, DDHM21, Deb21, DL21, DDHM22, HTY22, DYY}; the case $k = 2$ also appears in~\cite{Gia73a, TY19, FH19, HKT19, MV21}. 
The Smith homomorphisms in~\eqref{spinzk_Smith} for $n = 4$ appear
in~\cite{DDHM22}.  
We work out the Anderson-dualized long exact sequences corresponding to~\eqref{spinzk_Smith} for the $n=3$ case in \cite[\S IV.D]{PhysSmith}, and for the $n=4$ case in \cite[\S IV.E]{PhysSmith}.
\end{example}
\begin{example}
\label{another_Wall_exm}
We elaborate on \cref{Zk_spin} when $n = 2$. The rotation representation is isomorphic to $2\sigma$, where
$\sigma$ denotes the real sign representation; we will also let $\sigma$ denote the associated bundle over $B\Z/2$.

Everything in \cref{Zk_spin} still works for $n = 2$, but now we have more options: we can start with an odd number
of copies of $\sigma$. In this case, the fiber of the Smith map is the Thom spectrum of the Möbius bundle
$(\sigma-1)\to\T$; one can directly check that the Thom space of $\sigma$ is $\RP^2$, so the Thom spectrum of
$\sigma-1$ is $\Sigma^{-1+\infty}\RP^2$. Therefore we have a Smith cofiber sequence
\begin{equation}
\label{odd_sigma}
	\Sigma^{-1+\infty}\RP^2 \longrightarrow (B\Z/2)^{(2k-1)(\sigma-1)} \xrightarrow{\sm_{2\sigma}}
	\Sigma^2 (B\Z/2)^{(2k+1)(\sigma-1)}.
\end{equation}
Out of all the examples we have studied in this section, this is the first one where the pullback of $V$ to the
sphere bundle of $W$ is nontrivial.

As usual, we smash~\eqref{odd_sigma} with various bordism spectra. The map $\sm_{2\sigma}$ is the composition of two
iterations of $\sm_\sigma$ from~\eqref{smith_Z2}, so some of the resulting cofiber sequences look familiar from that
perspective. We only discuss a few examples, but plenty more are out there.
\begin{itemize}
	\item If we smash~\eqref{odd_sigma} with $\MTSO$, we obtain a cofiber
	sequence first discussed by Atiyah~\cite[(4.3)]{Ati61}:
	\begin{equation}
	\label{other_Wall_seq}
		\mathcal W\longrightarrow \MTO \xrightarrow{\sm_{2\sigma}} \Sigma^2\MTO,
	\end{equation}
	where $\mathcal W$ is Wall's bordism spectrum (see \cref{Wall_exm}).  Here we use the
	identifications $\Sigma \MTO\simeq\MTSO\wedge B\Z/2$ and $\mathcal W\simeq\MTSO\wedge \Sigma^{-1}\RP^2$, both
	due to Atiyah~\cite[\S 4]{Ati61}, that we discussed in \cref{Wall_exm,SO_Z2_exm}, respectively.
	\item If we instead smash~\eqref{odd_sigma} with $\MTSpin$, we obtain a cofiber sequence
	\begin{equation}
	\label{pinm_pinp}
		\MTSpin\wedge \Sigma^{-1}\RP^2\longrightarrow \MTPin^\pm \xrightarrow{\sm_{2\sigma}}
		\Sigma^2\MTPin^\mp,
	\end{equation}
	which was first constructed by Kirby-Taylor~\cite[Lemma 7]{KT90pinp}. Here we have used the identifications of
	\pinp, resp.\ \pinm bordism as $(B\Z/2, 3\sigma)$, resp.\ $(B\Z/2, \sigma)$-twisted spin bordism that we
	discussed in \cref{spin_4periodic}. In \cref{fig:Pinm_Pinp_bordism_LES}, we calculate the long exact sequence on bordism groups corresponding to~\eqref{pinm_pinp} (specifically, the \pinm to \pinp case) in low degrees.
 See~\cite{DDHM22} for an application of a related but different Smith
	homomorphism in physics.

	The Smith homomorphism in~\eqref{pinm_pinp} is the composition of two of the Smith homomorphisms in the
	$4$-periodic collection of \cref{spin_4periodic}, where we go from \pinp to $\Spin\times\Z/2$ to \pinm, or from
	\pinm to $\Spin\times_{\set{\pm 1}}\Z/4$ to \pinp. The other two compositions, which exchange the spin bordism
	of $B\Z/2$ with $\Spin\times_{\set{\pm 1}}\Z/4$ bordism, are~\eqref{spinzk_Smith} for $n = 2$.
\end{itemize}
\end{example}

\subsection{A few more examples}
In this section, we record some examples of Smith cofiber sequences that do not arise from real or complex line bundles.
\begin{example}
\label{quater_exm}
Like our previous examples over $B\Z/2$ and $B\T$, we can study Smith homomorphisms for the tautological
quaternionic line bundle $V\to B\SU(2)$. Once again, the sphere bundle of $V$ is contractible, as it is $E\SU(2)\to
B\SU(2)$, so we obtain Smith cofiber sequences like in~\eqref{smith_Z2} and~\eqref{smith_T}:
\begin{equation}
\label{smith_SU2}
	\mathbb S\longrightarrow (B\SU(2))^{k(V-4)}\xrightarrow{\sm_V} \Sigma^4 (B\SU(2))^{(k+1)(V-4)}.
\end{equation}
For $k = 0$, this sequence splits, yielding a third Smith isomorphism $\Sigma^\infty B\SU(2)\xrightarrow{\cong}
(B\SU(2))^V$. This equivalence is well-known, e.g.\ \cite[\S 2]{Tam97}.

This bundle has a $G$-structure for $G$ including $\O$, $\SO$, $\Spin$, $\Spin^c$, $\U$, $\SU$, and $\mathrm{Sp}$,
and in all of these cases, smashing with $\mathit{MTG}$ produces Smith homomorphisms similar to those in
\cref{Z2_MO,cpx_triv_exms}. The proof of \cref{U_smith_split} still works in this setting, and for these $G$ we
obtain $\mathit{MTG}$-module splittings
\begin{equation}
\label{symplectic_splitting}
	\mathit{MTG}\wedge (B\SU(2))_+\simeq \bigvee_{k\ge 0} \Sigma^{4k}\mathit{MTG}.
\end{equation}
When $G = \Spin$, \eqref{symplectic_splitting} is closely related to the splittings in \cref{CP_ko,HP_ko}.

The Smith map~\eqref{smith_SU2}, after smashing with $\mathit{MTSp}$, was studied by Landweber~\cite[\S 5]{Lan68} and Ray~\cite[\S 3]{Ray72}. An analogue in \term{self-conjugate bordism} (see~\cite[\S 8]{SS68}) was studied by Gozman~\cite[\S 3]{Goz77}.
\end{example}
\begin{example}
\label{string_SU2}
As we mentioned in \cref{string_Z2}, Lie $2$-group extensions of a compact Lie group $G$ by $B\U(1)$ are classified by $H^4(BG;\Z)$~\cite{SP11, Wei22}. Let $\String\text{-}\SU(2)$ be the Lie $2$-group belonging to the extension
\begin{equation}
    \shortexact{}{}{B\U(1)}{\String\text{-}\SU(2)}{\Spin\times\SU(2)}{}
\end{equation}
classified by $\lambda + p_1^\H\in H^4(B\Spin\times B\SU(2);\Z)$; here $\lambda\in H^4(B\Spin;\Z)$ and $p_1^\H\in H^4(B\SU(2);\Z)$ are the canonical generators of $H^4$ of a connected, simply connected, simple Lie group. (For $B\Spin$, we may use $B\Spin(n)$ with $n\gg 0$.) Using the usual map $\Spin\to\O$, we obtain a tangential structure $B(\String\text{-}\SU(2))\to B\O$; by an argument similar to the one in~\cite[\S 10.4]{DDHM22} (see~\cite[\S 3.3.1]{BDDM23}), $\mathit{MT}(\String\text{-}\SU(2))\simeq \MTString\wedge (B\SU(2))^{L-4}$, where $L\to B\SU(2)$ is the tautological quaternionic line bundle. Thus~\eqref{smith_SU2} with $k = 0$, smashed with $\MTString$, produces a Smith isomorphism
\begin{equation}\label{string_smith_isom}
    \mathrm{sm}_L\colon \widetilde\Omega_*^\String(B\SU(2)) \overset\cong\longrightarrow
    \Omega_{*-4}^{\String\text{-}\SU(2)}.
\end{equation}
As $L\to B\SU(2)$ is not string, the argument for~\eqref{symplectic_splitting} does not apply, and indeed one can show $\mathit{MT}(\String\text{-}\SU(2))$ does not split in that way. Thus this Smith isomorphism is expressing something nontrivial about string-$\SU(2)$ bordism. To our knowledge, \eqref{string_smith_isom} is the first such nontrivial quaternionic Smith isomorphism known.

String-$\SU(2)$ bordism appears in~\cite[\S 3]{BDDM23} as an intermediary to other twisted string bordism computations, and Bruner-Rognes~\cite[\S 1.4, Chapter 8, \S 12.3, Appendix D]{BR21} study a closely related object called $\mathit{tmf}/\nu$.
\end{example}
\begin{example}
\label{spinh_exm}
Consider the Smith homomorphisms coming from the tautological rank-$3$ vector bundle $V\to B\SO(3)$. 
The one-point compactification of $\mathfrak{so}(3)/\mathfrak{u}_1$ is isomorphic to $\SO(3)/\T \cong S^2$. Since $\mathfrak{so}(3)/\mathfrak{u}_1 \oplus \R$ is isomorphic to the defining representation $V$ of $\SO(3)$, we obtain a cofiber sequence of spectra
\begin{equation}
\label{SO3_smith_cofib}
	(B\T)^{k(L-2)} \longrightarrow (B\SO(3))^{k(V-3)} \longrightarrow \Sigma^3 (B\SO(3))^{(k+1)(V-3)}.
\end{equation}
We are most interested in smashing this sequence with $\MTSpin$.\footnote{\label{spinu_footnote}It is also interesting to
smash~\eqref{SO3_smith_cofib} with $\MTSpin^c$: in this case one obtains a codimension-$3$, $2$-periodic family of
Smith homomorphisms exchanging the \spinc bordism of $B\SO(3)$ with ``spin-$\U(2)$ bordism,'' i.e.\ bordism of the
group $\Spin\times_{\set{\pm 1}}\U(2) \cong \Spin^c\times_{\set{\pm 1}} \SU(2)$. Davighi-Lohitsiri~\cite{DL20, DL21}
introduced Spin-$\U(2)$ bordism and calculated it in low dimensions; spin-$\U(2)$ structures also appear in
Seiberg-Witten theory (e.g.\ \cite{FL02, DW19}) under the name \term{spin\textsuperscript{$u$} structures}.} Note that $V$ is
not spin, but because $V$ is oriented, $2V$ is spin; therefore we obtain a $2$-periodic family of codimension-$3$
Smith homomorphisms exchanging the spin bordism of $B\SO(3)$ and $(B\SO(3), V)$-twisted spin bordism.
Freed-Hopkins~\cite[(10.20)]{FH16} identify $(B\SO(3), V)$-twisted spin bordism with bordism for the group
$G^0\coloneqq \Spin\times_{\set{\pm 1}}\SU(2)$, which is in various sources called spin\textsuperscript{$h$}
bordism, spin\textsuperscript{$q$} bordism, spin-$\SU(2)$ bordism, or $G^0$ bordism.\footnote{To the best of our
knowledge, spin\textsuperscript{$h$} structures were first studied in~\cite{BFF78} in the context of quantum
gravity; they have also been applied to Seiberg-Witten theory~\cite{OT96}, index theory, e.g.\ in~\cite{May65,
Nag95, Bar99, FH16, Che17}, almost quaternionic geometry, e.g.\ in~\cite{Nag95, Bar99, AM20}, immersion
problems~\cite{Bar99, AM20}, and the study of invertible field theories~\cite{FH16, BC18, WWW19, DY22}. See~\cite{Law23} for a review and~\cite{BM23, Hu23, Mil23, DK24} for additional related work.} The fiber
we've seen before in \cref{spinc_spin_Smith}: \spinc bordism when $k$ is odd in~\eqref{SO3_smith_cofib}, and the
spin bordism of $B\T$ when $k$ is even.

In summary, we have two Smith cofiber sequences
\begin{subequations}
\label{spinh_cofib}
\begin{gather}
\label{spinh_to}
	\MTSpin^c\longrightarrow \MTSpin^h \xrightarrow{\sm_V} \Sigma^3 \MTSpin\wedge (B\SO(3))_+\\
	\MTSpin\wedge (B\T)_+ \longrightarrow \MTSpin\wedge (B\SO(3))_+ \xrightarrow{\sm_V} \Sigma^3 \MTSpin^h.
\end{gather}
\end{subequations}
The long exact sequence of bordism groups associated to~\eqref{spinh_to} appears in \cref{euler_exm_main} as an example where one must use the cobordism Euler class to calculate the Smith homomorphism: ordinary cohomology Euler classes give the wrong answer. 
Other works studying anomalies of spin\textsuperscript{$h$} QFTs include~\cite{FH16, WW19, WWW19, WWZ19, DL20, WW20a, BCD22, DY22, WY22, DYY}.
\end{example}
\begin{remark}
Freed-Hopkins~\cite{FH16} also study two unoriented analogues of spin\textsuperscript{$h$} structures, called
pin\textsuperscript{$h\pm$} or $G^\pm$ structures, corresponding to the groups $\Pin^\pm\times_{\set{\pm 1}}\SU(2)$.
It would be interesting to work out analogues of the Smith homomorphisms such as the ones in
\cref{spin_4periodic,spinh_exm} for pin\textsuperscript{$h\pm$} structures and apply them to symmetry breaking; see~\cite{DK24} for some work in that direction. Pin\textsuperscript{$h\pm$} manifolds are also studied in~\cite{BC18, GPW18, LS19, AM20, DYY}.
\end{remark}
\begin{example}
\label{smith_su2_vector}
If we pull \cref{spinh_exm} back to $B\SU(2)$, we obtain a Smith long exact sequence which makes an appearance both in \cite[\S IV.F]{PhysSmith}
and in Appendix~\ref{s:eu_counter}. 
\label{jamesQP}

The tautological quaternionic line bundle over $B\SU(2)$ is \textit{not} isomorphic to the bundle associated to $\mathfrak{su}_2 \oplus \mathbb{R}$, where $\mathfrak{su}_2$ is the adjoint representation of $\SU(2)$. Rather, since $\mathfrak{su}_2 \cong \mathbb{R} \oplus \mathfrak{su}_2/\mathfrak{u}_1$, the map $B\T \to B\SU(2)$ exhibits $B\T$ as the unit sphere bundle in the adjoint representation of $\SU(2)$. It follows that there is a cofiber sequence
\begin{equation}
\label{SU2_Vec}
    B\U(1) \longrightarrow B\SU(2) \xrightarrow{\sm_V} \Sigma^3 (B\SU(2))^{V-3},
\end{equation}
where $V\to B\SU(2)$ is the vector bundle associated to $\mathfrak{su}(2)$. We claim the first map is induced by the inclusion of a maximal torus into $\SU(2)$.
To see that the sphere bundle is $B\U(1)$ as claimed, identify $\SU(2)\to\SO(3)$ with $\Spin(3)\to\SO(3)$ and $\T\to\SU(2)$ with $\Spin(2)\to\Spin(3)$; by the third isomorphism theorem, $\Spin(3)/\Spin(2)\cong \SO(3)/\SO(2)$, and in \cref{spinh_exm} we identified that quotient with the unit sphere inside $\R^3$. Therefore taking associated bundles, we end up with $B\Spin(2)$ as the fiber in~\eqref{SU2_Vec}.

Since $\SU(2)$ is simply connected, $B\SU(2)$ is $2$-connected and therefore all of its vector bundles admit spin structures. Thus, when we smash~\eqref{SU2_Vec} with $\MTSpin$, we obtain a cofiber sequence
\begin{equation}
\label{spin_SU_22}
    \MTSpin\wedge (B\U(1))_+\longrightarrow \MTSpin\wedge (B\SU(2))_+ \xrightarrow{\sm_V} \Sigma^3 \MTSpin\wedge (B\SU(2))_+.
\end{equation}
The Thom spectrum $(B\SU(2))^{\mathfrak{su}(2)}$ is known as James' ``quasiprojective space'' (see \cite{James-Stiefel}).
The Anderson dual of~\eqref{spin_SU_22} appears in a physics application in \cite[\S IV.F.]{PhysSmith}.

The same Thom isomorphism applies for any $\MTSpin$-oriented ring spectrum, such as $\MTSO$ or $\ko$; if we used $\ko$ instead of $\MTSpin$ in~\eqref{spin_SU_22}, we would obtain the cofiber sequence in~\eqref{SGCof}.
\end{example}

\begin{example}\label{spin_u1_1_periodic}
In \cite[\S III.B.1]{PhysSmith},
we study the SBLES in twisted spin bordism corresponding to the vector bundle $2L\to B\U(1)$, where $L$ denotes the tautological bundle. Since $2L$ is spin, we obtain a one-periodic family of Smith homomorphisms of the form
\begin{subequations}
\begin{equation}
\label{2Lcofiba}
    S(2L)\longrightarrow B\U(1) \xrightarrow{\sm_{2L}} \Sigma^4 (B\U(1))^{2L - 4}.
\end{equation}
The new wrinkle is showing that $S(2L)\to B\U(1)$ is homotopy equivalent to the map $S^2\to B\U(1)$ given by the inclusion of the $2$-skeleton. But this is not so hard: using the long exact sequence in cohomology associated to the cofiber sequence, one learns that if $C$ is the cofiber of $\mathrm{sm}_{2L}$, $\widetilde H^*(C;\Z)$ vanishes except in degree $3$, where it is $\Z$; this characterizes $S^3$, so the fiber, which is the total space of the sphere bundle, is $S^2$. Stably this splits as $\mathbb S \vee \Sigma^2\mathbb S$, so our cofiber sequence is
\begin{equation}
\label{2Lcofibb}
   \mathbb S\vee \Sigma^2 \mathbb S\longrightarrow \Sigma_+^\infty B\U(1) \xrightarrow{\sm_{2L}} \Sigma^4 (B\U(1))^{2L-4}.
\end{equation}
\end{subequations}
This cofiber sequence is a complexified version of~\eqref{Zk_smith}. One therefore wonders what happens if we consider it within its family \begin{equation}\label{2L_family}
    \mathrm{sm}_{2L}\colon (B\U(1))^{kL - 2k} \longrightarrow \Sigma^4 (B\U(1))^{(k+1)L - 2k-2}.
\end{equation}
If we smash with $\MTSpin$, this is a $2$-periodic family: it only matters whether $k$ is odd or even. For $k$ even we reduce to~\eqref{2Lcofibb} above; for $k$ odd, we have a very similar cofiber sequence, but the sphere bundle does not split: we obtain for the fiber $(\CP^1)^{\mathcal O(-1) - 2}\simeq \CP^2$:
\begin{equation}
    \MTSpin\wedge\CP^2 \longrightarrow \MTSpin^c \longrightarrow \Sigma^4 \MTSpin^c,
\end{equation}
using the identification $\MTSpin\wedge (B\U(1))^{L-2}\simeq\MTSpin^c$ from \cref{twists_of_spin}. This is the complex analogue of~\eqref{pinm_pinp}.
\end{example}
\begin{remark}
There is a related example where one uses $L\oplus L^*\to B\U(1)$ instead of $2L$; the corresponding long exact sequence in twisted $\mathrm{SU}$-bordism was studied by Conner-Floyd~\cite[\S\S6, 14, 17]{CF66}. When $k$ is odd, the third term in the long exact sequence, corresponding to the sphere bundle, is the bordism of manifolds with \term{$c_1$-aspherical structures} or \term{complex Wall structures}, first introduced by Conner-Floyd~\cite{CF66}, and also discussed by
Stong~\cite[Chapter VIII]{Sto68}. Complex Wall bordism plays an important role in the calculation of
$\Omega_*^\SU$ via the Adams-Novikov spectral sequence~\cite[\S 7]{Nov67}, and has also been studied in the context
of complex orientations~\cite{Buh72, CP21, Che22}.
\end{remark}
\begin{example}
\label{GMTW_exm}
The unit sphere bundle to the tautological bundle $V_{n+1}\to B\O(n+1)$ is homotopy equivalent to the map $B\O(n)\to B\O(n+1)$. This is because $S^n\cong\O(n+1)/\O(n)$, so the unit sphere bundle can be described by the mixing construction
\begin{equation}
    S^n\times_{\O(n+1)} E\O(n+1)\cong (\O(n+1)/\O(n))\times_{\O(n+1)} E\O(n+1)\cong E\O(n+1)/\O(n)\cong B\O(n).
\end{equation}
More generally, if $\xi_{n+1}\colon B_{n+1}\to B\O(n+1)$ is an unstable tangential structure and $\xi_n\colon B_n\to B\O(n)$ is the pullback of $\xi_{n+1}$ by $B\O(n)\to B\O(n+1)$, the sphere bundle of $\xi_{n+1}^*V_{n+1}$ is the pullback of $S(V_{n+1}) = B\O(n)$ by $\xi_{n+1}$, which is $\xi_n$. If you then pull $\xi_{n+1}^*V_{n+1}$ back across $B_n\to B_{n+1}$, it splits as $V_n\oplus\underline\R$, so there is a Smith cofiber sequence
\begin{equation}
    \Sigma^{-1} B_n^{n-V_n} \longrightarrow B_{n+1}^{n+1-V_{n+1}} \longrightarrow \Sigma_+^\infty B_{n+1}.
\end{equation}
This cofiber sequence is due to Galatius-Madsen-Tillmann-Weiss~\cite[(3.3), \S 5]{GMTW09}. The spectrum $\Sigma^n B_n^{n -\xi_n^*{V_n}}$ is often denoted $\mathit{MT\xi}_n$.
\end{example}

\section{Long exact sequence of invertible field theories}\label{LESIFTs}
In this section, we turn to physical applications of the Smith fiber sequence (\cref{the_cofiber_sequence}). 
In \cref{subsec:anomalies-and-IFT} we give a brief overview of anomalies and how they relate to invertible field theories.
In \cref{subsec:thom-IFT-anomalies} we review the mathematical classification of invertible field theories as Anderson dual groups of Thom spectra. In \cref{subsec:LES-of-IFT} we show how the Smith fiber sequence gives rise to a long exact sequence of invertible field theories (\cref{IZ_LES_cor}), which we call the symmetry breaking long exact sequence (SBLES). Finally, we give a mathematically-oriented introduction to our companion work \cite{PhysSmith} and a table of cross-listed examples (\cref{tab:my-table}).

Quantum field theory is not yet completely mathematically formalized, and the applications of the Smith homomorphism in~\cite{HKT19, COSY19, PhysSmith} take place at a physical level of rigor, not a mathematical one. As such, parts of this section are also only at a physical level of precision.

\subsection{Anomalies and invertible field theories}
\label{subsec:anomalies-and-IFT}
Our homotopy-theoretic techniques in this paper apply to the classification of invertible field theories, while the physical objects we ultimately wish to study are (not necessarily invertible) quantum field theories with potentially anomalous symmetries. To apply our techniques, we take the following perspective on anomalies.\footnote{There are many things called ``anomalies'' in quantum field theory, and we are not claiming our definitions or approach is universal. Instead, our application is to a broad class of anomalies.}

Let us first consider anomalies of symmetries. Physically, we begin with a $n$-dimensional quantum system (such as a quantum field theory) with partition function $Z$, as well as a symmetry group $G$, which is usually a compact Lie group. We can couple the theory to a background $G$-gauge field $A$. 
Then, put simply, the $G$ symmetry is \emph{anomalous} if the partition function evaluated on a pair of a closed $n$-manifold $M$ and a background gauge field with connection $A$ is \emph{not} gauge invariant but rather transforms with a phase. That is, under a gauge tranformation $A \mapsto A^g$, the partition function transforms as\begin{equation}
Z(M,A^g) = e^{i \alpha(M,A,g)} Z(M,A),
\end{equation}
where $e^{i \alpha(M,A,g)}$ is a phase factor that cannot be cancelled by local counter-terms.

Mathematically, this means that when we evaluate the theory on a $n$-dimensional manifold $M$ with a $G$-principal bundle with connection $A$, the partition function is a section of a non-trivial line bundle. The non-triviality of the line bundle is the anomaly.

To relate anomalies to invertible field theories, we use the notion of \emph{anomaly inflow}: under mild hypotheses, and in all known cases, there is a local counterterm $e^{i\omega(K,A)}$ defined in one dimension higher, so that if $K$ is an $(n+1)$-dimensional manifold with boundary and $\partial K = M$ (and $A$ extends into $K$), then
\begin{equation}\label{eq:anomaly-inflow}
e^{i \omega(K,A^g) - i \omega(K,A)} = e^{i\alpha(M,A,g)}.
\end{equation}
We may interpret $e^{i\omega(K,A)}$ as the partition function of an $n+1$-dimensional invertible field theory with $G$ symmetry. It is invertible, as stacking with the $e^{-i\omega(K,A)}$
theory gives the trivial theory. 
Furthermore, $Z$ naturally lives at the boundary of this invertible theory, and together they are gauge invariant by \cref{eq:anomaly-inflow}.
Therefore we can interpret the existence of the $G$-anomaly, which is the failure of $Z(M, A)$ to be gauge invariant, as the statement that $Z$ is a boundary theory of a non-trivial $(n+1)$-dimensional invertible field theory, which we call the bulk theory. This perspective, formulated mathematically by Freed-Teleman~\cite{freed2014relative}, is the link between anomalies and invertible field theories; see also Freed~\cite{Fre23}.

We also study anomalies of families of field theories parameterized by a topological space $X$.
An anomaly of an $X$-family of field theories indicates a failure of the partition function to be consistently defined over the space of background $X$-fields.\footnote{Typically in physics, $X$ carries more structures, such as a smooth structure or Riemannian metric. The anomalies we consider here will not depend on those structures.}  In this case, the bulk theory is an one-dimension-higher $X$-family of invertible field theories. More generally, we can ask for $X$ to be a space with $G$-action, and consider anomalies of $G$-equivariant $X$-families field theories.
We refer the reader to \cite[\S II.A]{PhysSmith} and the references therein for more detail. 

\subsection{Invertible field theories and Thom spectra}\label{subsec:thom-IFT-anomalies}

Here we review the mathematical classification of (reflection-positive) invertible theories. First we must relate symmetries and tangential structures.
\subsubsection{Symmetries and tangential structures}\label{math_tangential}

For any tangential structure in the sense of \cref{defn_tangstr}, there is a notion of topological field theory.
Given a field theory whose anomaly we want to investigate, which tangential structure $\xi$ do we want
our invertible field theories to carry?\footnote{For non-topological invertible field theories, there is also the question of enriching the tangential structure to something more geometric, such as including the data of a Riemannian metric or a connection for a principal bundle. At the level of invertible field theories this is accounted for by working with Anderson duals of bordism spectra, so our calculations include these geometric modifications.}

The answer typically depends only on the symmetries of our field theory, not on its field content (the anomaly
itself---which invertible field theory we get out of all the invertible field theories on $\xi$-manifolds---uses
more information from the theory).
We follow Freed-Hopkins~\cite[\S 2]{FH16}, who take the stance that since we
typically study QFTs in Minkowski signature but invertible field theories are Euclidean, we should Wick-rotate the
group of symmetries to define our tangential structure.

Assume the dimension $n$ is at least $2$. Let $\mathcal I(1, n-1)$ be the
isometry group of Minkowski space, and let $\mathcal I(1, n-1)^{\uparrow}\subset \mathcal I(1, n-1)$ be the
subgroup of isometries that preserve the direction of time. The group of symmetries of our theory is a Lie group
$\mathcal H(1, n-1)$ with a map $\rho(n)\colon \mathcal H(1, n-1)\to \mathcal I(1, n-1)^\uparrow$. Let $K\coloneqq
\mathrm{ker}(\rho(n))$; we assume $K$ is compact. Assume that the normal subgroup of translations
$\R^{1,n-1}\subset\mathcal I(1, n-1)$ lifts to a normal subgroup of $\mathcal H(1, n-1)$, and let
$H(1, n-1)\coloneqq \mathcal H(1, n-1)/\R^{1,n-1}$. Now:
\begin{subequations}
\begin{enumerate}
	\item Let $\O(1, n-1)^\uparrow\coloneqq \O(1, n-1)\cap\mathcal I(1, n-1)$. There is an exact sequence
	\begin{equation}
		0\longrightarrow K\longrightarrow H(1, n-1)\longrightarrow \O(1, n-1)^\uparrow.
	\end{equation}
	\item This exact sequence can be extended to an exact sequence of complexifications:
	\begin{equation}
		0\longrightarrow K(\C)\longrightarrow H(n, \C)\longrightarrow \O(n, \C),
	\end{equation}
	\item and then to compact real forms of these complex Lie groups:
	\begin{equation}
		0\longrightarrow K \longrightarrow H(n) \longrightarrow \O(n).
	\end{equation}
\end{enumerate}
\end{subequations}
The tangential structure that the anomaly field theory has is $\xi\colon BH(n)\to B\O(n)$. Just as it is not a
priori clear that the anomaly field theory extends to dimension $n+1$, it is also not necessarily clear that $\xi$
extends to an $(n+1)$-dimensional unstable tangential structure, but Freed-Hopkins~\cite[Theorem 2.19]{FH16} prove
that it does in nearly every situation one might want, as we discuss below. In this paper, we will always be in the
situation that $\xi$ extends to $(n+1)$-manifolds.

In practice, one can often use an idea of Stehouwer~\cite{Ste21} to compute $\xi$ using the formalism of \term{fermionic groups} (sometimes called \term{supergroups}).
\begin{definition}[{Benson~\cite[\S 7]{Ben88}}]
A \term{fermionic group} is a topological group $G$ together with data of:
\begin{itemize}
	\item a central element squaring to $1$, which we call \term{fermion parity} and denote $-1\in G$, and
	\item a group homomorphism $\theta\colon G\to\Z/2$ such that $\theta(-1) = 0$.
\end{itemize}
\end{definition}
We think of $\theta$ as defining a $\Z/2$-grading on $G$, and we refer to elements of $G$ as odd or even. The even
elements form a subgroup $G_0\subset G$, which is itself a fermionic group with $\theta$ trivial.

Given two fermionic groups $G$ and $H$, one can take their \term{fermionic tensor product} (\textit{ibid.}) $G\otimes H\coloneqq (G\times H)/\ang{(-1, -1)}$. This is a fermionic group, with central element $(-1, 1) = (1, -1)$ and grading $\theta((g, h))$ equal to the sum mod $2$ of the gradings on $g$ and on $h$.

Fermionic groups describe symmetries of theories with fermions: $-1$ acts by fermion parity, which may mix
nontrivially with other symmetries in the theory; and $\theta$ describes whether elements of $G$ act unitarily or
antiunitarily. Given a fermionic group $G$, Stolz~\cite[\S 2.6]{Sto98} defines a tangential structure $\xi_G\colon B\to B\O$ as follows: let $H$ be the even subgroup of the fermionic tensor product $\Pin^+\otimes G$; here, to make $\Pin^+$ into a fermionic group, we use the usual $-1$, and the grading homomorphism is $\pi_0\colon\Pin^+\to\O(1) \cong \Z/2$. Then $B\coloneqq BH$, and the map $\xi\colon B\to B\O$ is induced from the usual map $\Pin^+\to\O$ and the constant map to the identity on the quotient of $G_0$ by fermion parity.
See~\cite{Ste21} for several examples of computations of tangential structures from data of the symmetries of a
theory. 

Tangential structures can encode not only the symmetries of quantum systems, but also the parameter space. For example, $S^1$-families of fermionic theories with an internal unitary $\Z/2$ symmetry may be described using the tangential structure $B\Spin \times B\Z/2 \times S^1 \to B\Spin \to BO$. 
There are also variations and twists: if $\Z/2\Z$ acts on $S^1$, then we replace 
$B\Z/2\Z \times S^1$ with the homotopy quotient $S^1/(\Z/2\Z)$. To encode a time-reversal symmetry $T$ with $T^2 = (-1)^F$, we replace $B\Spin \times B\Z/2$ with $B\Pinp$. Note that the process of describing the symmetry and parameter space of a field theory can be subtle.
See ~\cite{PhysSmith} for more examples of converting the data of symmetries and parameter spaces to tangential structures.

\subsubsection{Invertible field theories and bordism invariants}\label{sss:IFT_bord}
At this point in the story we have turned the physics question of determining the possible anomalies of a theory
with a given collection of symmetries into the mathematical question of classifying (reflection-positive)
invertible field theories for a fixed tangential structure $\xi\colon B\to B\O$, with $B = BG$. 

In this subsubsection we discuss how this
classification question reduces to a well-studied problem in algebraic topology: the computation of groups of
bordism invariants. See Freed~\cite[Lectures 6--9]{Fre19} and Galatius~\cite{Gal21} for more detailed reviews of this story.

Fix nonnegative integers $k\le n$; $n$ will be the dimension of the field theories we consider and $k$ will be the category number. Let $\Bord_n^\xi$ denote the \term{bordism $k$-category} of $\xi$-manifolds, whose objects are closed $(n-k)$-dimensional $\xi$-manifolds, morphisms are ($\xi$-diffeomorphism classes rel boundary of) compact $(n-k+1)$-dimensional $\xi$-structured bordisms between them, $2$-morphisms are ($\xi$-diffeomorphism classes rel boundary of) $(n-k+2)$-dimensional $\xi$-structured bordisms between bordisms, and so on; see~\cite{Lur09, CS19} for more details. We will also fix a symmetric monoidal (weak) $k$-category $\cat C$. The motivating cases, when $k = 1$, are the categories $\cat{Vect}$ and $\cat{sVect}$ of complex vector spaces, resp.\ super vector spaces, with tensor product as the symmetric monoidal structure.

Following Atiyah~\cite{Ati88} and Segal~\cite{Seg88}, by a \term{field theory} we mean a symmetric monoidal functor $Z\colon\Bord_n^\xi\to\cat C$; the \term{dimension} of this theory is $n$.
\begin{definition}[{Freed-Moore~\cite[Definition 5.7]{FM06}}]
An $n$-dimensional field theory $Z\colon\Bord_n^\xi\to\cat C$ is \textit{invertible} if there is some other theory $Z^{-1}$ such that $Z\otimes Z^{-1}$ is the trivial theory.
\end{definition}
This tensor product is evaluated ``pointwise,'' meaning that $(Z\otimes Z^{-1})(M)\coloneqq Z(M)\otimes Z^{-1}(M)$, where $M$ is an object, morphism, etc.\ in the bordism category; therefore invertibility implies that $Z$, as a functor, factors
through the Picard sub-$k$-groupoid of units $\cat C^\times$ inside $\cat C$, meaning that if $X$ is any object,
morphism, or higher morphism in $\Bord_n^\xi$, $Z(X)$ is invertible: $\otimes$-invertible if $X$ is an object, and
composition-invertible if $X$ is a (higher) morphism. If $X$ is invertible, then we must have data of an
isomorphism $Z(X^{-1}) \xrightarrow{\cong} Z(X)^{-1}$ because $Z$ is symmetric monoidal; thus, even if $X$ is not
invertible, we can heuristically \emph{define} $Z(X^{-1})\coloneqq Z(X)^{-1}$ as if $X^{-1}$ existed. These
definitions are compatible as $X$ varies, in the sense that $Z$ extends to the \term{Picard $k$-groupoid
completion} $\overline{\Bord}{}_n^\xi$ of $\Bord_n^\xi$: the Picard $k$-groupoid defined by formally adding
inverses to all objects, morphisms, higher morphisms, etc.\ of $\Bord_n^\xi$. Thus, an invertible field theory
$Z\colon\Bord_n^\xi\to\cat C$ is equivalent data to a morphism of Picard $k$-groupoids
\begin{equation}
	Z\colon \overline{\Bord}{}_n^\xi\longrightarrow \cat C^\times.
\end{equation}
So to compute deformation classes of invertible field theories, we should compute the groups of symmetric monoidal
functors between these Picard $k$-groupoids, modulo natural isomorphisms. The homotopy theory of Picard groupoids
embeds in the usual stable homotopy category: if $\cat D$ is a Picard groupoid, the geometric realization
$\abs{N\cat D}$ of the nerve of $\cat D$ has an $E_\infty$-structure arising from the monoidal product on
$\cat D$, and the Picard condition implies $\abs{N\cat D}$ is grouplike. Therefore it is equivalent data to a
connective spectrum $\abs{\cat D}$, which we call the \term{classifying spectrum} of $\cat D$. This turns out to be
a complete invariant of Picard $k$-groupoids.
\begin{theorem}[Stable homotopy hypothesis (Moser-Ozornova-Paoli-Sarazola-Verdugo~\cite{MOPSV22})]
\label{stable_htpy_hyp}
There is an equivalence of $\infty$-categories between the $\infty$-category of Picard $k$-groupoids and the
$\infty$-category of spectra whose homotopy groups vanish outside of $[0,k]$.
\end{theorem}
\begin{remark}
For $k = 1$, the stable homotopy hypothesis was originally a folklore theorem: proofs or sketches appear
in~\cite{BCC93, HS05, Dri06, Pat12, JO12, GK14}. For $k =2$, the stable homotopy hypothesis was proven by
Gurski-Johnson-Osorno~\cite{GJO17}.
\end{remark}
Therefore we need to compute the group of homotopy classes of maps of spectra
$\abs{\overline{\Bord}{}_n^\xi}\to\abs{\cat C^\times}$. A reasonable first step would be to identify these two
classifying spectra. For the domain, the Picard $k$-groupoid completion of the bordism category, this is due to
Galatius-Madsen-Tillmann-Weiss~\cite{GMTW09} and Nguyen~\cite{Ngu17} for the bordism $(\infty, 1)$-category and to
Schommer-Pries~\cite{SP17} for more general $(\infty, k)$-categories.
\begin{theorem}[Galatius-Madsen-Tillmann-Weiss~\cite{GMTW09}, Nguyen~\cite{Ngu17}, Schommer-Pries~\cite{SP17}]
If $\Bord_n^\xi$ denotes the $(\infty, k)$-category of bordisms of $\xi_n$-structured manifolds in dimensions $n-k,
\dotsc,n$, then there is a natural equivalence $\abs{\overline{\Bord}_n^\xi}\simeq \Sigma^k \mathit{MT\xi}_n$.
\end{theorem}
Here $\mathit{MT\xi}_n$ is a Madsen-Tillmann spectrum as in \cref{mads_till}.

Freed-Hopkins-Teleman~\cite{FHT10} then applied this result to classify invertible field theories in terms of
$\mathit{MT\xi}_n$. For that, one needs to determine $\abs{\cat C^\times}$, which of course depends on the choice of $\cat
C$---Freed-Hopkins~\cite[\S 5.3]{FH16} argue that the (shifted) \term{character
dual of the sphere spectrum} $\Sigma^n I_{\C^\times}$ is a universal choice, and that a related object called the
(shifted) \term{Anderson dual of the sphere spectrum} $\Sigma^{n+1}I_\Z$
should appear when one wants to classify
deformation classes of invertible field theories. For applications to anomalies, we are interested in deformation
classes, so use $\Sigma^{n+1}I_\Z$.

The Anderson dual $I_\Z$ is characterized by its universal property that for any spectrum $\mathcal{X}$, there is a short exact sequence~\cite{And69, Yos75}
\begin{equation}
	\label{IZ_SES}
		\shortexact{\varphi}{\psi}{\mathrm{Tors}(\Hom(\pi_{n+1}\mathcal{X}, \C^\times))}{[\mathcal{X},
		\Sigma^{n+2}I_\Z]}{\Hom(\pi_{n+2}\mathcal{X}, \Z)}.
	\end{equation}

We are interested in anomalies of unitary QFTs, hence we expect the anomaly theories to satisfy the Wick-rotated
analogue of unitarity: reflection positivity. Freed-Hopkins~\cite[\S 7.1, \S 8.1]{FH16} define
reflection positivity for invertible TFTs using $\Z/2$-actions on $\Bord_n^\xi$ and $\cat C$,\footnote{The
definition of reflection positivity for extended not-necessarily-invertible TFTs is still open: see~\cite{JF17, DAGGER24, Ste24, MS23} for work in this direction.} and prove two key results allowing for a complete classification of reflection
positive invertible TFTs following their definition.
\begin{theorem}[{Freed-Hopkins~\cite[Theorem 2.19]{FH16}}]
\label{FH_stabilization}
If $n\ge 3$ and $\xi_n\colon BH(n)\to B\O(n)$ is a tangential structure arising from a representation $\rho\colon
H(n)\to\O(n)$ with $H(n)$ a compact Lie group and $\SO(n)\subset\mathrm{Im}(\rho)$, then there is a stable tangential
structure $\xi\colon BH\to B\O$ such that $\xi_n$ is the pullback of $\xi$ along $B\O(n)\to B\O$.
\end{theorem}
\begin{theorem}[{Freed-Hopkins~\cite[Theorem 5.23]{FH16}, Grady~\cite{Gra23}}]
\label{IFT_classification}
Suppose $\xi\colon BH(n)\to B\O(n)$ satisfies the hypotheses of \cref{FH_stabilization}.
The abelian group of deformation classes of 
$n$-dimensional, reflection positive invertible field
theories on manifolds with $\xi$-structure is naturally isomorphic to $[\mathit{MT\xi},
\Sigma^{n+1}I_\Z]$.
\end{theorem}
So after we require reflection positivity, the classification changes from unstable bordism to bordism in the
usual sense, which is easier to calculate.

\begin{remark}
There are some other approaches to the classification of invertible topological field theories, due to
Yonekura~\cite{Yon19}, Rovi-Schoenbauer~\cite{RS22}, and Kreck-Stolz-Teichner (unpublished; see~\cite{StolzTalk}).
\end{remark}
\Cref{IFT_classification} has a nice interpretation from the point of view of anomalies. Using the defining
property of $I_\Z$, there is a short exact sequence
\begin{equation}
\label{anomaly_SES}
	\shortexact{\varphi}{\psi}{\mathrm{Tors}(\Hom(\Omega_{n+1}^\xi, \C^\times))}{[\mathit{MT\xi},
	\Sigma^{n+2}I_\Z]}{\Hom(\Omega_{n+2}^\xi, \Z)},
\end{equation}
where $\mathrm{Tors}(\text{--})$ denotes the torsion subgroup. 
Using the Pontrjagin-Thom theorem (\cref{PTthm}) to interpret $\Omega_k^\xi$ geometrically as $\xi$-manifolds modulo bordism, the first and third terms in this short exact
sequence have anomaly-theoretic interpretations.
\begin{itemize}
	\item The quotient $\Hom(\Omega_{n+2}^\xi, \Z)$ is a free abelian group consisting of characteristic classes of
	$(n+2)$-dimensional $\xi$-manifolds; under this identification, the map $\psi$ sends an anomaly field theory to
	the corresponding anomaly polynomial, which is one degree higher, such as Chern-Simons and Chern-Weil forms. This data is visible to perturbative techniques, and is sometimes called
	the \term{local anomaly}.
	\item The subgroup $\mathrm{Tors}(\Hom(\Omega_{n+1}^\xi, \C^\times))$ is identified with the torsion subgroup
	of $[\mathit{MT\xi}, \Sigma^{n+2}I_\Z]$; these are the reflection positive invertible field theories which are
	topological. Such field theories' partition functions are bordism invariants, and the identification of these
	reflection positive invertible TFTs with $\mathrm{Tors}(\Hom(\Omega_{n+1}^\xi, \C^\times))$ assigns to a
	reflection positive invertible TFT its partition function. Typically this data is invisible to perturbative
	methods and is called the \term{global anomaly}.
\end{itemize}
Yamashita-Yonekura~\cite{YY21} and Yamashita~\cite{Yam21} relate the short exact sequence~\eqref{anomaly_SES} to a
differential refinement of $\mathrm{Map}(\mathit{MT\xi}, \Sigma^{n+2}I_\Z)$.

\subsection{Long exact sequence of invertible field theories}
\label{subsec:LES-of-IFT}
Just as the map of spectra of \cref{smith_defn_round_1} induces a Smith homomorphism on bordism, it dually induces a map of invertible field theories. Explicitly, we obtain this by mapping into the Anderson dual and taking homotopy groups. As in \cref{math_section_twisted_tangential_structures}, consider a topological space $X$, a tangential structure $\xi\colon B\to B\O$, a virtual bundle $V\to X$, and a vector bundle $W\to X$ of rank $r$. 
As we discussed in the end of \cref{math_tangential}, the tangential structure encodes both symmetries of the theory as well as the parameter space.
\begin{definition}\label{def:defect-anomaly-map}
    The \textit{defect anomaly map}, generalized from \cite{HKT19} section 4.2, is the map
\begin{equation}
    I_\Z\Omega^{k-r}_\xi(X^{V+W-r}) \xrightarrow{\mathrm{Def}_W} I_\Z\Omega^k_\xi(X^V)
\end{equation}
of invertible field theories induced by the zero section map.
\end{definition}
Physically, we interpret $I_\Z\Omega^k_\xi(X^V)$ to be classifying anomalies of QFTs in dimension $k-1$ with symmetry according to a $(X,V)$-twisted $\xi$-structure. The group mapping in, $I_\Z\Omega^{k-r}_\xi(X^{V+W-r})$, classifies the anomalies of defect theories in dimension $k-r-1$. These defect theories are created from the bulk theory, physically, by setting a non-trivial boundary condition on a symmetry-breaking order parameter, which corresponds to a section of the vector bundle $W$. See \cite[\S III.A]{PhysSmith} for further explanation.

For clarity, we begin a running example.

\begin{example}
    Recall the $\Z/2$ family of examples from \cref{spin_4periodic} and specifically \cref{pinp_spin_z2}:
    \begin{equation}\label{eq:this-one}
        \MTPin^+\xrightarrow{\sm_\sigma} \Sigma \MTSpin\wedge (B\Z/2)_+.
    \end{equation}
    Here, we take $\xi$ according to spin bordism and twist with the tautological line bundle $W=\sigma$ over $X=B\Z/2$. We take $V=3\sigma$ and apply the results of \cref{periodicity_and_shearing} to simplify (i.e., we use the fact that $4\sigma$ is spin).
    The corresponding map on invertible field theories is
    \begin{equation}
        I_\Z\Omega_{\Spin\times\Z/2}^{k-1} \xrightarrow{\mathrm{Def}_\sigma} I_\Z\Omega_{\Pinp}^{k}.
    \end{equation}

    Physically, we begin with a field theory with a pin$^+$ symmetry; that is, a fermionic theory with an additional time reversal symmetry $T$ that squares to fermion parity: $T^2=(-1)^F$. 
    We require the physical assumption that there is a $\Z/2$-odd bosonic operator $\phi$ such that the theory is gapped when the theory is deformed by $\phi$. One example is the Majorana mass term for $2+1$D Majorana fermions (see \cite[\S III.A.1]{PhysSmith}).
    We can define the theory on any manifold $M$ with a pin$^+$ structure $P \colon M \to B\Pinp$ on its tangent bundle. 
    If we choose generic configuration for the $\phi$ field, i.e., choose a section of the tautological line bundle $P^*\sigma$, we arrive at an effective theory whose excitations are localized at the zero set of $\phi$. We view this as a theory in one dimension lower, which is called a domain wall theory. Note that the bordism class of the zero set of $\phi$ is precisely the image of $[M]$ under $\sm_\sigma$ in \eqref{eq:this-one}.
    The domain wall theory no longer has the symmetry of the bulk theory: instead, it is a fermionic theory with a unitary internal symmetry $U$ squaring to $1$; i.e. it has $\Spin\times \Z/2$ tangential structure.
    We direct the reader to \cite[\S III.A]{PhysSmith} and \cite[\S 3.1]{HKT19} for more details.
\end{example}

In this context, anomaly matching refers to the process of identifying pairs of preimages and images of anomaly classes under the defect anomaly map $\mathrm{Def}_W$. 
To perform anomaly matching, we must understand not only the two groups classifying the possible bulk and defect theories, but also the kernel and cokernel of the map $\mathrm{Def}_W$.
In some cases, one may deduce that information from an understanding of explicit bordism generators, but in general this approach is difficult.
To address this question, we derived the map of spectra and identified its fiber, forming the Smith fiber sequence of \cref{math_Smithfibersequence}. Now, just as for bordism, we may form a long exact sequence.

\begin{corollary}
\label{IZ_LES_cor}
Applying $I_\Z$ to the cofiber sequence~\eqref{the_cof_seq}, we obtain the following long exact sequence of Anderson-dualized bordism groups, or in light of \cref{IFT_classification}, groups of invertible field theories:
\begin{equation}
\label{math_SBLES_2}
    \dotsb \longrightarrow I_\Z\Omega^{k-r}_\xi(X^{V+W-r}) \xrightarrow{\mathrm{Def}_W}
        I_\Z\Omega^k_\xi(X^V) \xrightarrow{\mathrm{Res}_W}
        I_\Z\Omega^k_\xi(S_X(W)^V) \xrightarrow{\mathrm{Ind}_W}
        I_\Z\Omega^{k-r+1}_\xi(X^{V+W-r}) \longrightarrow\dotsb
\end{equation}
\end{corollary}
This long exact sequence is our mathematical model for the symmetry-breaking long exact sequence (SBLES) of \cite{PhysSmith} (see \cref{tab:my-table}).
In addition to the defect anomaly map $\mathrm{Def}_W$ defined in \cref{def:defect-anomaly-map}, we call $\mathrm{Res}_W$ the \emph{residual anomaly map} and $\mathrm{Ind}_W$ the \emph{index anomaly map}.

\begin{table}
    \centering
    \begin{tabular}{c c}
    \toprule
       Smith fiber sequences  & SBLES \\
       \midrule
       \cref{spinc_spin_Smith}  &  \cite[\S IV.A]{PhysSmith} \\ 
        \cref{SO_Z2_exm} & \cite[\S IV.B]{PhysSmith} \\ 
        \cref{spin_4periodic} & \cite[\S IV.C]{PhysSmith}\\
        \cref{Zk_spin} & \cite[\S IV.D \& E]{PhysSmith}\\
        \cref{quater_exm,spinh_exm} & \cite[\S IV.F]{PhysSmith}\\
    \bottomrule
    \end{tabular}
    \caption{A cross-list of Smith fiber sequences with the corresponding symmetry breaking long exact sequences (SBLES) that we study in \cite[\S IV]{PhysSmith}.}
    \label{tab:my-table}
\end{table}
\begin{figure}[h]
\centering
    \begin{center}
    \begin{tikzcd}[column sep=normal, row sep=small, ampersand replacement=\&]
	\& {I_\Z\Omega^{k-1}_{\Spin\times\Z/2}} \& {I_\Z\Omega_{\Pinp}^k} \& {I_\Z\Omega^k_\Spin} \\
	{-1} \& 0 \& 0 \& {\Z} \\
	0 \& \Z \& {\Z/2} \& {0} \\
	1 \& 0 \& 0 \& {\Z/2} \\
	2 \& {(\Z/2)^2} \& {\Z/2} \& {\Z/2} \\
	3 \& {(\Z/2)^2} \& {\Z/2} \& {\Z} \\
	4 \& {\Z\oplus \Z/8} \& {\Z/16} \& {0} \\
        \arrow[from=1-2, to=1-3, "\mathrm{Def}_\sigma"]
        \arrow[from=1-3, to=1-4, "\mathrm{Res}_\sigma"]
	\arrow[from=2-4, to=3-2, out=-30, in=150]
	\arrow[from=3-2, to=3-3]
	\arrow[from=4-4, to=5-2, out=-30, in=150]
	\arrow[from=5-2, to=5-3]
	\arrow[from=5-4, to=6-2, out=-30, in=150]
	\arrow[from=6-2, to=6-3]
	\arrow[from=6-4, to=7-2, out=-30, in=150]
	\arrow[from=7-2, to=7-3]
\end{tikzcd}
\end{center}
\caption{Long exact sequence of field theories associated to \cref{pinp_spin_z2}. Observe that all maps in low degrees are determined by exactness. This long exact sequence also appears in \cite[\S IV.C]{PhysSmith}}
\label{fig:pinp_spin_z2_SBLES}
\end{figure}
\begin{example}
    In degree $k=4$, there is a map of invertible field theories 
    \begin{equation}\label{eq:that-one}
    I_\Z\Omega^3_{\Spin\times\Z/2} \cong \Z\oplus \Z/8 \xrightarrow{\mathrm{Def}_\sigma} \Z/16 \cong I_\Z\Omega^4_\Pinp.
    \end{equation}
    The $\Z/16$ classifies anomalies of Majorana fermions in $2+1$ dimensions.
    An associated domain wall theory often has $1+1$d chiral fermion modes. To answer the question of what particular chiral fermions can live on the domain wall, we need to identify the map $\mathrm{Def}_\sigma$ in \eqref{eq:that-one}. To do so, we turn to the long exact sequence of invertible field theories, which we draw out in \cref{fig:pinp_spin_z2_SBLES}.

By exactness, we deduce that the defect matching map in degree $k=4$ sends $(a,b) \in \Z \oplus \Z/8$ to $-a+2b \in \Z/16$. This matches with the physical computation in \cite[\S 3.1]{HKT19}, and its physical significance is discussed further there and in \cite[\S III.C.3]{PhysSmith}.
\end{example}

As a computational tool, the long exact sequence allows us to determine the defect matching maps with ease. Moreover, the other two maps in the SBLES \eqref{math_SBLES_2} also have physical interpretations.
The residual anomaly map classifies the obstruction to gapping a QFT after symmetry breaking, as we explain in \cite[\S III.B]{PhysSmith}, while the index anomaly map generalizes the relationship between Berry phases and the ground-state degeneracy in $0+1$D systems; see \cite[\S III.C]{PhysSmith}.

\appendix

\section{The Long Exact Sequence in Bordism}\label{appendix_bordism_LES}

In this appendix, we explicitly describe the Smith long exact sequence of bordism groups and work through an example. As in \cref{math_Smithfibersequence}, let $(X,\xi)$ be a stable tangential structure, let $V$ be a virtual bundle over $X$, and let $W$ be a real vector bundle over $X$ of rank $r$.
The corresponding long exact sequence of bordism groups of \cref{IZ_LES_cor} is
\begin{equation}
\begin{tikzcd}
    \dots \to \Omega^\xi_k(S_X(W)^{p^*V}) \ar[r,"p"] & \Omega^\xi_k(X^{V}) \ar[r,"\sm_W"] & \Omega^\xi_{k-r}(X^{V + W-r}) \ar[r,"\delta"]
    & \Omega^\xi_{k-1}(S_X(W)^{p^*V}) \to \cdots
\end{tikzcd}
\end{equation}
Here, $p\colon S_X(W)\to X$ is the projection, $\sm_W$ is the Smith homomorphism, and $\delta$ is the connecting map. In this section, we will be explicit about taking the pullback $p^*V$ of $V$ to $S_X(W)$.

Starting from the left, $\Omega_k^\xi(S_X(W)^{p^*V})$ is the bordism group of $k$-manifolds $M$ equipped with a map $f\colon M\to S_X(W)$ together with a $\xi$-structure on $TM\oplus f^*(p^*V)$.
Next, $\Omega^\xi_k(X^{V})$ is the bordism group of $k$-manifolds equipped with a map to $X$ with the analogous twisted $\xi$-structure, and $\Omega_{k-r}^\xi(X^{V+W-r})$ is the bordism group of $(k-r)$-manifolds $N$ equipped with a map $g$ to $X$ with a $\xi$-structure on $TM\oplus g^*V\oplus g^*W$. Note that the Smith homomorphism lowers the dimension by $r$ and twists the tangential structure condition by $W$.

Now we describe each map at the level of manifolds.
  \begin{enumerate}
        \item $p$:
        Let $M$ be a closed $k$-manifold equipped with a map $h\colon M \to S_X(W)$ such that $TM \oplus h^*V$ has a $\xi$-structure, so that $M$ represents a bordism class in $\Omega^\xi_k(S_X(W)^{p^*V})$. The image of $M$ under $p$ is represented by the same manifold $M$ with an $(X, V)$-twisted $\xi$-structure given by the composition with the projection. That is, equip $M$ with the map $M \overset{h}{\longrightarrow} S_X(W) \overset{p}{\longrightarrow} X$.
        
        \item $\sm_W$:  
        Now let $M$ be a closed $k$-manifold equipped with a map $f\colon M \to X$ such that $TM \oplus f^*V$ has a $\xi$-structure. Let $s\colon M\to W$ be a generic section, which is transverse to the zero section $s_0$. Then, the intersection
        $N \coloneqq s(M) \pitchfork s_0(M)$ is a $(k-r)$-dimensional manifold. Let $g$ be the composite $g\colon N \hookrightarrow M \overset{f}{\to} X$.
        Since the normal bundle $\nu$ to $N$ satisfies $\nu \cong f^*W|_N = g^* W$, $TM|_N \cong TN \oplus \nu \cong TN \oplus g^*W$, and hence $N$ carries an $(X, V + W)$-twisted $\xi$-structure coming from the $(X, V)$-twisted $\xi$-structure on $M$. We have $\sm_W\colon M\mapsto N$.
        
        \item $\delta$:
        This is the connecting map in the long exact sequence.
        Start with a closed $k-r$-manifold $N$ with an $(X, V+W)$-twisted $\xi$-structure given by, as above, $g\colon N \to X$ and a $\xi$-structure on $TN \oplus g^*V\oplus g^*W$. Consider the sphere bundle $S_N(g^*W)$ of $W$ restricted to $N$: it has a map to $S_X(W)$ given by inclusion.
        
        We claim that $S_N(g^*W)$ is the image under $\delta$ of $N$, but     
        it remains to show that $S_N(g^*W)$ has the appropriate tangential structure.
        This will be a corollary of a general splitting result of tangent bundles of sphere bundles.
\begin{lemma}
\label{sphere_bundle_splitting}
For any vector bundle $\pi\colon V\to B$, there is an isomorphism of vector bundles, canonical up to a contractible space of choices,
\begin{equation}
    TS(V)\oplus\underline\R \xrightarrow{\cong} \pi^*(TB) \oplus \pi^*(V).
\end{equation}
\end{lemma}
\begin{proof}
Choose a metric and connection on $V$; both of these are contractible choices. For any fiber bundle $\pi\colon E\to B$ of smooth manifolds, the choice of connection splits $TE$ as a direct sum of the horizontal subbundle, which is isomorphic to $\pi^*(TB)$, and the vertical tangent bundle $T_vE = \ker(\pi_*)$, which when pulled back to a fiber is the tangent bundle of that fiber.

Let $\nu$ be the normal bundle of $S(V)\hookrightarrow V$. Then there is a canonical isomorphism $T_vS(V)\oplus\nu\cong \pi^*(V)$, which is a parametrized version of the standard isomorphism $TS^n\oplus\nu_{S^n\hookrightarrow\R^{n+1}} \cong \underline\R^{n+1}$. Combining this with the previous paragraph,
\begin{equation}
    TS(V) \oplus\nu \cong \pi^*(TB) \oplus T_vS(V)\oplus\nu\cong \pi^*(TB)\oplus \pi^*(V),
\end{equation}
and the fiberwise outward unit normal vector field trivializes $\nu$.
\end{proof}
        If we now analyze the vertical and horizontal pieces of the tangent bundle to $S_N(g^*W)$,
        we find that $T(S_N(g^*W)) \oplus \underline{\R} \cong p^*TN \oplus p^*g^*W$. Then, we can pull back the relationship describing the tangential structure of $N$ to see that $p^*TN\oplus p^*g^*W \oplus p^*g^*V$ over $S_N(g^*W)$ has a $\xi$-structure. So, $T(S_N(g^*W)) \oplus \underline{\R} \oplus p^*g^*V$ has a $\xi$-structure, and thus $S(g^*W)$ has a $(S_X(W), p^*V)$-twisted $\xi$-structure.
    \end{enumerate}

\label{explicit_pin}

Let us now go through the long exact sequence of bordism groups for the Smith map \ref{pinm_pinp}. In this case, the Smith homomorphism is a map
\begin{equation}
\sm_{2\sigma}\colon \Omega_k^{\Pinm} \longrightarrow \Omega_{k-2}^{\Pinp}
\end{equation}
between the bordism group of $k$-dimensional \pinm manifolds to the bordism group of $(k-2)$-dimensional \pinp manifolds, described by sending a \pinm manifold $M$ to any closed submanifold $N$ whose homology class is Poincaré dual to $w_1(M)^2$. Alternatively, in view of \cref{Smith_homomorphism_intersection_defn}, we could define $\sm_{2\sigma}$ by choosing a section $s$ of the pullback of $2\sigma$ to $M$ transverse to the zero section, then letting $N$ be the zero locus of $s$.
Recall from \cref{twists_of_spin} that a \pinm structure is a trivialization of $w_1(M)^2+w_2(M)$, while a \pinp structure on $M$ is equivalent to a trivialization of $w_2(M)$. Equivalently, a \pinm manifold $M$ admits a spin structure on $TM\oplus \Det(M)$, while a \pinp manifold $M$ admits a spin structure on $TM\oplus 3\Det(M)$. These conditions mean that if $N$ is Poincaré dual to $w_1(M)^2$ inside a \pinm manifold $M$, then $N$ acquires a \pinp structure.

The third set of groups in this long exact sequence corresponds to the homotopy groups of the fiber, $MT\Spin\wedge \Sigma^{-1}\R P^2$.
By Pontrjagin-Thom, these are the groups $\widetilde\Omega_{*+1}^{\Spin}(\RP^2)$: bordism groups of spin manifolds $X$ equipped with maps $f\colon X\to \RP^2$, modulo the subgroup for which $f$ is null-homotopic. Equivalently, we may consider the twisted bordism groups $\Omega^\Spin_{*}(\RP^1,\sigma)$. Elements of this group are represented by manifolds $N$ with maps $f\colon N\to \RP^1$ such that $TN\oplus f^*\sigma$ is spin.

We next describe the other two maps that appear alongside $\sm_{2\sigma}$ in the bordism long exact sequence and provide several lemmas that help us understand the geometry.
\begin{definition}
Define a map $p\colon \Omega_*^\Spin(\RP^1, \sigma)\to\Omega_*^{\Pin^-}$ by sending $(N, f\colon N\to\RP^1)$ to $N$.
\end{definition}
\begin{lemma}
If $N$ has an $(\RP^1, \sigma)$-twisted spin structure, then $N$ has a canonical \pinm structure (so the map $p$ lands in \pinm bordism as claimed).
\end{lemma}
\begin{proof}
The orientation of $TN\oplus f^*\sigma$  is equivalent data to an isomorphism $\Det(TN)\xrightarrow{\cong} f^*\sigma$, so we obtain a spin structure on $TN\oplus\Det(TN)$, i.e.\ a \pinm structure.
\end{proof}

The third map in the long exact sequence is the connecting map $\delta\colon\Omega_*^{\Pin^+}\to\Omega_{*+1}^{\Spin}(\RP^1, \sigma)$. The map $\delta$ sends a \pinp manifold $M$ to the total space of the sphere bundle $S(2\Det(TM))$.\footnote{Note that $S(2\Det(TM))\simeq S(g^*(2\sigma))$.} The key to understanding $\delta$ is showing that $S(2\Det(TM))$ has a $(\RP^1, \sigma)$-twisted spin structure; in particular, we must cook up a map to $\RP^1$.

\begin{definition}
\label{varphi_M}
Given a \pinp manifold $M$, choose a metric on $\Det(TM)$ (a contractible choice); then, given $x\in M$ and $p,q\in\sigma_x$ with $\sqrt{\lvert p^2\rvert + \lvert q^2\rvert} = 1$, so that $(x, p, q)\in S(2\Det(TM))$, the two sections of $\pi^* (2\Det(TM))$
\begin{equation}
\begin{aligned}
    (x, p, q) &\mapsto (p, q)\\
    (x, p, q) &\mapsto (-q, p)
\end{aligned}
\end{equation}
are everywhere linearly independent, so $\pi^*(2\Det(TM))$ is canonically trivial. This allows us to define a map $\varphi_M\colon S(2\Det(TM))\to\RP^1$: given $(x, p, q)\in S(2\Det(TM))$ as above, $(p,q)\in (\pi^*(2\Det(TM)))_{(x, p, q)}$, which is canonically identified with $\R^2$, send $(p,q)$ to its image $[p:q]\in\RP^1$ (using that $p$ and $q$ are never both $0$).

\end{definition}
\begin{definition}
Let $\delta\colon \Omega_*^{\Pin^+}\to \Omega_{*+1}^\Spin(\RP^1,\sigma)$ be the map sending $M\mapsto (S(2\Det(TM)), \varphi_M)$, where $\varphi_M$ is defined above in \cref{varphi_M}.
\end{definition}
If $\sigma\to\RP^1$ is the Möbius bundle, then $\varphi_M^*(\sigma) = \pi^*(\Det(TM))$.
\begin{lemma}
\label{last_pinp_morph}
If $M$ is \pinp, $(S(2\Det(TM)), \varphi_M)$ has a canonical $(\RP^1, \sigma)$-twisted spin structure, up to a contractible space of choices, so that $\delta$ lands in $\Omega_{*+1}^{\Spin}(\RP^1,\sigma)$ as claimed.
\end{lemma}
\begin{proof}
Plugging in $V = 2\Det(TM)$ to \cref{sphere_bundle_splitting}, we learn
\begin{subequations}
\begin{equation}
    TS(2\Det(TM))\oplus\underline\R\cong \pi^*(TM)\oplus 2\pi^*(\Det(TM)).
\end{equation}
Since $\varphi^*_M(\sigma) \cong \pi^*(\Det(TM))$,
\begin{equation}
\label{finally_twspin}
    TS(2\Det(TM))\oplus \varphi_M^*(\sigma)\oplus \underline\R\cong \pi^*(TM)\oplus 3\pi^*(\Det(TM)).
\end{equation}
Since $M$ is \pinp, the right-hand-side of~\eqref{finally_twspin} is spin, so the left-hand side is too; by two-out-of-three, this means $TS(2\Det(TM))\oplus \varphi_M^*(\sigma)$ is also spin.
\end{subequations}
\end{proof}

The maps $\sm_{2\sigma}$, $p$, and $\delta$ assemble into a long exact sequence in bordism, as we will draw in \cref{fig:Pinm_Pinp_bordism_LES}. To write out this long exact sequence, we need to know the relevant bordism groups in low dimensions. Giambalvo~\cite[\S 2, \S 3]{Gia73} computes $\Omega_k^{\Pin^+}$ for $k\le 12$, more than good enough for us, and gives generating manifolds in all degrees we need except $k = 2,3$ (though see~\cite{KT90pinp} for a correction); the rest were given by Kirby-Taylor~\cite[Proposition 3.9, Theorem 5.1]{KT90}. Anderson-Brown-Peterson~\cite[Theorem 5.1]{ABP69} computed \pinm bordism groups, with generating manifolds again described by Giambalvo~\cite[Theorem 3.4]{Gia73} and Kirby-Taylor~\cite[Theorem 2.1]{KT90}. However,
the twisted spin bordism of $\RP^1$ is less well-documented, so we calculate it here, using another Smith homomorphism.
\begin{lemma}
\label{tw_RP1}
There is an abelian group $A$ of order $4$ such that
\begin{equation}
\Omega_k^\Spin(\RP^1, \sigma) \cong
\begin{cases}
    \Z/2, &k = 0,1,3,4\\
    A, &k = 2\\
    0, &k = 5.
\end{cases}
\end{equation}
\end{lemma}
\begin{proof}
We may start the computation of $\Omega^\Spin(\RP^1,\sigma)$ using the observation of Kirby and Taylor \cite{KT90} that the degree two map
\begin{equation}\label{spintimes2}
\mathbb{S} \overset{\cdot 2}{\longrightarrow} \mathbb{S} \longrightarrow \Sigma_+^{\infty-1} \RP^2
\end{equation}
of \cref{Wall_exm} induces multiplication by two on spin bordism.
Taking the spin bordism long exact sequence of \ref{spintimes2} and inputting the spin bordism of a point, we may deduce the groups $\Omega^\Spin(\RP^1,\sigma)$ in low dimensions, up to one ambiguity, as indicated in \cref{fig:spintimes2LES}.
\end{proof}

\begin{figure}[h!]
    \centering
\begin{tikzcd}
	{*} & {\Omega_*^\Spin} & {\Omega^\Spin_*} & {\Omega^\Spin_*(\RP^1,\sigma)} \\
	5 & 0 & 0 & 0 \\
	4 & \Z & \Z & {\Z/2} \\
	3 & 0 & 0 & {\Z/2} \\
	2 & {\Z/2} & {\Z/2} & A \\
	1 & {\Z/2} & {\Z/2} & {\Z/2} \\
	0 & \Z & \Z & {\Z/2}
	\arrow[from=3-2, to=3-3]
	\arrow[from=3-3, to=3-4]
	\arrow[from=4-4, to=5-2, in=150, out=-30]
	\arrow[from=5-2, to=5-3]
	\arrow[from=5-3, to=5-4]
	\arrow[from=5-4, to=6-2, in=150, out=-30]
	\arrow[from=6-2, to=6-3]
	\arrow[from=6-3, to=6-4]
	\arrow[from=6-4, to=7-2, in=150, out=-30]
	\arrow[from=7-2, to=7-3]
	\arrow[from=7-3, to=7-4]
\end{tikzcd}
\caption{Long exact sequence in spin bordism partially determining $\Omega_*^\Spin(\RP^1,\sigma)$}
\label{fig:spintimes2LES}
\end{figure}

\begin{remark}
To address the question as to whether $A$ is isomorphic to $\Z/4$ or $\Z/2\oplus\Z/2$,
one could appeal to geometric arguments or an Adams spectral sequence calculation, but it turns out that the Smith long exact sequence that we will study in \cref{fig:Pinm_Pinp_bordism_LES} provides a cleaner argument that $A\cong\Z/4$.
\end{remark}

\begin{figure}[h!]
    \centering
\begin{tikzcd}
	{*} & {\Omega_{*}^{\Spin}(\RP^1,\sigma)} & {\Omega_*^{\Pinm}} & {\Omega_{*-2}^{\Pinp}} \\
	6 & 0 & {\Z/16} & {\Z/16} \\
	5 & 0 & 0 & {\Z/2} \\
	4 & {\Z/2} & 0 & {\Z/2} \\
	3 & {\Z/2} & 0 & 0 \\
	2 & {\Z/4} & {\Z/8} & {\Z/2} \\
	1 & {\Z/2} & {\Z/2} & 0 \\
	0 & {\Z/2} & {\Z/2} & 0
	\arrow["\ref{6dim_pin}", from=2-3, to=2-4]
	\arrow["\ref{5dim_pin}"', from=3-4, to=4-2, in=150, out=-30]
	\arrow["\ref{4dim_pin}"', from=4-4, to=5-2, in=150, out=-30]
	\arrow["\ref{2dim_pin_1}", from=6-2, to=6-3]
	\arrow["\ref{2dim_pin_2}", from=6-3, to=6-4]
	\arrow["\ref{1dim_pin}", from=7-2, to=7-3]
	\arrow["\ref{0dim_pin}", from=8-2, to=8-3]
\end{tikzcd}
    \caption{Bordism Long Exact Sequence for $\Pinm \rightsquigarrow \Pinp$}
    \label{fig:Pinm_Pinp_bordism_LES}
\end{figure}

We will provide some explicit descriptions of the interesting maps in this sequence using knowledge of the generators of each bordism group, which for \pinp and \pinm may be found in \cite{KT90}. For the twisted spin bordism of $\RP^1$, we use what we learned in \cref{tw_RP1}.

\begin{enumerate}[label=\textrm{(\alph*)}]
\item\label{0dim_pin} $*=0$: The group $\Omega_0^\Spin(\RP^1,\sigma)\cong \Z/2$ is generated by the class of the point equipped with the inclusion $i$ into $\RP^1$. The condition  of $T\pt\oplus i^* \sigma$ being spin is satisfied since $i^*\sigma$ is trivial. The map $f$ forgets $i$, so sends this generator to the point with its \pinm structure, which is a generator of $\Omega_0^\Pinm\cong \Z/2$.

\item\label{1dim_pin} $*=1$: Consider the circle with spin structure induced from its Lie group framing, denoted $S_{\mathit{nb}}^1$, equipped with the degree two map $\phi\colon S^1\to S^1\simeq \RP^1$. If $x\in H^1(\RP^1;\Z/2)$ is the generator, we have
\begin{equation}
w(TS^1\oplus \phi^*\sigma) = w(TS^1)\phi^*w(\sigma) = (1)(1+2\phi^*(x)) = 1,
\end{equation}
so $(S^1_{\mathit{nb}}, \phi)$ has an $(\RP^1, \sigma)$-twisted spin structure. The map $p$ forgets $\phi$, so sends the bordism class of $(S^1_{\mathit{nb}}, \phi)$ to $S_{\mathit{nb}}^1$, which generates $\Omega_1^\Pinm\cong \Z/2$~\cite[Theorem 2.1]{KT90}.

\item\label{2dim_pin_1} $* = 2$ (part 1): Exactness of the Smith long exact sequence at $\Omega_2^\Spin(\RP^1, \sigma)\cong A$ implies that $A$ maps injectively to $\Omega_2^{\Pin^-}\cong\Z/8$, so $A\cong\Z/4$, and we have resolved the extension problem from \cref{tw_RP1}.

The Klein bottle $K$ is an $S^1$-bundle over $\RP^1$, with the monodromy of the fiber $S^1$ around the base given by reflection. Therefore $K = S(\sigma\oplus\underline\R)$ as $S^1$-bundles over $\RP^1$. Let $\pi\colon K\to \RP^1$ be the bundle map; then \cref{sphere_bundle_splitting} defines an isomorphism $TK\oplus\underline\R\cong \pi^*(\sigma)\oplus\underline\R^2$ (using the Lie group trivialization of $T\RP^1$). The Möbius bundle $\sigma$ represents the nonzero class in $[\RP^1, B\O] = \pi_1(B\O)\cong\Z/2$, so $2\sigma$ is trivializable,\footnote{To make this argument carefully, one must know that addition in $[S^1, B\O]$ corresponds to direct sum of vector bundles. A priori this is not true---addition in $[S^1, X]$ is built from the pinch map $S^1\to S^1\vee S^1$. That this coincides with the group structure on $[S^1, B\O]$ arising from direct sum of virtual vector bundles depends on the Eckmann-Hilton argument.} and in particular spin, meaning that $(K, \pi)$ admits an $(\RP^1, \sigma)$-twisted spin structure (in fact, it admits $4$).

That $(K, \pi)$ generates $\Omega_2^\Spin(\RP^1, \sigma)$ depends on which of the four $(\RP^1, \sigma)$-twisted spin structures one chooses. Specifically, each $(\RP^1, \sigma)$-twisted spin structure restricts to a spin structure on the fiber $S^1$, and we need this to be the spin structure on $S^1$ induced by the Lie group framing. Two of the four $(\RP^1, \sigma)$-twisted spin structures satisfy this. To then see that either of these two Klein bottles generates, one can play with the Smith long exact sequence from \cref{Wall_exm}
\begin{equation}
\label{cof2_spin}
    \dotsb\longrightarrow
    \Omega_k^\Spin\overset{\cdot 2}{\longrightarrow}
    \Omega_k^\Spin \longrightarrow
    \Omega_k^\Spin(\RP^1, \sigma)\xrightarrow{\sm_\sigma}
    \Omega_{k-1}^\Spin\longrightarrow \dotsb
\end{equation}
to see that $\sm_\sigma\colon \Omega_2^\Spin(\RP^1,\sigma)\to\Omega_1^\Spin$ is the unique surjective map $\Z/4\to\Z/2$; the Poincaré dual to $w_1(\sigma)$ is represented by the fiber $S^1$ in $K$, which we chose to have the Lie group spin structure, so $\sm_{\sigma}(K, \pi) = S_{\mathit{nb}}^1$, which generates $\Omega_1^\Spin$, implying $(K, \pi)$ generates $\Omega_2^\Spin(\RP^1, \sigma)$.

Now take $f(K, \pi)$, which amounts to forgetting $\pi$ and finding the \pinm bordism class of $K$. The Arf-Brown-Kervaire invariant is a complete invariant $\Omega_2^{\Pin^-}\xrightarrow{\cong}\Z/8$~\cite{Bro71, KT90}, so it suffices to compute this invariant on $K$, as has been explicitly worked out in~\cite[\S II.D]{Tur20}. Our choice of the nonbounding spin structure on the fiber implies that the Arf-Brown-Kervaire map $\Omega_2^{\Pin^-}\xrightarrow{\cong}\Z/8$ sends $[K]\mapsto\pm 2$, so $f\colon\Z/4\to\Z/8$ sends $1\mapsto 2$, as required by exactness.

\item\label{2dim_pin_2} $*=2$ (part 2): There are two \pinm structures on $\RP^2$, and both are generators of $\Omega_2^{\Pin^-}\cong\Z/8$~\cite[\S 3]{KT90}. Pick either of these \pinm structures;
the class $w_2(\sigma)\in H^2(\RP^2;\Z/2)\cong \Z/2$ is a generator, and the Smith homomorphism $\Omega_2^\Pinm\to \Omega_0^\Pinp$ maps the input $\RP^2$ to the Poincaré dual of $w_2(2\sigma)$. The class $\mathit{PD}(w_2(2\sigma))$ is $1\in H_0(\RP^2;\Z/2)\cong \Z/2$ and is represented by a single \pinp point. The class of the point also corresponds to the zero-dimensional intersection of the zero section and a generic section of
$2\sigma$.
\item\label{4dim_pin} $*=4 \to 3$: $\Omega_2^{\Pinp}\cong \Z/2$ is generated by the Klein bottle $K$, where as before we need the nonbounding spin structure on the $S^1$ fiber of $K$. The connecting map $\delta$ sends $K$ to $S(2\Det(K))$; we saw above in part~\ref{2dim_pin_1} that $\Det(K)\cong\sigma$ and $2\sigma$ is trivialized over $K$, so $S(2\Det(K))\cong S^1\times K$.

Tracking the (twisted) spin structures through this argument, one sees that we obtain the nonbounding spin structure on $S^1$, so $g(K) = [S^1_{\mathit{nb}}\times K]\in\Omega_3^{\Spin}(\RP^1, \sigma)\cong\Z/2$, and $[S^1_{\mathit{nb}}\times K]$ is indeed the generator.\footnote{Another way to see this is that because the connecting morphism in the Smith long exact sequence is obtained from a map of spectra by taking homotopy groups, the connecting morphism commutes with the $\pi_*(\mathbb S)$-actions on $\Omega_*^{\Pin^+}$ and $\Omega_*^\Spin(\RP^1, \sigma)$. The Pontrjagin-Thom theorem identifies this $\pi_*(\mathbb S)$-action on bordism groups with taking products with stably framed manifolds; focusing specifically on the nonzero element of $\pi_1(\mathbb S)$, which is represented by the bordism class of $S_{\mathit{nb}}^1$. Thus, since $\times S_{\mathit{nb}}^1\colon \Omega_2^{\Pin^+}\to\Omega_3^{\Pin^+}$ is an isomorphism~\cite[\S 5]{KT90} and the Smith maps $\Omega_{k-2}^{\Pin^+}\to\Omega_k^\Spin(\RP^1, \sigma)$ are isomorphisms for $k = 3,4$ as we saw in the long exact sequence, then $\times S_{\mathit{nb}}^1\colon \Omega_3^\Spin(\RP^1, \sigma)\to\Omega_4^\Spin(\RP^1, \sigma)$ is also an isomorphism.}
\item\label{5dim_pin} $*=5 \to 4$: $\Omega_3^{\Pin^+}\cong\Z/2$ is generated by $S^1_{\mathrm{nb}}\times K$~\cite[\S 5]{KT90},
and $\Omega_4^\Spin(\RP^1, \sigma)\cong\Z/2$ is generated by $S^1_{\mathit{nb}}\times S^1_{\mathit{nb}}\times K$, with the map to $\RP^1$ induced from the fiber bundle $K\to\RP^1$ from part~\ref{2dim_pin_1}.\footnote{Another choice of generator is the K3 surface with trivial map to $\RP^1$, as follows from~\eqref{cof2_spin}. The complicated topology of the K3 surface makes this generator harder to work with explicitly.} Thus the story is the same as in~\ref{4dim_pin}, crossed with $S_{\mathit{nb}}^1$.
\item\label{6dim_pin} $*=6$: The group $\Omega_6^{\Pinm}$ is generated by $\RP^6$ with either of its two \pinp structures, while $\Omega_4^{\Pinp}$ is generated by $\RP^4$ with either of its two \pinm structures. Since the normal bundle to $\RP^4$ inside $\RP^6$ is indeed the restriction of $2\sigma$, $\RP^4$ represents the Poincaré dual homology class to $e(2\sigma)$ and is the image of the Smith homomorphism applied to $\RP^6$.
\end{enumerate}

\section{Computation of a Cobordism Euler Class}

\label{s:eu_counter}
To compute the Smith homomorphism, it often suffices to take a Poincaré dual of the $\Z$- or $\Z/2$-cohomology Euler class of
a vector bundle $V\to X$, for example in~\cite{KTTW15, COSY19, HKT19}. 
However, our general definition (\cref{smith_homomorphism_euler_definition}) of the Smith homomorphism instead uses the
(possibly twisted) $\xi$-cobordism Euler class, tying more closely with invertible field theories.
In this appendix, we
walk through a 
concrete, low-dimensional example of a cobordism Euler class that goes beyond the cohomological approximation. 

Recall that a spin$^h$ structure is a $(B\SO(3), V_3)$-twisted spin structure, where $V_3\to B\SO(3)$ is the
tautological vector bundle. Then, as we discussed in~\eqref{spinh_to}, there is a Smith homomorphism
\begin{equation}
	\sm_V\colon \Omega_k^{\Spin^h}\to \Omega_{k-3}^\Spin(B\SO(3)).
\end{equation}
\begin{theorem}
\label{euler_exm_main}
Give $S^4$ the spin$^h$ structure whose $\SO_3$-bundle is classified by either map $S^4\to B\SO(3)$ whose homotopy
class generates $\pi_4(B\SO(3))\cong\Z$.
\begin{enumerate}
	\item Exactness of the Smith long exact sequence forces $\sm_V(S^4)$ to be the bordism class of $\Snb^1$ with
constant map to $B\SO(3)$ in $\Omega_1^\Spin(B\SO(3))$.
	\item $e(V)\in H^3(S^4;\Z) = 0$, and there is no way to assign every smooth representative of the Poincaré dual
	of $e(V)$ a spin structure whose bordism class equals that of $\Snb^1$.
	\item The spin cobordism Euler class of $V$ is nonzero, and all smooth representatives of its Poincaré dual
	have the spin bordism class of $\Snb^1$ and a constant map to $B\SO(3)$.
\end{enumerate}
\end{theorem}
This is why the cohomology Euler class does not suffice in this example. 

We work with $\xi = \Spin$ and its twists throughout this appendix; see \cref{euler_other} for other tangential
structures. Let $\ko$ denote the connective real $K$-theory spectrum; work of Anderson-Brown-Peterson~\cite{ABP67}
shows that the Atiyah-Bott-Shapiro map $\MTSpin\to\ko$~\cite{ABS64} is $7$-connected, meaning that as long as we
restrict to manifolds of dimension $7$ and below, we may replace twisted spin bordism with twisted $\ko$-homology;
in particular, we will work with $\ko$-cohomology Euler classes.

Another consequence of the Atiyah-Bott-Shapiro map is that vector bundles with spin structure are oriented for
$\ko$-cohomology, meaning that if $V\to X$ is a spin vector bundle, the Euler class $e^\ko(V)$ that a priori lives
in $\ko^r(X^{V-r})$ in fact can be passed by the Thom isomorphism to $e^\ko(V)\in\ko^r(X)$.

We will compute these Euler classes by using a formula for the Thom class in equivariant $\KO$-theory. Recall that the \term{Clifford algebra} $\Cl_n$ is the free $\Z/2$-graded $\R$-algebra on $n$ odd generators $e_1,\dotsc,e_n$ that anticommute and each square to $1$; likewise $\Cl_{-n}$ is the free $\Z/2$-graded $\R$-algebra on $n$ anticommuting odd elements $e_1,\dotsc,e_n$, each squaring to $-1$. See Atiyah-Bott-Shapiro~\cite{ABS64} for more on Clifford algebras.

When discussing modules, tensor products, Morita equivalence, etc.\ for Clifford algebras, we will always consider the $\Z/2$-grading and take the Koszul sign rule into account. We will use $\hat\otimes$ to denote this $\Z/2$-graded tensor product. 

Let $G$ be a compact Lie group and $(X, A)$ a pair of $G$-spaces such that $X/A$ is compact. Then, following Karoubi~\cite{karoubi_algebres_1968, karoubi_equivariant_2002}, we will represent classes in $\KO_G^n(X)$ as (equivalence classes of) $G$-equivariant $\Cl_n$-module bundles over $X$ equipped with a $G$-invariant trivialization over $A$.
\begin{theorem}[{Atiyah~\cite[Theorem 6.1]{atiyah_bott_1968}}]
\label{atiyah_KOG}
Let $G$ be a compact Lie group, $X$ be a compact $G$-space, and $V\to X$ be an $8k$-dimensional spin $G$-equivariant vector bundle. Let $\mathcal S\to X$ denote the associated spinor bundle. Let $u\in\KO_G^0(V, V\setminus 0)$ be the class obtained from the $G$-equivariant Clifford module bundle associated to $\mathcal S$ by the method of Atiyah-Bott-Shapiro~\cite[\S 11]{ABS64}. Then the $G$-equivariant Thom class of $V$ is $U_{G,V}\coloneqq w^{-k}u\in\KO_G^{8k}(V, V\setminus 0)$, where $w\in\KO_G^{-8}(\pt)$ is the Bott class.
\end{theorem}
Recall the exceptional isomorphism $\Spin(3)\cong\SSp(1)$, and recall that $\ko^*\cong\Z[\eta, v, w]/(2\eta, \eta^3,
2v, 4w-v^2)$ with $\abs\eta = -1$, $\abs v = -4$, and $\abs w = -8$.\footnote{The negative grading is a feature of
generalized cohomology: for any spectrum $E$, $E^k(\mathrm{pt}) = E_{-k}(\mathrm{pt}) = \pi_{-k}(E)$.}

The following result is stated without proof by Davis-Mahowald~\cite[\S 2]{DM79}; see Bruner-Greenlees~\cite[Theorem 5.3.1]{BG10} for a proof.
\begin{proposition}
\label{ko_sp1}
There is an isomorphism of $\ko^*$-modules $\ko^*(B\SSp(1))\cong\ko^*[[p_1^\H]]$ with $\abs{p_1^\H} = 4$.
\end{proposition}
The class $p_1^\H$ is called the \term{first symplectic $\ko$-Pontrjagin class}. The specific isomorphism in
\cref{ko_sp1} can be fixed uniquely by requiring that the image of $p_1^\H$ under $\ko\to H\Z$ is the usual first
symplectic Pontrjagin class, which is positive on the tautological quaternionic line bundle over $\HP^1$.

Given a rank-$3$ spin vector bundle $V\to X$, let $\mathcal S_V\to X$ be the associated spinor bundle, which is the
quaternionic line bundle associated to the accidental isomorphism $\Spin(3)\cong \SSp(1)$. We will show how the following result is a corollary of \cref{atiyah_KOG}.
\begin{corollary}
\label{rank_3_ko_Euler}
Let $V\to X$ be a rank-$3$ vector bundle with spin structure. Then $e^\Z(V)\in H^3(X;\Z)$ and $e^{\Z/2}(V)\in
H^3(X;\Z/2)$ both vanish, and
\begin{equation}
	e^\ko(V) = \eta p_1^\H(\mathcal S_V)\in\ko^3(X).
\end{equation}
\end{corollary}
We thank an anonymous referee for help with the following proof.
\begin{proof}[Proof sketch of \cref{rank_3_ko_Euler}]
We first make two reductions.
\begin{enumerate}
    \item The connective covering map $\ko^3(B\SU_2)\to \KO^3(B\SU_2)$ is an isomorphism: since $B\SU_2$ is $2$-connected, all maps $B\SU_2\to \Sigma^3\KO$ factor through the $2$-connected cover of $\Sigma^3\KO$, which is $\Sigma^3\ko$. The connective cover map $\ko\to\KO$ sends Euler classes to Euler classes; therefore it suffices to show $e^\KO(V) = \eta p_1^\H(\mathcal S)$ (where here we abuse notation and let $p_1^\H$ denote its image in $\KO$-cohomology).
    \item We can reduce further to a question in the $\SU_2$-equivariant $\KO$-theory of a point. Specifically, the Atiyah-Segal completion theorem~\cite[Corollary 2.2]{atiyah_equivariant_1969} (see also \cite{atiyah_vector_1961}) identifies the ``Borelification'' map
    \begin{equation}
    \label{borelify}
        \KO_{\SU_2}^*(\pt) \longrightarrow \KO^*(B\SU_2)
    \end{equation}
    as completion at a certain ideal, so if we can produce preimages of $e^\KO(V)$ and $\eta p_1^\H(\mathcal S)$ under~\eqref{borelify}, then show that they are equal, this will suffice to finish the proof.
\end{enumerate}
Given a compact Lie group $G$ and a $d$-dimensonal real spin $G$-representation $W$, let $e_G^\KO(W)\in\KO_G^d(\pt)$ denote the \term{$G$-equivariant $\KO$-Euler class}, defined to be the pullback of the equivariant $\KO$-theory Thom class $U_{G,W}\in\KO_G^d(W, W\setminus 0)$ under the zero section. Atiyah~\cite[\S 6]{atiyah_bott_1968} shows that the Atiyah-Segal completion map sends $U_{G,W}$ to the nonequivariant $\KO$-theory Thom class of the associated bundle $W\to BG$,\footnote{We thank Yigal Kamel for help with this point.} so the analogous fact is true for Euler classes: the image of $e_G^\KO(W)$ under~\eqref{borelify} is $e^\KO(W)\in \KO^d(BG)$.

Let $\mathbf 2$ denote the defining representation of $\SU_2$ and $[\mathbf 2]$ denote its class in $\mathit{RSp}(\SU_2) \cong\mathit{KSp}_{\SU_2}^0(\pt)\cong \KO_{\SU_2}^4(\pt)$.\footnote{The notation $\mathbf 2$ means that this representation is two-dimensional over $\C$; we work over $\R$, so think of $\mathbf 2$ as four-dimensional.} Bruner-Greenlees~\cite[Theorems 5.3.1 and 5.3.5]{BG10} show that $2-[\mathbf 2]$ is a preimage of $p_1^\H(\mathcal S)$ under the completion map. To summarize, it now suffices to show that $e_{\SU_2}^\KO(V) = \eta(2-[\mathbf 2])$ in $\KO^3_{\SU_2}(\pt)$. 

To proceed, we reduce this to a purely algebraic question. As mentioned above, we may represent elements of $\KO_G^n(\pt)$ as equivalence classes of finite-dimensional real vector spaces with commuting $G$-representation and $\Cl_n$-module structures, and we may represent elements of $\KO_G^n(W, W\setminus 0)$ as $G$-equivariant $\Cl_n$-module bundles over $W$, equipped with $G$-invariant trivializations on $W\setminus 0$. Atiyah-Bott-Shapiro~\cite[\S 5]{ABS64} show that $\eta\in\KO^{-1}(\pt)$ is represented by $\Cl_{-1}$ as a module over itself; giving this module the trivial $\SU_2$-action produces a representative of $\eta\in \KO_{\SU_2}^{-1}(\pt)$. Thus, $\boldsymbol 2\mathbin{\hat\otimes} \Cl_{-1}$ (with $\boldsymbol 2$ regarded as purely even) is a $\Cl_4\mathbin{\hat\otimes}\Cl_{-1}$-module. For $1\le m\le n$, there is a Morita equivalence
\begin{subequations}
\begin{equation}
\label{morita13}
    \cat{Mod}_{\Cl_n\mathbin{\hat\otimes}\Cl_{-m}}\overset\simeq\longrightarrow
    \cat{Mod}_{\Cl_{n-m}}
\end{equation}
implemented by tensoring with the $(\Cl_n\mathbin{\hat\otimes}\Cl_{-m}, \Cl_{n-1})$-bimodule
\begin{equation}
    M(n, m)\coloneqq \Cl_{n-m}\mathbin{\hat\otimes} \underbracket{\R^{1\mid 1}\mathbin{\hat\otimes}\dotsm\mathbin{\hat\otimes}\R^{1\mid 1}}_{\text{$m$ copies}},
\end{equation}
\end{subequations}
with bimodule data given by:
\begin{itemize}
    \item $e_1\mathbin{\hat\otimes} 1,\dotsc,e_{n-m}\mathbin{\hat\otimes} 1\in\Cl_n\mathbin{\hat\otimes}\Cl_{-m}$ act by left multiplication by $e_1,\dotsc,e_{n-m}$ on $\Cl_{n-m}$, tensored with the identity on the $\R^{1\mid 1}$ factors.
    \item $\Cl_{n-m}$ acts by right multiplication on $\Cl_{n-m}$ tensored with the identity on the $\R^{1\mid 1}$ factors.
    \item $e_{n-m+i}\mathbin{\hat\otimes} 1\in\Cl_n\mathbin{\hat\otimes}\Cl_{-m}$ acts on the left by the identity on $\Cl_{n-1}$ tensored with $\begin{psmallmatrix}0&1\\1&0\end{psmallmatrix}$ on the $i^{\mathrm{th}}$ factor of $\R^{1\mid 1}$, and the identity on the remaining factors.
    \item $1\mathbin{\hat\otimes} e_i\in\Cl_n\mathbin{\hat\otimes} \Cl_{-m}$ acts on the left by the identity on $\Cl_{n-m}$ tensored with $\begin{psmallmatrix}\phantom{-}0&1\\-1&0\end{psmallmatrix}$ on the $i^{\mathrm{th}}$ factor of $\R^{1\mid 1}$, and the identity on the remaining factors.
\end{itemize}
This Morita equivalence follows from the standard isomorphism of $\Z/2$-graded algebras $\Cl_1\mathbin{\hat\otimes}\Cl_{-1}\cong\mathrm{End}(\R^{1\mid 1})$ (see, for example, \cite[Lemma 6.17]{DG18}).

Specializing to $n = 4$ and $m = 1$, and giving this bimodule the trivial $\SU_2$-action, we see that the class $[\boldsymbol 2]\eta\in\KO_{\SU_2}^3(\pt)$ is represented by the $\SU_2$-equivariant $\Cl_3$-module
\begin{equation}
    (\boldsymbol 2\mathbin{\hat\otimes}_\R \Cl_{-1}) \mathbin{\hat\otimes}_{\Cl_4\mathbin{\hat\otimes}\Cl_{-1}} M(4, 1).
\end{equation}
For the Euler class, we use Atiyah's \cref{atiyah_KOG} identifying the $G$-equivariant Thom class of an $8k$-dimensional spin vector bundle. Apply this to $V\oplus\R^5$, an $8$-dimensional $\SU_2$-representation, thought of as an $\SU_2$-equivariant vector bundle over $\pt$. Since $B\SU_2$ is $3$-connected, all $\SU_2$-representations $\rho$ are spin, as $w_2(\rho)\in H^2(B\SU_2;\Z/2) = 0$. Therefore we learn from \cref{atiyah_KOG} that the Thom class $U\coloneqq U_{\SU_2,V\oplus\R^5}$ of $V\oplus\R^5$ is, as an $\SU_2$-equivariant $\Cl_8$-bundle, constructed by Atiyah-Bott-Shapiro's method~\cite[\S 11]{ABS64} from the spinor representation $\mathcal S$ of $\Spin_8$ restricted by the standard inclusion $\SU_2 \cong\Spin_3\hookrightarrow\Spin_8$. Thus we have an $\SU_2$-equivariant $\Cl_8$-module bundle
\begin{equation}
    U\in \KO_{\SU_2}^8(V\oplus\R^5, (V\oplus\R^5)\setminus 0)\cong \KO_{\SU_2}^8(\Sigma^5(V, V\setminus 0)).
\end{equation}
The suspension isomorphism identifies this group with $\KO_{\SU_2}^3(V, V\setminus 0)$; to see this in terms of Clifford modules, use the trivial $\R^5$ summand of $V\oplus\R^5$ to define a $\Cl_{-5}$-action on $U$ in a similar manner to~\cite[\S 4.1]{BE23}. The fiber $U_0$ over $0\in V$ is therefore a $(\Cl_8\mathbin{\hat\otimes}\Cl_{-5})$-module with commuting $\SU_2$-action. Now use the Morita equivalence~\eqref{morita13} with $n = 8$ and $m = 5$, and give the bimodule $M(8, 5)$ implementing this Morita equivalence the trivial $\SU_2$-action. Then $U_0\mathbin{\hat\otimes}_{\Cl_8\mathbin{\hat\otimes}\Cl_{-5}} M(8,5)$ is an $\SU_2$-equivariant $\Cl_3$-module, hence defines a class in $\KO_{\SU_2}^3(\pt)$.

What remains is a tedious but straightforward computation to identify the classes of these two $\SU_2$-equivariant $\Cl_3$-modules.
\end{proof}

\begin{proof}[Proof of \cref{euler_exm_main}]
Recall from~\eqref{spinh_to} that the Smith homomorphism $\sm_V\colon \Omega_k^{\Spin^h}\to\Omega_{k-3}^\Spin(B\SO(3))$ belongs to a long exact sequence whose third term is \spinc bordism:
\begin{equation}
\label{explicit_spinh_smith}
    \dotsb \to\Omega_2^\Spin(B\SO(3))\to
    \Omega_4^{\Spin^c}\to
    \Omega_4^{\Spin^h} \overset{\sm_V}{\to}
    \Omega_1^\Spin(B\SO(3))\to
    \Omega_3^{\Spin^c}\to\dots
\end{equation}
From Stong~\cite[Chapter XI]{Sto68} we know $\Omega_3^{\Spin^c} =0$ and $\Omega_4^{\Spin^c}\cong\Z^2$, from Freed-Hopkins~\cite[Theorem 9.97]{FH16} we know $\Omega_4^{\Spin^h}\cong\Z^2$, and from Wan-Wang~\cite[\S 5.5.3]{WW19} we know $\Omega_1^\Spin(B\SO(3))\cong\Z/2$ and $\Omega_2^\Spin(B\SO(3))$ is torsion. Plugging this into~\eqref{explicit_spinh_smith}, we see that $\sm_V$ is surjective.

Wan-Wang's argument implies that the map $\Omega_1^\Spin\to \Omega_1^\Spin(B\SO(3))$ choosing the trivial $\SO(3)$-bundle is an isomorphism, so the generator of $\Omega_1^\Spin(B\SO(3))$ is any nonbounding spin $1$-manifold with trivial $\SO(3)$-bundle. Hu~\cite[Appendix A]{Hu23} shows that $\CP^2$ and $S^4$ generate $\Omega_4^{\Spin^h}$, where $\CP^2$ has spin$^h$ structure induced from its \spinc structure via the standard inclusion $\U(1)\cong\SO(2)\to\SO(3)$, and $S^4$ has spin$^h$ structure whose principal $\SO(3)$-bundle $V\to S^4$ is induced from the tautological quaternionic line bundle on $\HP^1\cong S^4$: this has an associated $\SSp(1)$-bundle, and we quotient by $\{\pm 1\}$ to get an $\SO(3)$-bundle. In particular, $\CP^2$ is in the image of $\Omega_4^{\Spin^c}\to\Omega_4^{\Spin^h}$, so because $\sm_V$ is surjective, $\sm_V(S^4, V)$ must be $\Snb^1$ with trivial map to $B\SO(3)$, proving the first part of the theorem.

Because $H^3(S^4;\Z)$ and $H^3(S^4;\Z/2)$ both vanish, the $\Z$ and $\Z/2$ cohomology Euler classes of $V$ are zero. Therefore any null-homologous $1$-manifold in $S^4$ (i.e.\ any closed, oriented $1$-manifold mapping to $S^4$) is a smooth representative of the Poincaré dual of $e(V)$. Most of these manifolds, such as the standard $S^1\subset S^4$, can be given a nonbounding spin structure, but the empty submanifold cannot, even though it is Poincaré dual to $e(V)$. This proves the second part of the theorem.

As discussed above, the Atiyah-Bott-Shapiro map is $7$-connected, and therefore for discussing degree-$3$ spin cobordism of $S^4$, we may use $\ko$-cohomology without losing information. The Atiyah-Hirzebruch spectral sequence quickly implies
\begin{equation}
    \ko^*(S^4)\cong \ko^*[z]/(z^2),\ \abs{z} = 4.
\end{equation}
In particular, $\ko^3(S^4)\cong\Z/2$, generated by $\eta z$.

The spinor bundle of $V$ is the quaternionic line bundle associated to the identification $\Spin(3)\cong\SSp(1)$. Since $V$ came from the identification $S^4\cong\HP^1$, the spinor bundle of $V$ is the tautological quaternionic line bundle $L_\H\to\HP^1$. This is classified by the inclusion $j\colon \HP^1\to\HP^\infty\simeq B\SSp(1)$ as the $4$-skeleton; considering the map of Atiyah-Hirzebruch spectral sequences for $\ko$-cohomology induced by $j$ shows that $p_1^\H\in\ko^4(B\SSp(1))$ pulls back by $j$ to $z\in\ko^4(S^4)$. Thus by \cref{rank_3_ko_Euler}, $e^\ko(V) = \eta z\ne 0$ in $\ko^3(S^4)$.

Because $\ko^3(S^4)$ has only one nonzero element, the Poincaré dual of the nonzero element must be the unique nonzero element $x$ of $\ko_1(S^4)\cong\Z/2$. Pulling back to spin bordism, the same argument we made for $B\SO(3)$ shows that the smooth representatives of $x$ are precisely the nonbounding spin $1$-manifolds with null-bordant map to $S^4$---and composing with the map $S^4\to B\SO(3)$ classifying $V$, we have shown that every smooth representative of the Poincaré dual of $e^\ko(V)$ (hence also the spin cobordism Euler class) represents the image of $(S^4, V)$ under the Smith homomorphism.
\end{proof}

In the remainder of this appendix, we give another proof of \cref{rank_3_ko_Euler}. This second proof uses more homotopy theory, but no equivariant methods; as such, we think the two proofs complement each other.
\begin{lemma}[{Greenlees-May~\cite[\S 15]{GM95}}]
\label{CP_ko}
There is a $\ko$-module equivalence
\begin{equation}\label{ko_BT}
	\ko\wedge (B\U(1))_+ \simeq \ko\vee \bigvee_{n\ge 0} \Sigma^{4n+2}\ku.
\end{equation}
\end{lemma}
\begin{lemma}
\label{HP_ko}
There is a $\ko$-module equivalence
\begin{equation}\label{ko_BSp}
	\ko\wedge (B\SSp(1))_+ \simeq \bigvee_{n\ge 0} \Sigma^{4n}\ko.
\end{equation}
\end{lemma}
\begin{proof}
Recall that in~\eqref{symplectic_splitting}, we showed using a variant of the proof of \cref{U_smith_split} that there's an $\MTSpin$-module equivalence
\begin{equation}
    \MTSpin\wedge (B\SSp(1))_+ \simeq \bigvee_{n\ge 0}\Sigma^{4n}\MTSpin.
\end{equation}
The result immediately follows by applying $\bl\wedge_{\MTSpin}\ko$ to both sides, where the algebra map $\MTSpin\to\ko$ is the Atiyah-Bott-Shapiro orientation~\cite{ABS64, Joa04}.
\end{proof}
\begin{remark}
Analogues of \cref{CP_ko,HP_ko} for the periodic theory $\KO$ and its generalizations to the \term{higher real $K$-theories} $\mathit{EO}_\Gamma$ are known: see Bousfield~\cite{Bou90}, Meier~\cite[Theorem 2.8]{Mei17}, Chatham~\cite[Theorems 5.13 and 5.14]{Cha20}, Bhattacharya-Chatham~\cite[Main Theorem 1.7]{BC22}, and Chatham-Hu-Opie~\cite[Example 2.10]{CHO24}.
\end{remark}
\begin{definition}
Recall the complexification map $c\colon\ko\to\ku$. The cofiber of $c$ is a map $R\colon \ku\to\Sigma^2\ko$,
denoted \term{realification}.
\end{definition}
As $\ku\not\simeq \ko\vee \Sigma^2\ko$, the third map in the cofiber sequence begun by $c$ and $R$ must be
nontrivial in
\begin{equation}
	\pi_0\mathrm{Map}_\ko(\Sigma\ko, \ko)\cong \mathrm{Map}_{\mathbb S}(\Sigma\mathbb S, \ko)\cong
	\pi_1\ko\cong\Z/2,
\end{equation}
so must be the unique nontrivial class, namely the Hopf map $\eta\colon\Sigma\ko\to\ko$. That is, we have found the
\term{Wood cofiber sequence}
\begin{equation}
\label{wood_seq}
	\ko\overset c\longrightarrow \ku\overset R\longrightarrow \Sigma^2\ko\overset\eta\longrightarrow
	\Sigma\ko\longrightarrow \dotsb
\end{equation}
which we identified as a Smith cofiber sequence in \cref{Wall_exm}.

Recall from \cref{GMTW_exm} that the unit sphere bundle inside the tautological rank-$3$ vector bundle $V_3\to
B\Spin(3)$ is homotopy equivalent to the map $B\Spin(2)\to B\Spin(3)$, which can be identified via accidental
isomorphisms to the map $B\U(1)\to B\SSp(1)$ given by the inclusion of a maximal torus. Choose for concreteness the
standard maximal torus, given by the map $\U(1)\to\SU(2)\cong\SSp(1)$ defined by
\begin{equation}
	i\colon z\mapsto \begin{bmatrix}z & 0\\0 & z^{-1}\end{bmatrix}.
\end{equation}
Thus there is a Smith cofiber sequence
\begin{equation}
\label{SGCof}
	\ko\wedge (B\U(1))_+\overset{i_*}{\longrightarrow} \ko\wedge (B\SSp(1))_+\overset{\frown e^\ko(V)}{\longrightarrow} \ko\wedge
	\Sigma^3 (B\SSp(1))^{V_3-3},
\end{equation}
which is the cofiber sequence in \cref{jamesQP} smashed with $\ko$.\footnote{In \cite[\S IV.F]{PhysSmith}, we computed the Anderson dual long exact sequence in low degrees.} This sequence is also studied, and placed in context, in \cref{smith_su2_vector}.

Since $V_3\to B\SSp(1)$ is spin, the Thom isomorphism identifies the third term in this sequence with
$\Sigma^3\ko\wedge (B\SSp(1))_+$.
\begin{proposition}
\label{which_map_is_it}
The identifications in \cref{CP_ko,HP_ko} may be chosen to produce the following identifications of $\ko$-module
homomorphisms.
\begin{subequations}
\begin{enumerate}
	\item The map $i_*\colon \ko\wedge (B\U(1))_+\to \ko\wedge (B\SSp(1))_+$ is the direct sum of the maps
	\begin{equation}
		\Sigma^{4n+2}R\colon\Sigma^{4n+2}\ku\longrightarrow\Sigma^{4n}\ko,
	\end{equation}
	together with the identity $\ko\to\ko$ on the basepoint.
	\item\label{ko_what_we_prove} The fiber of $i_*$, which is a map $y\colon \ko\wedge \Sigma^2(B\SSp(1))_+\to
	\ko\wedge (B\U(1))_+$, is the direct sum of the maps
	\begin{equation}
		\Sigma^{4n+2}c\colon\Sigma^{4n+2}\ko\longrightarrow \Sigma^{4n+2}\ku.
	\end{equation}
	\item The map $\frown e^\ko(V)\colon \ko\wedge (B\SSp(1))_+\to \Sigma^3 \ko\wedge (B\SSp(1))_+$  is the direct
	sum of the maps
	\begin{equation}
		\Sigma^{4n-1}\eta\colon\Sigma^{4n}\ko\longrightarrow \Sigma^{4n-1}\ko,
	\end{equation}
	together with the zero map on the copy of $\ko$ in degree $0$.
\end{enumerate}
\end{subequations}
\end{proposition}
\begin{proof}
Using the Wood cofiber sequence~\eqref{wood_seq}, any one of these three results implies the other two; we will
prove~\eqref{ko_what_we_prove}.

Restricted to $\Sigma^{4k+2}\ko$, $y$ is a map
\begin{equation}
\label{restr_ku}
	y|_{\Sigma^{4k+2}\ko}\colon \Sigma^{4k+2}\ko\longrightarrow \ko\vee \bigvee_{\ell\ge 0} \Sigma^{4\ell+2}\ku.
\end{equation}
We will show that it is possible to choose the equivalences in~\cref{CP_ko,HP_ko} to make $y$ ``diagonal,'' i.e.\ after
composing to the projection onto each summand of~\eqref{restr_ku} \emph{except $\Sigma^{4k+2}\ku$},
$y|_{\Sigma^{4k+2}\ko}$ is trivial. We know that the ``diagonal terms,'' i.e.\ the maps obtained by restricting $y$
to $\Sigma^{4k+2}\ko$ and then projecting to the $\Sigma^{4k+2}\ku$ summand in the codomain, must be $\pm c$,
because this is the only choice compatible with base change along $\ko\to H\Z$ inducing maps on $\Z$ cohomology
which are isomorphisms in those degrees: this is because
\begin{equation}
\label{ko_Sp_maps}
	\pi_0\mathrm{Map}_\ko(\Sigma^{4k+2}\ko, \Sigma^{4k+2}\ku)\cong \pi_0\mathrm{Map}_{\mathbb
	S}(\mathbb S, \ku)\cong \pi_0\ku\cong\Z
\end{equation}
and $c$ is a generator; thus we must obtain either $c$ or $-c$ on the equal-degree summand.

The map out of $\Sigma^{4k+2}\ko$ is trivial when projected to the $\ko$ in degree $0$, because we need that
$\Sigma^0\ko$ summand to map to the degree-$0$ $\ko$ summand in the cofiber $\ko\wedge (B\SSp(1))_+$, because that
map arose from a basepoint-preserving map of spaces. In the rest of the proof, we will address the
$\Sigma^{4\ell+2}\ku$ summands.

A map of $\ko$-modules $\Sigma^m\ko\to\Sigma^n\ku$ is equivalent data to a map of spectra $\Sigma^m\mathbb
S\to\Sigma^n\ku$, which is classified by $\pi_n(\ku)$. Since $\ku$ is connective, all ``off-diagonal terms'' vanish
unless $4k+2\ge 4\ell+2$; therefore for our $\Sigma^{4k+2}\ko$ summand we may restrict to the map
\begin{equation}
\label{become_matrix}
	y\colon \Sigma^2\ko\vee\dotsb\vee \Sigma^{4k+2}\ko \longrightarrow \Sigma^2\ku\vee\dotsb\vee \Sigma^{4k+2}\ku.
\end{equation}
We may therefore describe $y$ as a $(k+1)\times (k+1)$ matrix. Connectivity of $\ku$ implies this matrix is upper
triangular.

We saw in~\eqref{ko_Sp_maps} that if $m \ge \ell$, then $\pi_0\mathrm{Map}_\ko(\Sigma^{4m+2}\ko, \Sigma^{4\ell+2}\ku)\cong\pi_{2(m-\ell)}\ku\cong\Z$; tracing through the identifications there, we learn that this $\Z$ of maps is the set of
scalar multiples of the map $b^{2(m-\ell)}c$, where $b\colon\Sigma^2\ku\to\ku$ is the connective version of the
Bott periodicity map.
Therefore there are integers
$\lambda_{ij}$ for $1\le i < j\le k+1$ such that the map~\eqref{become_matrix} is given by the following upper
triangular matrix:
\begin{equation}
\begin{bmatrix}
\pm c & \lambda_{12}b^2c & \lambda_{13}b^4c & \dotsb & \lambda_{1(k+1)} b^{2k}c\\
& \pm c & \lambda_{23}b^2c & \dotsb & \lambda_{2(k+1)}b^{2k-2}c\\
& & \ddots & \ddots & \vdots\\
& & & \pm c & \lambda_{k(k+1)} b^2c\\
& & & & \pm c
\end{bmatrix}.
\end{equation}
This matrix can clearly be row-reduced over $\ko_*$ to $c\cdot \mathrm{Id}$, and the requisite row operations correspond
to automorphisms of $\ko\vee\dotsb\vee\Sigma^{4k+2}\ko$. The row operations are compatible with adding on more
summands by increasing $k$, so we may conclude.
\end{proof}
\begin{lemma}[{Bruner-Greenlees (see~\cite[Theorem 3.8]{BPR23})}]
\label{cap_pont}
Recall $\ko^*(B\SSp(1))\cong\ko^*[[p_1^\H]]$ from \cref{ko_sp1}.
There is an isomorphism $\varphi\colon \ko_*(B\SSp(1))\xrightarrow{\cong}\ko_*[x]$,
where $\abs x = 4$,\footnote{Here we use
polynomial notation only for conciseness; we have not defined any ring structure on $\ko_*(B\SSp(1))$.}
such that
the $\ko^*(B\SSp(1))$-module structure on $\ko_*(B\SSp(1))$ is the one uniquely specified by
\begin{equation}
	p_1^\H \frown x^k = x^{k-1}.
\end{equation}
\end{lemma}
Finally, we can calculate the $\ko$-Euler class!
\begin{proof}[Proof of \cref{rank_3_ko_Euler}]
It suffices to work universally with the tautological bundle $V_3\to B\Spin(3)$; the spinor bundle is the
tautological quaternionic line bundle associated to $\Spin(3)\cong\SSp(1)$, and so $p_1^\H(\mathcal S_{V_3})$ is the
class we called $p_1^\H\in\ko^4(B\SSp(1))$ in \cref{ko_sp1}.

By \cref{which_map_is_it},
\begin{equation}
\label{cap_recipe}
	e^\ko(V_3)\frown x^k = \eta x^{k-1},
\end{equation}
where we define $x^{-1} = 0$ for convenience.\footnote{As we have not been careful about explicit choices of
isomorphisms, there could be a sign factor in the choice of $x^k$, but since $2\eta = 0$, the possible sign error
goes away.} A general element of $\ko^3(B\SSp(1))$ is of the form
\begin{equation}
\label{ko3_power_series}
	\sum_{k\ge 0} \eta (p_1^\H)^k w^{k-1}.
\end{equation}
We know how $\eta$ and $w^{k-1}$ act on $\ko_*(B\SSp(1))$ because the $\ko$-theory cap product is linear over $\ko^*$.
We know how $p_1^\H$ acts on $\ko_*(B\SSp(1))$ thanks to \cref{cap_pont}. Using these, we can see that the only
class of the form~\eqref{ko3_power_series} whose cap product matches that of $e^\ko(V_3)$ in~\eqref{cap_recipe} is
$\eta p_1^\H$.

Finally, we have to check that $e^\Z(V_3)$ and $e^{\Z/2}(V_3)$ both vanish. $B\SSp(1)$ is $3$-connected, so
$H^3(B\SSp(1);\Z)$ and $H^3(B\SSp(1);\Z/2)$ both vanish.
\end{proof}
\begin{remark}[Euler classes of low-rank spin vector bundles]
For $2\le n \le 6$, $\Spin(n)$ participates in an accidental isomorphism with another Lie group, and one can run similar arguments to
compute $\ko$-Euler classes of other low-rank vector bundles as corollaries of Atiyah's \cref{atiyah_KOG}.
\begin{enumerate}
	\item If $L$ is a real line bundle with spin structure, $e^\ko(L) = 0$, because $e^\ko$ pulls back from the
	twisted Euler class over $B\SO(1) = *$; see, e.g., Crabb~\cite[Corollary 3.37(i)]{Cra91}.
	\item If $V_2$ has rank $2$, one can use the accidental isomorphism $\Spin(2)\cong\U(1)$ and the fact that the map $c\colon \ko^*(B\U(1))\to \ku^*(B\U(1))$ is injective~\cite[\S 5.2]{BG10} to show that $e^{\ko}(V_2)$ is determined by $e^\ku(V)$, hence also by $e^K(V)$, the image in periodic $K$-theory. In particular, $V_2$ acquires the structure of a complex line bundle, and there is a formula for the $K$-theory Euler classes of complex vector bundles, e.g.\ in Bott~\cite[(7.2)]{Bot69}.
	\setcounter{enumi}{3}
	\item If $V_4\to X$ has rank $4$, its spinor bundle factors as $\mathcal S = \mathcal S^+\oplus \mathcal S^-$,
	where the two factors $\mathcal S^\pm$ are quaternionic line bundles associated to the two factors of
	$\phi\colon \Spin(4)\xrightarrow{\cong}\SSp(1)\times\SSp(1)$. There is a choice of $\phi$ such that
	\begin{equation}
		e^\ko(V_4) = p_1^\H(\mathcal S^+) - p_1^\H(\mathcal S^-)\in\ko^4(X).
	\end{equation}
	\item There is an accidental isomorphism $\Spin(5)\cong\SSp(2)$, and $\ko^*(B\SSp(2))\cong\ko^*[[p_1^\H,
	p_2^\H]]$ with $\abs{p_1^\H} = 4$ and $\abs{p_2^\H} = 8$ (see~\cite[\S 2]{DM79} or~\cite[Theorem 5.3.5]{BG10}). Therefore
	$\ko^5(B\SSp(2))\cong 0$, so for any rank-$5$ spin vector bundle $V_5$, $e^\ko(V_5) = 0$; see, e.g., Crabb~\cite[Corollary 3.37(i)]{Cra91}.
\end{enumerate}
\end{remark}
\begin{remark}
\label{euler_other}
We saw above that for twisted spin bordism, the $\ko$-theoretic Euler class suffices to approximate the cobordism Euler class. For other tangential
structures, one may need more or less information.
\begin{itemize}
	\item Unoriented bordism decomposes as a sum of shifts of mod $2$ homology, and this splitting is compatible
	with the Smith homomorphism. Therefore in this setting, one can use the $\Z/2$-cohomology Euler class.
	\item Wall~\cite{Wal60} showed that $\MTSO$, localized at $2$, splits as a sum of shifts of $H\Z$ and $H\Z/2$.
	Therefore when one studies Smith homomorphisms for twisted oriented bordism, the $\Z$-cohomology Euler class
	will be accurate up to odd-primary torsion. On odd-primary torsion, oriented and spin bordism coincide, so in
	that setting one can use $\ko$-Euler classes for twisted oriented bordism.
	\item Analogously to spin and $\ko$, one can use $\ku$-theory Euler classes for twisted \spinc bordism Smith
	homomorphisms.
\end{itemize}
\end{remark}

\section*{Data Availability and Competing Interests}

This work does not involve any datasets. One
can obtain the relevant materials from the references below.

The authors have no competing interests to declare that are relevant to the contents of this article.

\bibliography{mathbib}
 \bibliographystyle{alpha}

 \end{document}